\documentclass[a4paper,dvipsnames]{article}
\usepackage[utf8]{inputenc}
\usepackage[english]{babel}
\usepackage{amsmath,amsfonts}
\usepackage{bbm}
\usepackage{color}
\usepackage{amsfonts,enumitem}
\usepackage{amssymb}
\usepackage{hyperref} 
\usepackage{graphicx,tikz,todonotes,comment}
\usepackage{stmaryrd}
\hypersetup{hidelinks}

\usepackage{titlesec}

\titleformat{\subsubsection}[runin]
  {\normalfont\normalsize\bfseries}
  {\thesubsubsection}
  {1em}
  {}
  [.]
\titlespacing*{\subsection}
  {0pt}
  {1.8ex plus .3ex minus .2ex}
  {0.5ex}
  
\titlespacing*{\subsubsection}
  {0pt}
  {1.2ex plus .2ex minus .2ex}
  {0.8em}

\usepackage{algorithmic}
\usepackage[ruled,noend]{algorithm2e}
\usepackage[skins]{tcolorbox}

\newcommand{\N}{\mathbb{N}}
\newcommand{\HH}{\mathcal{H}}
\newcommand{\CC}{\mathbb{C}}

\newcommand{\Dist}{\operatorname{Dist}}

\newcommand{\RR}{\mathbb{R}}
\newcommand{\NN}{\mathbb{N}}

\newtheorem{theorem}{Theorem}
\newtheorem{lemma}[theorem]{Lemma}
\newtheorem{proposition}[theorem]{Proposition}
\newtheorem{corollary}[theorem]{Corollary}

\newtheorem{definition}[theorem]{Definition}
\newtheorem{assumption}[theorem]{Assumption}

\newtheorem{remark}[theorem]{Remark}
\newtheorem{lem-Jn-1}[theorem]{lem-Jn-1}

\newenvironment{proof}[1][]{\noindent {\bf Proof #1:\;}}{\hfill $\Box$}

\usepackage[pagewise]{lineno} 
\usepackage{geometry} 
\usepackage{etoolbox}

\apptocmd{\thebibliography}{%
  \small
  \setlength{\itemsep}{1.8pt}
  \setlength{\parsep}{1.5pt}
  \setlength{\parskip}{1.8pt}
}{}{}

\newcommand{\R}{\mathbb{R}}

\newcommand{\vertiii}[1]{{\vert\kern-0.25ex\vert\kern-0.25ex\vert #1
    \vert\kern-0.25ex\vert\kern-0.25ex\vert}}
\makeatletter
\newcommand{\customlabel}[2]{%
   \protected@write \@auxout {}{\string \newlabel {#1}{{#2}{\thepage}{#2}{#1}{}} }%
   \hypertarget{#1}{#2}
}

\makeatother

\everymath{\displaystyle}

\numberwithin{equation}{section}
\numberwithin{theorem}{section}

\title{Quantitative Fredholm backstepping and rapid stabilization}

\author{Ludovick Gagnon\footnote{Université de Lorraine, CNRS, Inria,  F-54000 Nancy, France. E-mail: \texttt{ludovick.gagnon@inria.fr.}}, Amaury Hayat\footnote{CERMICS, \'{E}cole des Ponts ParisTech, 6 - 8, Avenue Blaise Pascal, Cité Descartes—Champs sur Marne, 77455 Marne la Vall\'{e}e, France. E-mail: \texttt{amaury.hayat@enpc.fr.}}, Swann Marx\footnote{LS2N,
Ecole Centrale de Nantes \& CNRS UMR 6004, F-44000 Nantes, France. E-mail:
\texttt{swann.marx@ls2n.fr}
},\\ Shengquan Xiang\footnote{School of Mathematical Sciences, Peking University, 100871, Beijing, China.
E-mail: \texttt{shengquan.xiang@math.pku.edu.cn}}, and Christophe Zhang\footnote{Sorbonne Université, CNRS, Inria équipe CAGE,  F-75005 Paris, France. E-mail: \texttt{Christophe.zhang@inria.fr.}}}
\date{}

\begin{document}

\maketitle

\begin{abstract}
In this paper, we address the existence of Fredholm backstepping transformations for self-adjoint and skew-adjoint operators $A$. Under suitable assumptions on the operator $A$ and the possibly unbounded control operator $B$, we prove the existence of a Fredholm backstepping transformation for operators of order strictly greater than $1$.

This work overcomes two major limitations of the classical Fredholm backstepping framework. One of the main contributions is the explicit identification of the underlying isomorphism used in the construction of the transformation $T$, thereby bypassing the compactness arguments and Riesz basis mechanisms traditionally used in the literature. This explicit structure enables us to derive quantitative and sharp estimates for $\|T\|_{\mathcal{L}(H;H)}$ and $\|T^{-1}\|_{\mathcal{L}(H;H)}$ with respect to the decay rate $\lambda$.
As a consequence, we obtain quantitative rapid stabilization results for a broad class of operators. 
\end{abstract}

\setcounter{tocdepth}{1}

\tableofcontents

\section{Introduction}

Consider the abstract system, 
\begin{equation}
\label{eq:sys}
\begin{cases}
\dfrac{d}{dt} y(t) = A y(t) + Bu(t), \quad t\geq 0\\
y(0)=y_0,
\end{cases}
\end{equation}
with $y_0\in H$. Here, $H$ denotes a complex Hilbert space endowed with the norm $\|.\|_H$, the unbounded operator $A:D(A)\subset H\rightarrow H$ is densely defined and the control operator $B\in \mathcal{L}(U; D(A^*)^\prime)$ is such that
$\textrm{dim} \, U < +\infty$.
We furthermore assume that $A$ generates a strongly continuous semigroup that we denote by $(e^{tA})_{t\geq 0}$.

A classical problem in control theory is to establish the {\it rapid stabilization} of equations of the form \eqref{eq:sys}, that is to say, for every $\lambda>0$, to seek a feedback operator $K:D(K)\subseteq H \rightarrow \R$ such that the control in feedback form $u(t)=Ky(t)$ and the solution $y(t)$ of the closed-loop system,
\begin{equation}
\label{eq:cl-sys}
\begin{cases}
\frac{d}{dt} y(t) = (A+BK)y(t),\quad t\geq 0,\\
y(0) = y_0,
\end{cases}
\end{equation}
 satisfies,
\begin{equation}
\label{eq:expstablestim}
\|y(t)\|_H \leq Ce^{-\lambda t} \|y_0\|_H,
\end{equation}
for $C>0$. This has been studied since, at least, the pioneering work of Slemrod in 1972 (see \cite{slemrod1972linear}). Many techniques exist in the literature to establish the rapid stabilization of \eqref{eq:cl-sys} from controllability properties stemming from \eqref{eq:sys} (these are merely several examples and are far from exhaustive, the Riccati theory \cite{Barbu-book, Breiten-Kunisch-Rodriques,  Lasiecka-Triggiani-book}, Gramian method \cite{Urquiza, Nguyen-sta-cocv}, equivalence between observability and stabilizability \cite{Trelat-Wang-Xu, LWXY}, frequency Lyapunov method \cite{Xiang-heat-2020, Xiang-NS-2020}, frequency-domain criterion \cite{KWY},  various Lyapunov approaches \cite{Coron-Trelat-2004, Hayat-Shang-2021}, among others). 

The backstepping method for PDEs, initially introduced by Balogh and Krsti\'c for boundary control systems \cite{BK1,BK2, vazquez2026backstepping}, is a powerful stabilization method which consists in finding explicit feedbacks using an appropriate Volterra transformation of the control system \cite{coron2017null, coron2021small, Xiang-KdV2019}. 
This backstepping method can be seen as an infinite-dimensional extension of the F-equivalence method introduced by Brunovsky for finite-dimensional systems in \cite{Brunovsky}.
\vspace{2mm}

There were recently several efforts to extend the backstepping method to more general transformations, namely Fredholm transformations \cite{CGM,CoronLu14,CoronLu15,coron:hal-03161523, deutscher2019fredholm,GLM,Gagnon-Hayat-Xiang-Zhang,WaterWave,hayat2024fredholm}.
Although more intricate, the Fredholm transformation possesses the advantage of having a systematic approach, based on controllability and spectral assumptions on \eqref{eq:sys}. 
This allows us to treat PDEs where the existence of a Volterra transformation is difficult to establish, due in particular to the non-trivial boundary condition on the diagonal (see for instance \cite{KrsticSmyshlyaev_Book}). 
These generalizations show that the Fredholm backstepping method (also referred to as an infinite-dimensional $F$-equivalence method) is also a way to relate controllability properties to the rapid stabilization of \eqref{eq:cl-sys}.

However, the Fredholm backstepping approach that has been developed in the last decade has two essential limitations: due to the use of compactness arguments and Riesz basis mechanisms \cite{CGM,CoronLu14, coron:hal-03161523, Gagnon-Hayat-Xiang-Zhang,WaterWave,hayat2024fredholm}, the construction is still {\it implicit} and it does {\it not provide quantitative estimates} of the exponential stabilization with respect to the decay rate $\lambda$. In other words, how does $C$ given in \eqref{eq:expstablestim} behave with $\lambda$?

\subsection{Main results} 
In this paper, we overcome these two limitations of the Fredholm backstepping approach and show that it allows us to answer this  {\it quantitative} question for generic skew-adjoint or self-adjoint systems by providing an {\it explicit} expression of the Fredholm transformation $T$ involved.
We prove the existence of Fredholm backstepping transformations for operators under classical assumptions of controllability and superlinear growth of the eigenvalues. 

In this paper, we work with the following framework.
\begin{assumption}[Functional framework]
\label{assume:structure0}
\hphantom{}\\
\vspace{-\baselineskip}
\begin{enumerate}
    \item (Self-adjointness or skew-adjointness) The operator $A$ has compact resolvent and is either self-adjoint or skew-adjoint.
        \item (Decomposability of $A$) The Hilbert space $H$ can be decomposed as a direct sum of \textbf{finitely} many subspaces:
\begin{equation}
    H= \oplus_{i=1}^m H_i,
\end{equation}
such that for any $i\in \{1, 2,..., m\}$, the operator $A|_{H_i}$ is diagonalizable, with simple eigenvalues. We denote its eigenpairs by 
\begin{equation}
    \{(\lambda_{i, n}, \varphi_{i, n}): \varphi_{i, n}\in H_i, \, A \varphi_{i, n}=  \lambda_{i, n} \varphi_{i, n} , \, \, \forall n\in \mathbb{N}^*\}.
\end{equation}
    \item (Controllability and admissibility) There exists $c, C>0$,
    \begin{equation}
        \label{b_n-growth}
        0<c\leq \langle B, \varphi_{i,n}  \rangle\leq C, \quad \forall n\in \N^*,\; i\in\{1,...,m\}.
    \end{equation}
\item (Localization of eigenvalues) There exists $\alpha>1$ such that, on each subspace $H_i, \, i\in \{1,2, ..., m\}$, the eigenvalues satisfy
\begin{equation}
\label{eq:cond2}
n^{\alpha-1}\lesssim  |\lambda_{i,n+1}- \lambda_{i,n}|\lesssim n^{\alpha-1}, \; \forall n\in \mathbb{N}^*, \ \forall i\in \{1, \cdots, m\}.
\end{equation}
and the gap estimate
\begin{equation}
\label{eq:cond3}
ck^{\alpha-1}|k-n| \leq |\lambda_{i,k}- \lambda_{i,n}|, \quad  \forall k,n \in \mathbb{N}^*.
\end{equation}
\end{enumerate}
\end{assumption}

We refer to Section \ref{sec-func-framework} for a discussion on this framework. Our main results are the following

\begin{theorem}[Rapid stabilization via explicit feedback construction]\label{thm:intro1}
Assume $A:D(A)\subset H \rightarrow H$ and $B\in \mathcal{L}(U;D(A)')$, such that $\dim U < +\infty$, satisfy Assumption \ref{assume:structure0}. 
Then, we can explicitly construct $T\in \mathcal{L}(H;H)$ invertible from $H$ to itself and a feedback $K\in D(A^s)'$ for $s>\tfrac{1}{2\alpha}$ such that 
\begin{equation}\label{eq:introthm1}
\begin{cases}
TA+BK = (A-\lambda I)T, \\  
TB=B, \text{ in } D(A)',
\end{cases}
\end{equation}
holds for $\lambda>0$ non-resonant\footnote{The non-resonant condition is stated in \eqref{def:mathcalN}}. Moreover, for any $\lambda>0$, we can explicitly construct a feedback $K\in D(A^s)'$ such that the solution of the closed-loop system \eqref{eq:cl-sys} satisfies 
\[
\|e^{(A+BK)t}y_0 \|_H \leq C e^{-\lambda t} \|y_0\|_H.
\]
\end{theorem}

A further analysis of this explicit isomorphism allows us to characterize the cost of the backstepping transformation $T=T_\lambda$ with respect to the damping parameter $\lambda>0$. 
\begin{theorem}[Quantitative $e^{c\lambda^{1/\alpha}}$-type  cost]
For any $N\in \N$, there exists $\lambda \in [N,N+1]$ such that the transformation $T_\lambda$ of Theorem \ref{thm:intro1} satisfies 
\[
\|T_\lambda \|_{\mathcal{L}(H;H)}+\|T_\lambda^{-1}\|_{\mathcal{L}(H;H)}\leq Ce^{c\lambda^{1/\alpha}},
\]
where $c,C>0$ do not depend on either $N$ or $\lambda$, and where $\alpha>1$ is the order of the operator $A$.
\end{theorem}

To the best of our knowledge, this is the first quantitative rapid stabilization result obtained through the Fredholm backstepping approach. In particular, unlike previous constructions based on compactness arguments and Riesz basis mechanisms, the present method provides explicit formulas and sharp quantitative estimates on both the transformation and its inverse. Moreover, the results apply to operators $A$ of order $\alpha>1$,  namely they apply to a broad class of self-adjoint and skew-adjoint operators.

Quantitative rapid stabilization estimates are important for several further developments. Beyond their intrinsic interest, such estimates provide a foundation for the study of observer design, robustness with respect to disturbances, stochastic versions of the closed-loop system, and finite time stabilization. More directly, they also lead to constructive small-time null controllability results.

\begin{theorem}[Small-time null-controllability using feedback control]\label{thm:intro3}
Assume that $A$ satisfies Assumption \ref{assume:structure0}. For any $T>0$ and $s\in [0,1-\tfrac{1}{2\alpha})$, there exists a piecewise constant feedback $K(t), \, t\in (0,T)$, such that the closed-loop system 
\[
\begin{cases}
\dfrac{d}{dt} y(t)=(A+BK(t))y(t), \quad t\in (0,T)\\
y(0)=y_0,
\end{cases}
\]
with $y_0\in D(A^s)$, is well-posed in $D(A^s)$ and satisfies 
$y(T)=0$.
\end{theorem}

\subsection{Sketch of the proof and main ingredients}

The strategy  can be summarized in two main steps.

\begin{itemize}
    \item {\it A completely explicit construction of the transformation.} 

    The construction of the transformation operators $(T,K)$ traditionally relies on Riesz basis techniques and compactness arguments; see Section \ref{sec:his:Fredhom} for the historical development of this method. Such approaches make coefficients implicit and quantitative estimates difficult to obtain. 

    In the present work, we observe that these operators can be represented through an {\it infinite-dimensional Cauchy matrix}, without using Riesz basis and compactness. This observation enables us to derive explicit expressions for $(T,K)$ and for $T^{-1}$. For example, from equation \eqref{eq:expr-T-1:full}, the associated operator $\tilde T^{-1}_{i,j}$ is given by
    \begin{equation}
 \tilde T_{i, j}^{-1} =\dfrac{\lambda^2}{(\lambda_j-\lambda_i-\lambda)}  \prod_{\substack{m\in \N^\ast \\ m \neq i}} \left( 1+ \dfrac{\lambda}{\lambda_i-\lambda_m}  \right)  \prod_{\substack{n\in \N^\ast \\ n \neq j}} \left( 1- \dfrac{\lambda}{\lambda_j-\lambda_n}  \right), \quad  \forall i,j\in \N^\ast. \notag
\end{equation}

    See Section \ref{sec:finite-dim} for a finite-dimensional example illustrating the main ideas, and Section \ref{sec:InvertibilityT} for the rigorous construction and well-posedness analysis of the operator equation associated with $(T,K)$ in our general framework.

    \item {\it Sharp quantitative estimates for non-resonant $\lambda$.} 

    Thanks to the newly derived explicit formulas, Section \ref{sec:quantitative} establishes quantitative upper and lower bounds for $\|T\|$ and $\|T^{-1}\|$.  Note that this whole step has not been considered in the previous works. Although the estimates of the other terms are already delicate, the core difficulty lies in the lower bound. Indeed,
    \begin{equation}
 \|T^{-1}\|\leq   \| \check T^{-1}\| \sup_{n} |b_n^{-1}|  \sup_{n} |k_n^{-1}|,
\notag
\end{equation}
where
    \begin{equation}
  k_n b_n = \sum_{j\in \N^*} \tilde{T}^{-1}_{n, j}, \quad \forall n\in \N^*.   \notag
\end{equation}

    Recall that the sequence $\{b_n\}_n$ is prescribed. To estimate $\sup_n |k_n^{-1}|$, we develop a new four-step argument. See Section \ref{sec:42} for an overview of the strategy, and Sections \ref{sec:sec:reaknbn}--\ref{sec:estimDist} for the detailed analysis.
\end{itemize}

In addition, Section \ref{sec-well-posedness} establishes well-posedness of the closed-loop system on a broad class of Banach spaces; this extends the well-posedness results in \cite{Gagnon-Hayat-Xiang-Zhang,WaterWave,hayat2024fredholm}. In Section \ref{sec-null-control}, we prove constructive small-time null controllability using the Fredholm backstepping approach, where we use the iteration introduced in \cite{coron2017null}, see also \cite{coron2021small, Xiang-KdV2019}. This is also the first null controllability result obtained via Fredholm backstepping.

Finally, we summarize several other features of the present approach that distinguish it from previous constructions. 1)  The transformation and stabilization results remain valid on a wider scale of spaces;  2) In the present approach, the existence of a Riesz basis and the compactness issues appear as consequences of the construction rather than as prerequisites and passing the limit arguments for the analysis.

\subsection{History of the Fredholm backstepping method}\label{sec:his:Fredhom}
The Fredholm backstepping method has been successful lately to deduce the rapid stabilization of PDEs from their controllability properties. This method was first introduced by Coron and Lü \cite{CoronLu14, CoronLu15}.  The general idea is to seek a feedback operator $K$ and a transformation $T$ mapping the resulting closed-loop system onto an exponentially stable target system. The target system often has the form\footnote{Actually, more general stable target systems can be considered, but \eqref{eq:target} has the advantage of using the same domain of definition $D(A)$ for both the equation to stabilize and the target system. We refer for instance to \cite{coron:hal-03161523} for a case where the exponential stability comes from a change of boundary conditions, changing the domain of the target system and resulting in a more involved analysis, but leading successfully to the rapid stabilization with the backstepping method.},  
\begin{equation}
\label{eq:target}
\begin{cases}
\dfrac{d}{dt} w(t) = (A-\lambda I) w(t), \quad t\geq 0, \\ 
w(0)=w_0.
\end{cases}
\end{equation}
Thus 
\[
\|w(t)\|_H \leq C_{\lambda} e^{- \lambda t}\| w_0 \|_H,
\]
hence, taking $\lambda>0$ sufficiently large, one can prescribe any exponential decay rate for solution of \eqref{eq:target}. Then, assuming that $T \in \mathcal{L}(H,H)$ is invertible, choosing $w_0=Ty_0$ implies the rapid stabilization of \eqref{eq:cl-sys},
\[
\|y(t)\|\leq \|T^{-1}\| \|w(t)\|_H \leq e^{-\nu t} \|T^{-1}\| \| w_0 \|_H \leq e^{-\nu t} \|T^{-1}\| \|T \| \| y_0 \|_H.
\]
Operators $(T, K)$ thus play a fundamental role.

Let us {\it formally} deduce the equation on $T$ and $K$; see also \cite{WaterWave}. While the rigorous justification of the exponential stability and the functional framework will be specified in Section \ref{sec-func-framework}. First, differentiating $
Ty(t) = w(t)$ in time leads to, 
$$
T(A+BK)y(t) = (A-\lambda I)T y(t).
$$
The previous equation holds for every solution $y(t)$, implying that the operators $T$ and $K$ satisfy:
\begin{equation}
\label{eq:operator-identity}
T(A+BK) = (A-\lambda I)T.
\end{equation}
It is classical to impose the so-called \textit{uniqueness condition}, $TB=B$, 
which reduces \eqref{eq:operator-identity} to, 
\begin{equation}\label{intro:eqop}
\begin{cases}
TA+BK = (A-\lambda I)T,\\  
TB=B.
\end{cases}
\end{equation}

\subsubsection{Abstract approach}

To illustrate how to solve \eqref{intro:eqop} by the generalized backstepping method, we consider a diagonalizable operator $A$ with eigenfunctions $\varphi_n$ and eigenvalues $\lambda_n$. First, apply $(A-\lambda I)T -TA = BK$ on the eigenbasis $\{\varphi_n\}_{n\in \NN^*}$,
\[
(A-(\lambda_n+\lambda) I)T\varphi_n = BK\varphi_n.
\]
Since $k_n:=K\varphi_n$ is a scalar, we can define, assuming $\lambda \notin \sigma(A-\lambda_n I), \forall n\in \N^*$,
\begin{equation}\label{eq:defgn}
g_n :=(A-(\lambda_n+\lambda) I)^{-1}B.
\end{equation}
With this definition, we have $T\varphi_n = k_n g_n$. It then remains to identify the coefficients $k_n$ and to establish the continuity of the transformation $T$ and its invertibility. The coefficients of $k_n$ are obtained through the uniqueness condition. Indeed, 
\[
TB=T\left(\sum_{n\in \N^*} b_n \varphi_n \right) = \sum_{n\in \N^*} b_n k_n g_n, 
\]
and therefore $TB=B$ amounts to,
\begin{equation}\label{intro:TBdecomp}
 \sum_{n\in \N^*} b_n k_n g_n =  \sum_{n\in \N^*} b_n \varphi_n.
\end{equation}
The controllability property of \eqref{eq:sys} allows us to ensure that $b_n$ are bounded from above and below in a norm specific to the space of controllability (see for instance the framework of \cite{CGM}). Together with spectral properties of $A$, the usual approach is to prove that the $\{g_n\}_{n\in \NN^*}$ are a Riesz basis together with compactness arguments, allowing to solve \eqref{intro:TBdecomp} and therefore the uniqueness condition $TB=B$. 

\subsubsection{Riesz basis}\label{sec-riesz}

We briefly recall the definition and some basic properties of Riesz basis, and refer to \cite{christensen2003introduction} for more details.

\begin{definition}\label{def:riesz}
    Let $H$ be a Hilbert space. A family of vectors $\{\xi_n\}_{n\in \mathcal{I}}$, where $\mathcal{I}=\mathbb{Z}$, $\N$, or $\N^*$ is said to be a Riesz basis of $H$, if it is the image of some orthonormal basis by an isomorphism on $H$.
\end{definition}
Riesz bases have the following useful characterization:
\begin{proposition}\label{prop:quad}
    Let $H$ be a Hilbert space. Let $\{\xi_n\}_{n\in \mathcal{I}}$ (where $\mathcal{I}=\mathbb{Z}$, $\N$, or $\N^*$) be a family of vectors, and $\{e_n\}_{n\in \mathcal{I}}$ be an orthonormal basis of $H$. Assume that $\{\xi_n\}$ and $\{e_n\}$ are quadratically close. Then, if $\{\xi_n\}$ is complete, or $\omega$-independent, it is a Riesz basis of $H$.
\end{proposition}

\subsubsection{The quadratically close criterion}
When $|\lambda_n|\simeq n^\alpha$ with $\alpha>3/2$, it can be shown that the family $(g_n)$ is a Riesz basis through the quadratically close criterion of Proposition \ref{prop:quad} under classical assumptions, whether $A$ is self-adjoint or skew-adjoint. Indeed, assuming $|\lambda_n-\lambda_k|>n^{\alpha-1}|n-k|, \, \forall n,k\in \N$ to simplify the computations, the quadratically close criterion amounts roughly to,
\begin{align*}
\sum_{n\in \N} \| \varphi_n - q_n \|^2  \leq C \sum_{n\in \N} \sum_{k\neq n} \left( \dfrac{1}{|n|^{\alpha-1}|k-n|} \right)^2
\end{align*}
which converges if $\alpha>3/2$. We refer to \cite{CGM, GLM} for the application of this method for a skew-adjoint case (Schrödinger equation) and a self-adjoint case (degenerate heat equation) respectively.

\subsubsection{Duality-compactness method}
When $|\lambda_n|\simeq n^\alpha$ with $1<\alpha\leq 3/2$, then the growth of the eigenvalues at infinity is not sufficient to prove the existence of a Riesz basis through the quadratically close criterion, a threshold that has been longstanding since the introduction of the Fredholm backstepping method.

In the recent paper \cite{WaterWave}, we proved for skew-adjoint operators $A$ with the spectral gap and other classical assumptions that the family $\{g_n\}_{n\in \NN^*}$ defined by \eqref{eq:defgn} is still a Riesz basis. The proof relies on a novel duality-compactness method. The duality is used to prove that the family $\{g_n\}_{n\in \NN^*}$ is complete and $\omega$-independent, whereas the compactness method consists in proving a sharp regularizing property for the operator $\varphi_n \mapsto g_n-\frac{1}{\lambda}\varphi_n$, which implies its compactness. In turn, we prove that the family $\{g_n\}_{n\in \NN^*}$ is obtained as an invertible Fredholm transformation of the orthonormal basis $\varphi_n$, which yields the Riesz basis property from Definition \ref{def:riesz}. 

\begin{remark}
We underline that the Fredholm backstepping method in the case $\alpha=1$ was not discussed above, as this case is very specific and often requires techniques of its own (see \cite{coron:hal-03161523}). 
\end{remark}

\subsubsection{Previous results}
We find in the literature two different approaches for the rapid stabilization with the backstepping method with a Fredholm transformation : either the operator $A$ is of first order ($\alpha=1$) or of  order strictly greater than 1 ($\alpha >1$). The case  $\alpha=1$ was essentially not treated in the discussion above as it seems to be a very specific case with techniques of its own. We refer for instance to cases where the rapid stabilization for hyperbolic systems was established in \cite{CoronHuOlive16, MR4160602} through direct methods or by identifying the isomorphism applied to the eigenbasis leading to the Riesz basis \cite{coron:hal-03161523, Zhang-finite, ZhangRapidStab}.

Most of the results found in the literature for operators of order strictly greater than 1 ($\alpha >1$) were established through the quadratically close criterion, thanks to the sufficient growth of the eigenvalues. Following the steps described in the previous section, the rapid stabilization was obtained for the linearized bilinear Schr\"odinger equation \cite{CGM}, the KdV equation \cite{CoronLu14}, the Kuramoto--Sivashinsky equation \cite{CoronLu15}, a degenerate parabolic operator \cite{GLM,zbMATH07897759} and finally the heat equation for which the backstepping is proved in sharp spaces \cite{Gagnon-Hayat-Xiang-Zhang}. More recently, we have established necessary and sufficient conditions for the Fredholm backstepping method for skew-adjoint operators of order strictly greater than 1 with the introduction of the duality-compactness method \cite{WaterWave}, removing the restriction on the growth of the eigenvalues.

We highlight as well the link between the backstepping method and the pole shifting methods (\cite{ho1986spectral}, \cite{rebarber1989spectral}, \cite{shun1981spectrum}, \cite{xu1996spectrum}). Some of these techniques have inspired the computations developed in Section \ref{sec-well-posedness}.

\subsection{Outline of the paper}

First, we present the strategy of the paper in the finite-dimensional case. We highlight the key role of the Cauchy matrix in solving the $TB=B$ condition. This is done in Section \ref{sec:finite-dim}. In Section \ref{sec:InvertibilityT}, we establish the existence of a Fredholm backstepping transformation for \eqref{eq:cl-sys} by establishing the link between an underlying isomorphism and the finite-dimensional Cauchy matrix. Preliminary estimates with respect to $\lambda$ and its distance to the non-resonant set is also begun in this section. These estimates are completed in Section \ref{sec:quantitative}, where most of the technical estimates on $\|T\|$ and  $\|T^{-1}\|$ are performed. 
These sharp estimates allows us to conclude on small-time null controllability results in Section \ref{sec-null-control}. 
Prior to this last section, we prove well-posedness results of the closed-loop system in Section \ref{sec-well-posedness}.

\section{Finite-dimensional example}
\label{sec:finite-dim}
We start this article by illustrating the methodology used in this article using a finite-dimensional example, highlighting the differences with the existing literature.  

We use the eigenvectors $\{\varphi_n\}_{1\leq n \leq d} \in \R^d$ of $A$ to construct a set of vectors $\{g_n\}_{1\leq n \leq d}$ solving the first equation $TA+BK=(A-\lambda I)T$. Usual methods for the Fredholm backstepping here is to prove that the $\{g_n\}_{1\leq n \leq d}$ forms a basis of $\R^d$. This basis is then used to solve the equation $TB=B$, thereby retrieving the coefficients $\{k_n\}_{1\leq n \leq d}$ of the feedback $K$.  

Instead, we inject the expression of $\{g_n\}_{1\leq n \leq d}$ into the $TB=B$ equation, and exhibit an invertible mapping, allowing us to recover the feedback coefficients $\{k_n\}_{1\leq n \leq d}$. This gives another proof of,
\begin{theorem}(\cite{Brunovsky,Coron_ICIAM15})
Let the pair $(A,B)\in M(\R^d)\times \R^d$ be controllable. Then, for any $\lambda>0$, there exists a unique transformation $T\in GL(\R^d)$ and a unique feedback $K^*\in \R^d$ such that 
\[
\begin{cases}
TA+BK = (A-\lambda I)T,\\  
TB=B.
\end{cases}
\]
\end{theorem}
Thanks to the Cauchy matrix, the expression of $T^{-1}$ is explicit, and its dependency in the damping parameter $\lambda>0$ is tractable. The generalization to infinite dimensional systems of the Cauchy matrix makes it possible to overcome the quadratically close criterion and, at the same time, to have sharp estimates on the cost of the backstepping method.

\subsection{Solving the main equation}
Consider a matrix $A\in M_d(\R)$, $B\in \R^d$, and suppose that $A$ is diagonalizable with simple eigenvalues $\{\lambda_n\}_{1\leq n \leq d}$ and eigenvectors $\{\varphi_n\}_{1\leq n \leq d}$. Assume that $(A,B)$ is controllable and recall that we aim to solve 
\begin{equation}\label{sec2:eqop}
\begin{cases}
TA+BK = (A-\lambda I)T,\\  
TB=B.
\end{cases}
\end{equation}
Applying $(A-\lambda I)T -TA = BK$ to the $\varphi_n$, we get:
\[
(A-(\lambda_n+\lambda) I)T\varphi_n = BK\varphi_n.
\]
Since $k_n:=K\varphi_n$ is a scalar, assuming $\lambda \notin \sigma(A-\lambda_n I), \forall n\in \{1, \dots, d\}$, we deduce, 
\[
T\varphi_n =k_n (A-(\lambda_n+\lambda) I)^{-1}B,
\]
and therefore
\begin{gather}\label{eq:Tnk}
T_{n,k}:= \left\langle T\varphi_n,\varphi_k\right\rangle = \dfrac{k_n b_k}{\lambda_k-\lambda_n-\lambda}, \quad \forall n,k\in \{1, ..., d\}, \\
T\varphi_n= \sum_{k= 1}^d T_{n, k} \varphi_k, \quad \forall n\in \{1, ..., d\}. \label{eq:trans:Tnk}
\end{gather}

\subsection{The uniqueness condition and expression of the feedback.} The uniqueness condition \eqref{intro:TBdecomp} then writes:
\begin{equation*}
TB=\sum_{n=1}^d b_nT\varphi_n = \sum_{n=1}^d b_n k_n(A-(\lambda_n + \lambda)I)^{-1} B,
\end{equation*}
where $B=\sum_{n=1}^d b_n \varphi_n$. At this point, the standard procedure calls to prove that the family $\{g_n\}_{ 1 \leq n \leq d}$ is a Riesz basis, where $g_n:=(A-(\lambda_n + \lambda)I)^{-1} B$. Here, we proceed differently. 

Using again the decomposition of $B$, one obtains:
\begin{align*}
\sum_{n=1}^d b_n k_n(A-(\lambda_n + \lambda)I)^{-1} B &=\sum_{n=1}^d b_n k_n \left(\sum_{k=1}^d b_k (A-(\lambda_n+\lambda)I)^{-1} \varphi_k \right) \\ 
&=\sum_{n=1}^d b_n k_n \left( \sum_{k=1}^d \dfrac{b_k}{\lambda_k-\lambda_n-\lambda} \varphi_k \right) \\
&= \sum_{k=1}^d\sum_{n=1}^d  \dfrac{b_n k_n b_k}{\lambda_k-\lambda_n-\lambda} \varphi_k
\end{align*}
Hence $TB=B$ amounts to,  
\[
\sum_{k=1}^d \sum_{n=1}^d \dfrac{k_n b_n b_k }{\lambda_k-\lambda_n-\lambda} \varphi_k = \sum_{k=1}^d b_k \varphi_k,
\]
and since the controllability assumption requires (using Fattorini-Hautus, for example), $b_k\neq 0$, for any $k=1, \ldots , d$ (otherwise the controllability of this mode fails), this is equivalent to,
\begin{equation}
\label{eq:TBBeq-0}
\sum_{n=1}^d \dfrac{k_n b_n}{\lambda_k-\lambda_n-\lambda}=1,\quad k=1, \ldots , d
\end{equation}
\subsection{The explicit isomorphism  via Cauchy matrix} This last equality can be viewed as solving the matrix system, 
\begin{equation}\label{eq:TBfinie}
\tilde{T}
\begin{pmatrix} 
k_1 b_1  \\
k_2 b_2 \\
\vdots \\
k_d b_d
\end{pmatrix} =
\begin{pmatrix} 
1  \\
1 \\
\vdots \\
1 
\end{pmatrix}
    \end{equation}

where the entries of the matrix $\tilde{T}\in M_d(\mathbb{C})$ are given by,
\[
\tilde{T}_{i, j}=\dfrac{1}{\lambda_i-\lambda_j-\lambda}.
\]

This induces a new transformation $\check T$:
\begin{equation}\label{def:tran:checkT}
   \left\langle \check T\varphi_j,\varphi_i\right\rangle:=  \tilde T_{i, j} \textrm{ thus } \footnote{Note that this expression is slightly different from the one in \eqref{eq:trans:Tnk} concerning the transformation $T$: in  \eqref{eq:trans:Tnk} the sum is over the second index of $T$, while in this expression we sum over the first index of $\tilde T$.}  \;  \check T \varphi_j= \sum_{i= 1}^d   \tilde T_{i, j} \varphi_i.
\end{equation}

This is a particular case of a Cauchy matrix,
\[
C_{ij}=\dfrac{1}{x_i-y_j},
\]
with $x_i\neq y_j$ for $1\leq i,j\leq d$. Here we define, for $\tilde{T}$, the entries $x_i=\lambda_i$, $y_j=\lambda_j+\lambda$. Notice that $x_i \neq y_j$ is ensured by imposing the classical non-resonance condition $\lambda \notin \sigma(A-\lambda_jI), j=1,\ldots, d$. Moreover, since the eigenvalues of $A$ are simple, we also have $x_i\neq x_j$ and $y_i \neq y_j$ for $i\neq j$. 
Under these hypothesis, the invertibility of the Cauchy matrix $C$ is ensured (see for instance \cite{zbMATH03152374}) and
\[
C^{-1}_{ij}=(x_j-y_i)A_j(y_i)B_i(x_j),
\]
with $A_i,B_i$ being the Lagrange polynomials for $x_i$ and $y_i$ respectively:
\[
A_i(x)=\dfrac{A(x)}{A'(x_i)(x-x_i)}=\prod_{j\neq i} \frac{x- x_j}{x_i- x_j}, \quad B_i(x)=\dfrac{B(x)}{B'(y_i)(x-y_i)}=\prod_{j\neq i} \frac{x- y_j}{y_i- y_j},
\]
with
\[
A(x)=\prod_{i=1}^d (x-x_i), \quad B(x)=\prod_{i=1}^d (x-y_i).
\]
Hence, we deduce the explicit expression of the matrix $\tilde{T}_{i, j}^{-1}$,
\begin{align}
(\tilde{T}^{-1})_{ij}
&= (x_j- y_i) \left( \dfrac{\prod_{n\neq j} (y_i- x_n)}{\prod_{n\neq j} (x_j- x_n)} \right) \left(\dfrac{\prod_{m\neq i} (x_j- y_m)}{\prod_{m\neq i} (y_i- y_m) } \right)  \notag \\
&= -\frac{1}{x_j- y_i}  \left(\dfrac{\prod_{n} (y_i- x_n)}{\prod_{n\neq j} (x_j- x_n) } \right) \left(\dfrac{\prod_{m} (x_j- y_m)}{\prod_{m\neq i} (y_i- y_m) } \right)  \notag \\
&= -\frac{1}{x_j- y_i}  \left(\dfrac{\prod_{m} (y_i- x_m)}{\prod_{m\neq i} (y_i- y_m) } \right) \left(\dfrac{\prod_{n} (x_j- y_n)}{\prod_{n\neq j} (x_j- x_n) } \right)  \notag  \\
&=\dfrac{\lambda^2}{(\lambda_j-\lambda_i-\lambda)}  \prod_{\substack{m=1 \\ m \neq i}}^d  \left( 1+ \dfrac{\lambda}{\lambda_i-\lambda_m}  \right)  \prod_{\substack{n=1 \\ n \neq j}}^d \left( 1- \dfrac{\lambda}{\lambda_j-\lambda_n}  \right). 
\label{eq:expr-T-1}
\end{align}
Then, the coefficients $k_j$ are obtained through \eqref{eq:TBfinie}, using again the fact that $b_j \neq 0$. 

We define a new transformation $\check T^{-1}$ by 
\begin{equation}
     \left\langle \check T^{-1}\varphi_j,\varphi_i\right\rangle:= \tilde T_{i, j}^{-1} \textrm{ thus }  \;  \check T^{-1} \varphi_j= \sum_{i= 1}^d   \tilde
     T_{i, j}^{-1} \varphi_i .
\end{equation}

\subsection{Invertibility of the transformation $T$.} The coefficient of the feedback $\{k_n\}_{1\leq n \leq d}$ having been identified, we are able to go back to the transformation $T$ and establish its invertibility, 
\begin{align}
  Tf=\sum_{n=1}^d f_n T \varphi_n &= \sum_{n=1}^d f_n k_n (A-(\lambda_n+\lambda)I)^{-1}B  \notag\\
  &= \sum_{n=1}^d f_n k_n \sum_{k=1}^d \dfrac{b_k}{\lambda_k-\lambda_n-\lambda} \varphi_k \notag\\ 
   &= \sum_{k=1}^d \sum_{n=1}^d f_n k_n  \dfrac{b_k}{\lambda_k-\lambda_n-\lambda} \varphi_k \notag\\
   &= \tau_B \check{T} \tau_K f \label{T-tauBK-checkT}
\end{align}
where, 
\begin{equation}\label{def:eq:tauBKTcheck}
    \tau_K: \varphi_n \mapsto k_n \varphi_n, \quad \tau_B : \varphi_n \mapsto b_n \varphi_n, \quad \check{T} : \varphi_n \mapsto \sum_{k=1}^d  \dfrac{1}{\lambda_k-\lambda_n-\lambda}\varphi_k
\end{equation}

 $\check{T}$ is given by \eqref{def:tran:checkT}, and $\tau_B$ is invertible by hypothesis on the $b_n$. Hence, the invertibility of $T$ hinges on the invertibility of $\tau_{K}$ and therefore on $k_j$ being non-zero. This can be proved directly from the explicit representation of $k_n$ through $\tilde{T}^{-1}$, but it is quite technical. In the infinite-dimensional setting, this will be established through some technical estimates on the infinite sums and products involved. 

Formally we have,
\begin{equation}
    T^{-1}= \tau_K^{-1}\check T^{-1}\tau_B^{-1}
\end{equation}
 
Another way to obtain the invertibility of $T$ in finite-dimension is to prove that $\ker \, T=0$. To that end, note that if $x\in \textrm{Ker}(T^*)$, then by $TB=B$, we have 
\[
B^* x = B^* T^* x = 0.
\]
On the other hand, from the operator equation \eqref{eq:operator-identity}, we have
\[T^\ast (A^\ast -\lambda I) x=(A+BK)^\ast T^\ast x=0,\]
which implies that $\ker (T^\ast)$ is stable by $A^\ast$, hence invariant under the action of the semigroup $e^{tA^\ast}$. Combining the above two assertions yields 
\[B^\ast e^{tA^\ast} x=0,\]
which, by weak observability, implies that $x=0$. Hence, $T$ is invertible.

\section{Construction and invertibility of $T$}
\label{sec:InvertibilityT}

We now move on to the construction of $T$ and $K$ in the infinite-dimensional setting. Apart from the classical functional framework briefly recalled in Section \ref{sec-func-framework}, the whole construction developed in Sections \ref{sec:def:TK}--\ref{sec:invertT} relies entirely on the newly introduced infinite-dimensional Cauchy matrix structure, without using Riesz basis or compactness arguments.

\subsection{Functional framework} 
\label{sec-func-framework}
 We begin by introducing the functional framework and the assumptions on $A :D(A) \subset H \rightarrow H$ and $B \in \mathcal{L}(U;D(A)')$.
We will work with the following assumption

\begin{assumption}
\label{assume:structure}

Assumption \ref{assume:structure0} holds and, in addition
\begin{itemize}
\item if $A$ is self-adjoint, it is negative. Hence the eigenvalues of the operator are real, countable, and
        \[
    0<-\lambda_1 \leq -\lambda_2 \leq \ldots \leq -\lambda_n \rightarrow +\infty
    \]
\item if $A$ is skew-adjoint, its eigenvalues satisfy
    \[
    0<-i \lambda_1 \leq -i \lambda_2 \leq \ldots \leq -i \lambda_n \rightarrow +\infty
    \]
\end{itemize}
\end{assumption}

\begin{remark}
Compared to Assumption \ref{assume:structure0}, in Assumption \ref{assume:structure}, we assume that either $A$ or $iA$ is negative to simplify the computations. In particular, we assume that no eigenvalues are equal to zero. Note \textcolor{black}{that this is not needed and} is without loss of generality : the more general framework with a finite number of unstable eigenvalues (that is, with positive real part of the eigenvalues) can be treated by the change of variables $x(t):=e^{\mu t}y(t)$, leading to the operator $A- \mu I$ being negative for $\mu>0$ large enough. 
\end{remark}

\begin{remark}[Case $m=1$] Remark that if $\textrm{dim} \, U=m,$ with $ 1<m<+\infty $, and if $A$ admits an orthonormal basis of eigenfunctions, we can write $y=\sum_{i=1}^m y_i(t)$ with $y_i$ satisfying $\dot{y}_i = Ay_i + B_i u_i$ (see Assumption \ref{assume:structure}). Hence, we consider the case $m=1$ for the rest of the paper.
\end{remark}

\begin{remark}[Simple eigenvalues]
Item 4 implies that one can handle the case where all the eigenvalues are simple, without loss of generality: the general case reduces to a finite number of spaces on which one considers an operator with simple eigenvalues.
\end{remark}

\begin{remark}
It is easy to see in our context that Assumption \ref{assume:structure0}, item 3, implies that the operator $B$ is admissible in $H$ (see for instance \cite{TucsnakWeissBook} for a definition).
\end{remark}

\begin{remark}
Let us underline that item 4 of Assumption \eqref{assume:structure0} are satisfied in many cases, in particular if the operator is self-adjoint with simple eigenvalues given by $\lambda_{n}\sim -n^{\alpha}$ with $\alpha>1$.
\end{remark}

Finally, we point out that estimate \eqref{eq:cond3} in item 4 of Assumption \ref{assume:structure0} implies a third gap estimate:
\begin{lemma}\label{lem:thirdgap}
Assuming \eqref{eq:cond3}, then there exists $c>0$ such that,
\begin{equation}\label{eigenvalue-growth}
    |\lambda_n-\lambda_k|\geq c|n-k|^\alpha, \quad \alpha>1, \quad \forall k, n \in \N^\ast,
\end{equation}
\end{lemma}
\begin{proof}
Without loss of generality, assume $k>n$, the case $n=k$ being trivial. Then, from \eqref{eq:cond3},
\[
|\lambda_n-\lambda_k|\geq ck^{\alpha-1}|k-n| \textrm { and } |\lambda_n-\lambda_k|\geq cn^{\alpha-1}|k-n|
\]
and thus, the conclusion comes if we are able to prove,
\[
\max \{k^{\alpha-1}, n^{\alpha-1}\}\geq |k-n|^{\alpha-1},
\]
but this comes from the monotonicity of the function $x \in \R^+ \mapsto x^{\alpha-1}$ with $\alpha>1$. 
\end{proof}

For the remainder of the article, since our method is essentially designed for operators with simple eigenvalues, we will present our results in this framework. 
We now consider a real self-adjoint unbounded operator $A$ on $H$, with simple eigenvalues $\{\lambda_n\}_{n\in \N^\ast}$ and an orthogonal family of corresponding eigenvectors $\{\varphi_n\}_{n\in \N^\ast}.$ As usual, we define the domain of $A$ by
\[D(A):=\{f\in H, Af \in H\}.\]

Moreover, we introduce a family of rigged Banach spaces (and refer to \cite{staffans2005well} for the detailed construction and further properties), which can be defined spectrally since $A$ is diagonalizable.
For $s\geq 0$, we define
\[\begin{aligned}
    \HH^s&:=\{f\in H, \quad \sum_{n\in \N^\ast}n^{2s} |\langle f, \varphi_n \rangle|^2 <+\infty \}, \\
    D(A^s)&:=\{f\in H, \quad \sum_{n\in \N^\ast}|\lambda_n|^{2s} |\langle f, \varphi_n \rangle|^2 <+\infty \},
\end{aligned}
\]
Notice that if $cn^\alpha \leq |\lambda_n| \leq Cn^\alpha$, then $D(A^s)=\HH^{\alpha s}$. 
\begin{remark}\label{rem:DAs}
Notice that here and below, we use the abusive notation $D(A^s)$. Indeed, only $(-A)^s$ is uniquely defined for $s\in \R$ as $-A$ is positive. This is to alleviate notations, and throughout this paper, $D(A^s)$ should be understood as $D((-A)^s)$.    
\end{remark}

The spaces $\HH^s$ are naturally endowed with a Hilbert structure given by the scalar product
\[\langle f,g \rangle_s:=\sum_{n\in \N^\ast} (n^{2s}) \langle f, \varphi_n\rangle \langle g, \varphi_n \rangle, \quad \forall f,g \in \HH^s,\]
as well as a natural Hilbert basis given by $\left(\frac1{\sqrt{1+|n|^{2s}}}\varphi_n\right)_{n\in \N^\ast}$.

The spaces $D(A^s), \, s\geq 0$ can in turn be endowed with the closely related Hilbert structure (recall that $\lambda_n \neq 0$):
\[\langle f, g\rangle_{D(A^s)}=\sum_{n \in\N^\ast} |\lambda_n|^{2s}\langle f, \varphi_n \rangle \langle g, \varphi_n \rangle, \quad \forall f, g \in D(A^s).\]

For $s<0$, $\HH^s$ (resp. $D(A^s)$) can be defined by identifying it with the dual space $(\HH^s)^\prime$ (resp. $D(A)^\prime$), with pivot space $H$.

For the remainder of this article, we will denote by:
\begin{itemize}
    \item $\|\cdot\|$ the norm of the Hilbert space $H$;
    \item $\|\cdot\|_s$ the norm associated to the scalar product $\langle \cdot, \cdot \rangle_s$ on $\HH^s$, for $s\neq 0$;
    \item $\|\cdot\|_{D(A^s)}$ the norm associated to the scalar product $\langle \cdot, \cdot \rangle_{D(A^s)}$ on $D(A^s)$, for $s\neq 0$.
    \item For linear mappings on $H$ (resp. $D(A^s)$ or $D_s(A+BK)$ in Sections \ref{sec-well-posedness} the usual operator norm will be denoted $\|\cdot\|$ (resp. $\|\cdot\|_{D(A^s)}$ and $\|\cdot\|_{D_s(A+BK)}$).
\end{itemize}

Importantly, for $s<0$, $D(A^s)$ can also be defined as the completion of $H$ with respect to the norm $\|\cdot\|_{\rho, s}=\|(A-\rho I)^{s} \cdot \|$ for well-chosen $\rho$ in the resolvent set of $A$ (the fractional negative power of $A-\rho I$ can be defined using the Laplace transform, see for example \cite[Chapter IX,Section 11]{yosida2012functional}, or \cite[Proposition 4.20]{li2012optimal}).

Finally, note that if $s_1<s_2$, the injections $\HH^{s_2} \to \HH^{s_1}$ and $D(A^{s_1})\to D(A^{s_2})$ are continuous and compact, due to the fact that $A$ has compact resolvent.

\subsection{Definition of the transformation $T$ and the feedback $K$}\label{sec:def:TK}

 Recall that we work on a given invariant subspace $H_i$ where $A$ has simple eigenvalues. In the following, we denote $H_i$ and $(\lambda_{i, n}, \varphi_{i, n})$ by $H$ and $(\lambda_{n}, \varphi_{n})$ to alleviate notations.

Let us define the operators, at least formally for now, used throughout this article. For $\lambda >0$ such that $\lambda_j+\lambda-\lambda_i \neq 0, \, \forall i,j\in \N^\ast$,  we denote
\begin{align}
K \, &: \, \varphi_n \mapsto k_n \in \R, \quad \qquad \qquad \, \forall n\in \N^\ast, \label{def-K} \\
T \, : \, \varphi_n \mapsto k_n & (A-(\lambda_n+\lambda)I)^{-1}B \in H, \quad \forall n\in \N^\ast. \label{def-T}
\end{align}
Thus 
\begin{align}\label{eq:Tnk:full}
T_{n,k}:=& \left\langle T\varphi_n ,\varphi_k \right \rangle = \dfrac{k_n b_k}{\lambda_k-\lambda_n-\lambda}, \quad \forall n,k\in \mathbb{N}^*, \\
& \quad T\varphi_n= \sum_{k\in \mathbb{N}^*} T_{n, k} \varphi_k, \quad \quad  \, \quad \forall n\in \mathbb{N}^*. \label{eq:trans:Tnk:full}
\end{align}

Motivated by the finite dimensional example, one key element of this article is to introduce and prove several quantitative estimates, with respect to $\lambda$, on the auxiliary operator $\check{T}$ associated to the infinite dimensional extension of the Cauchy matrix $(\tilde T_{i, j})$. Similar to \eqref{def:tran:checkT}, we define
\begin{equation}\label{eq:def:tildT:full}
    \tilde{T}_{i, j}=\dfrac{1}{\lambda_i-\lambda_j-\lambda},
\end{equation}
and a transformation related to the matrix $(\tilde T_{i, j})$:
\begin{equation}\label{eq:deftildeT:full}
     \left\langle \check T\varphi_j,\varphi_i\right\rangle:=  \tilde T_{i, j} \textrm{ thus }  \;  \check T \varphi_j= \sum_{i\in \mathbb{N}^*}   \tilde T_{i, j} \varphi_i.
\end{equation}

Up to this point, the construction remains within the standard framework; see \cite{CGM, coron:hal-03161523, WaterWave}. The main novelty of the present construction is the explicit definition and characterization of $T^{-1}$ and $\check T^{-1}$ below, together with the explicit identification of the feedback operator $K$ later on. 
We also define another auxiliary transformation $\check{T}^{-1}$ associated to the inverse of the Cauchy matrix $(\tilde T_{i, j})$, which we shall prove to be the left inverse of $\check{T}$, by generalizing identity \eqref{eq:expr-T-1}: define
\begin{equation}\label{eq:expr-T-1:full}
 \tilde T_{i, j}^{-1} =\dfrac{\lambda^2}{(\lambda_j-\lambda_i-\lambda)}  \prod_{\substack{m\in \N^\ast \\ m \neq i}} \left( 1+ \dfrac{\lambda}{\lambda_i-\lambda_m}  \right)  \prod_{\substack{n\in \N^\ast \\ n \neq j}} \left( 1- \dfrac{\lambda}{\lambda_j-\lambda_n}  \right), \quad  \forall i,j\in \N^\ast,
\end{equation}
and 
\begin{equation}\label{eq:defi:T-1:full}\left<\check{T}^{-1}\varphi_j,\varphi_i\right>:=  \tilde{T}^{-1}_{i, j} \textrm{ thus }  \check T^{-1} \varphi_j= \sum_{i\in \mathbb{N}^*}   \tilde T_{i, j}^{-1} \varphi_i .
\end{equation}
Again we emphasize that \eqref{eq:trans:Tnk:full} concerning the transformation $T$ is slightly different from \eqref{eq:defi:T-1:full} on the transformation $\check T$.

Much of this article will revolve around the analysis of $\tilde{T}$, and $\tilde{T}^{-1}$. 
Its explicit representation will allow us in particular to deduce quantitative bounds on the operator norms $\|T\|$, $\|T^{-1}\|$, $\|T\|_{D(A^s)}$, $\|T^{-1}\|_{D(A^s)}$ with respect to $\lambda$. This allows us to obtain stability estimates that are quantitative with $\lambda$, a key difference with the previous analysis, for instance \cite{WaterWave}. This is also what allows us to derive a small-time controllability result (see Theorem \ref{thm:intro3}).

In the remainder of this section, we proceed as follows: in Section \ref{sec:technicalestim} we provide some technical estimates that will be useful in the analysis; in Section \ref{sec:analysistildeT}, we show that $\check{T}$ and $\check{T}^{-1}$ defined similarly as in the finite-dimensional case by \eqref{def:tran:checkT}--\eqref{eq:expr-T-1} are still well-defined operators in the infinite-dimensional setting and that $\check{T}^{-1}$ is indeed the inverse of $\check{T}$;  in Section \ref{sec:TBB} we show that we may choose $K$ appropriately such that the condition $TB=B$ holds and we give an explicit expression of $K$. This is a major benefit of this approach, in contrast with previous approaches such as \cite{GLM, Gagnon-Hayat-Xiang-Zhang,WaterWave}.

\subsubsection{Some technical estimates} 
\label{sec:technicalestim}

We begin by introducing a quantity that will play a fundamental role in the quantitative estimates, 
\begin{equation}
\label{eq:distalpha}
\Dist_\alpha(\lambda):=\inf_{i,j\in \N^\ast}|\lambda_j-\lambda_i+\lambda|. 
\end{equation}
Let us define the following countable set of numbers 
\begin{equation}\label{def:mathcalN}
 \mathcal{N}:= \{\lambda_i- \lambda_j: i, j\in \mathbb{N}^*\},  
\end{equation}
Clearly, $\text{Dist}_{\alpha}(\lambda) = 0$ if and only if $\lambda\in\mathcal{N}$.
We list here some properties related to $\Dist_\alpha(\lambda)$ that will be used throughout this article, as well as technical estimates on the $\lambda_j$. The proofs are in Appendix \ref{appendix-estimates-lambda}.

\begin{lemma}\label{lem:di}
Assume that the positive number $\lambda\notin \mathcal{N}$.
For every $i\in \N^*$, define 
\[
D_i(\lambda)=\min_{j\in \N^*} | \lambda_j-\lambda_i+\lambda|. 
\]
Then, 
\begin{itemize}
\item[(i)] There exists at most two $j_1,j_2 \in \N^*$ such that 
\[
| \lambda_{j_1}-\lambda_i+\lambda|=| \lambda_{j_2}-\lambda_i+\lambda|=D_i(\lambda),
\]
and $|j_1-j_2|=1$. 
\item[(ii)]
\[
\Dist_\alpha(\lambda) \leq D_i(\lambda), \quad \forall i\in \N^*.
\]
\end{itemize}
\end{lemma}

The following lemma ensures that the infinite product  \eqref{eq:expr-T-1} is well-defined.

\begin{lemma}\label{lem:boundT-1}
There exists some effectively computable constant $C>0$, such that for any given $\lambda>0$ and any $i,j \in \mathbb{N}^*$ there is
 \begin{itemize}
     \item[(i)]  $ \left| \prod_{\substack{m\in \N^\ast\\ m \neq i}} \left( 1+ \dfrac{\lambda}{\lambda_i-\lambda_m}  \right) \right|\leq C e^{C \lambda^{\frac{1}{\alpha}}}, \quad \forall \lambda\in (0, +\infty), \; \forall i\in \N;$
     \item[(ii)]  $ \left| \prod_{\substack{n\in \N^\ast\\ n \neq j}} \left( 1- \dfrac{\lambda}{\lambda_j-\lambda_n}  \right) \right| \leq C  e^{C \lambda^{\frac{1}{\alpha}}}, \quad \forall \lambda\in (0, +\infty), \; \forall j\in \N;$
     \end{itemize}
      Moreover, the power $\frac{1}{\alpha}$ in (i)-(ii) is sharp.
 \end{lemma}
 
Finally, the following estimate will be useful in Section \ref{sec:quantitative} concerning the upper and lower bound of $\|T^{-1}\|$.
\begin{lemma}\label{lem:sumlambdaijdistbound}
There exists $C>0$ such that, for  any $\lambda>0$ satisfying  $\lambda\not \in \mathcal{N}$, and any $i, j\in \mathbb{N}^*$ there is 
 \item[(i)] $\sum_{j\in \mathbb{N}^*} \left|\dfrac{\lambda^2}{(\lambda_j-\lambda_i-\lambda)}\right|\leq C \left(  \lambda^2  +  \dfrac{\lambda^2}{\Dist_\alpha(\lambda)} \right), \quad \forall i\in \N$;
 \item[(ii)] $\sum_{i\in \mathbb{N}^*} \left|\dfrac{\lambda^2}{(\lambda_j-\lambda_i-\lambda)}\right|\leq C \left(  \lambda^2  +  \dfrac{\lambda^2}{\Dist_\alpha(\lambda)} \right), \quad \forall j\in \N$.
 \end{lemma}

\subsubsection{Study of the auxiliary transformations $\check{T}$ and $\check{T}^{-1}$}
\label{sec:analysistildeT}

Recall the definition of the critical set $\mathcal{N}$ in \eqref{def:mathcalN}. From now on we fix the value of $\lambda \notin \mathcal{N}$, and investigate the transformations $\check{T}= \check{T}(\lambda)$ and $\check{T}^{-1}= \check{T}^{-1} (\lambda)$ given by \eqref{eq:deftildeT:full} and \eqref{eq:defi:T-1:full}. For the ease of notations, we omit the index $\lambda$.

First we prove that the operator $\check{T}^{-1}$ defined above is well-defined. Recall that $\check T$ is related to the matrix $(\tilde T_{i, j})$, and $\check T^{-1}$ is related to the matrix $(\tilde T_{i, j}^{-1})$ given by \eqref{eq:def:tildT:full} and \eqref{eq:expr-T-1:full}.

\begin{proposition}\label{prop:regT-1}
Let $\lambda \notin \mathcal{N}$.
The entries $(\tilde{T}^{-1})_{ij}, i,j\in \N^*$ are well-defined.  Moreover, 
\begin{itemize}
    \item[(i)] for every $j\in \N^*$, $\{( \tilde{T}^{-1})_{ij}\}_{i\in \N^*} \in \ell^1(\N^*;\mathbb{C})$ and $\sup_{j \in \N^*} \| \{( \tilde{T}^{-1})_{ij}\}_{i\in \N^*}\|_{\ell^1} \leq C < \infty$. 
     \item[(ii)] for every $i\in \N^*$, $\{( \tilde{T}^{-1})_{ij}\}_{j\in \N^*} \in \ell^1(\N^*;\mathbb{C})$ and $\sup_{i \in \N^*} \| \{( \tilde{T}^{-1})_{ij}\}_{j\in \N^*}\|_{\ell^1} \leq C < \infty$. 
\end{itemize}
\end{proposition}

\begin{proof}[of Proposition \ref{prop:regT-1}]
From assertion (i) and (ii) of Lemma \ref{lem:boundT-1}, we deduce that,
\begin{equation}\label{eq:firstboundT-1}
|(\tilde{T}^{-1})_{ij}| \leq C e^{C \lambda^{\frac{1}{\alpha}}}\left|\dfrac{\lambda^2}{(\lambda_j-\lambda_i-\lambda)}\right|,
\end{equation}
where the constant $C$ is independent of $\lambda$ and $i, j\in \mathbb{N}^*$.
Moreover, from assertion (ii) of Lemma \ref{lem:sumlambdaijdistbound},
\[
\sum_{i\in \N^*} |(\tilde{T}^{-1})_{ij}| \leq C e^{C \lambda^{\frac{1}{\alpha}}} \sum_{i\in \N^*} \left| \dfrac{\lambda^2}{\lambda_j-\lambda_i-\lambda} \right| \leq  C  e^{C \lambda^{\frac{1}{\alpha}}}\left(  \lambda^2  +  \dfrac{\lambda^2}{\Dist_\alpha(\lambda)} \right),
\]
where the constant $C$ is independent of $\lambda$ and  $j\in \mathbb{N}^*$. The second property can be proved using the same argument. 
\end{proof}

We now show that the entries $(\tilde{T}_{ij}^{-1})_{(i,j)\in\mathbb{N}^{*}}$ indeed define a continuous operator $\check{T}^{-1}$ on $H$.

\begin{proposition}
\label{prop:bound}
 Let $\lambda \notin \mathcal{N}$. The operators $\check T$ defined by \eqref{eq:deftildeT:full}, and  $\check{T}^{-1}$ defined by \eqref{eq:defi:T-1:full}, $i. e.$ 
\begin{gather}
 \check{T}: \varphi_{n}\rightarrow \sum\limits_{i\in\mathbb{N}^{*}}\tilde{T}_{i, n}\varphi_{i}, \\
    \check{T}^{-1}: \varphi_{n}\rightarrow \sum\limits_{i\in\mathbb{N}^{*}}\tilde{T}^{-1}_{i, n}\varphi_{i}
\end{gather}
are bounded linear operators from $H$ to $H$. Moreover,
\begin{gather}
\|\check{T}\| \leq C  \left( 1  +  \dfrac{1}{\Dist_\alpha(\lambda)} \right), \label{eq:boundcheckT} \\
  \|\check{T}^{-1}\| \leq C  \left( 1  +  \dfrac{1}{\Dist_\alpha(\lambda)} \right)   e^{C \lambda^{\frac{1}{\alpha}}}, \label{eq:boundT-1}
\end{gather}
where $C$ is independent of $\lambda$.
\end{proposition}
In fact, this bounded operator is indeed the inverse of $\check{T}$, as expected:

\begin{proposition}
\label{tinverse}
Let $\lambda \notin \mathcal{N}$. 
The operator $\check{T}^{-1}$ defined by \eqref{eq:defi:T-1:full} is the inverse of $\check{T}$ in $H$. 
\end{proposition}

\begin{proof}[of Proposition \ref{prop:bound}]
In the proof of this proposition, the constants $C$ may differ from line to line, but all of them are independent of the choice of $\lambda\notin \mathcal{N}$.
Using the expressions of $\check T$ and $\check{T}^{-1}$, we are able to prove that it is a bounded operator on $H$ from Lemma \ref{lem:boundT-1} : for every $f\in H$, there is
\begin{align*}
\|\check{T}^{-1} f \|^2 & = \left\| \sum_{n\in \N^*} f_n \check{T}^{-1} \varphi_n \right\|^2   \\
& = \left\| \sum_{n\in \N^*} f_n \left( \sum_{k\in \N^*} (\tilde{T}^{-1})_{k, n} \varphi_k\right) \right\|^2  \\
& = \left\| \sum_{k\in \N^*} f_k (\tilde{T}^{-1})_{kk} \varphi_k + \sum_{k\in \N^*} \left( \sum_{n\in \N^* \setminus\{k\}} f_n (\tilde{T}^{-1})_{k, n} \right) \varphi_k \right\|^2 \\
&\leq C    e^{C \lambda^{\frac{1}{\alpha}}}  \|f\|^2 + 2 \sum_{k\in \N^*} \left| \sum_{n\in \N^* \setminus\{k\}} f_n (\tilde{T}^{-1})_{k, n} \right|^2 \\ 
& \leq C  e^{C \lambda^{\frac{1}{\alpha}}}  \| f \|^2  + 2C e^{C \lambda^{\frac{1}{\alpha}}} \sum_{k \in \N^*} \left| \sum_{n\in \N^* \setminus \{k\}} \dfrac{|f_n|}{|\lambda_n-\lambda_k-\lambda|} \right|^2, 
\end{align*}
where in the last line we have used the inequality \eqref{eq:firstboundT-1}.
Thanks to the Cauchy-Schwarz inequality and Lemma  \ref{lem:sumlambdaijdistbound}, 
\begin{align*}
&  \sum_{k \in \N^*} \left| \sum_{n\in \N^* \setminus \{k\}} \dfrac{|f_n|}{|\lambda_n-\lambda_k-\lambda|} \right|^2  \\
&  \leq \sum_{k \in \N^*}  \left(\sum_{n\in \N^* \setminus \{k\}} \dfrac{|f_n|^2}{|\lambda_n-\lambda_k-\lambda|} \right) \left( \sum_{n\in \N^* \setminus \{k\}}\dfrac{1}{|\lambda_n-\lambda_k-\lambda|}  \right)  \\
&   \leq  C \left(  1  +  \dfrac{1}{\Dist_\alpha(\lambda)} \right) \sum_{k \in \N^*} \sum_{n\in \N^* \setminus \{k\}} \dfrac{|f_n|^2}{|\lambda_n-\lambda_k-\lambda|}  \\
&  \leq  C \left(  1  +  \dfrac{1}{\Dist_\alpha(\lambda)} \right) \sum_{k \in \N^*} \sum_{n\in \N^*} \dfrac{|f_n|^2}{|\lambda_n-\lambda_k-\lambda|}  \\
& \leq  C \left(  1 +  \dfrac{1}{\Dist_\alpha(\lambda)} \right) \sum_{n\in \N^*} \sum_{k \in \N^*} \dfrac{|f_n|^2}{|\lambda_n-\lambda_k-\lambda|}  \\
& \leq  C \left(  1  +  \dfrac{1}{\Dist_\alpha(\lambda)} \right)^2 \sum_{n\in \N^*} |f_n|^2  \leq  C \left(  1  +  \dfrac{1}{\Dist_\alpha(\lambda)} \right)^2 \| f \|^2 .
\end{align*}
Thus 
\begin{equation}
  \|\check{T}^{-1} f \| \leq C  \left( 1  +  \dfrac{1}{\Dist_\alpha(\lambda)} \right)   e^{C \lambda^{\frac{1}{\alpha}}}  \|f \|.
\end{equation}
This ends the proof of boundedness of $\check T^{-1}$ and gives directly \eqref{eq:boundT-1}. The proof of \eqref{eq:boundcheckT} is similar and even simpler, due to the simpler expression of $\tilde T_{i, j}$.
\end{proof}

We now turn to the proof of Proposition \ref{tinverse}

\begin{proof}[of Proposition \ref{tinverse}]
Let us fix the value of $\lambda\not \in \mathcal{N}$.
Recall the definition of $\check{T}$: 
\begin{equation}
    \check{T} : \varphi_{n} \mapsto \sum_p \tilde T_{p, n}\varphi_p=  \sum_p \frac{\varphi_p}{\lambda_p- \lambda_n\ - \lambda}\;\; \textrm{ with }\;\; \tilde{T}_{p, n}:= \frac{1}{\lambda_p- \lambda_n- \lambda}.
\end{equation}

Recall the definition of $\check{T}^{-1}$
\begin{equation}
    \check{T}^{-1} : \varphi_{n}  \mapsto \sum_i \tilde T_{i, n}^{-1}\varphi_i,
\end{equation}
where
\begin{equation}
    \tilde{T}^{-1}_{i, j}:= \frac{\lambda^2}{\lambda_j- \lambda_i- \lambda}\prod_{m\neq i} \left(1+ \frac{\lambda}{\lambda_i- \lambda_m}\right) \prod_{n\neq j} \left(1- \frac{\lambda}{\lambda_j- \lambda_n}\right).
\end{equation}
We will show that $\check{T}^{-1}$ is in fact the inverse of $\check{T}$, i.e.  
\begin{equation*}
    \check{T}\circ \check{T}^{-1}= \check{T}^{-1}\circ \check{T}= Id: H \rightarrow H.
\end{equation*}

We focus on the case of $ \check{T}^{-1}\circ \check{T}$, while the other case can be treated similarly.
Thanks to the boundedness of $\check{T}$ and $\check{T}^{-1}$ (see Proposition \ref{prop:bound}), this is a bounded operator from $H$ to itself. Since $\{ \varphi_n \}_{n\in \NN^*}$ is an orthogonal basis of $H$, it suffices to show that 
\begin{equation*}
   ( \check{T}^{-1}\circ \check{T}) \varphi_n= \varphi_n, \; \forall n\in \N^*.
\end{equation*}
Since 
\begin{align*}
  ( \check{T}^{-1}\circ \check{T}) \varphi_n= \check{T}^{-1}(\check{T}) \varphi_n  
  = \check{T}^{-1} \left(\sum_p \tilde T_{p, n} \varphi_p \right)
  &= \sum_p \tilde T_{p, n} \sum_k \tilde T^{-1}_{k, p} \varphi_k = \sum_k \varphi_k \left(\sum_p \tilde T^{-1}_{k, p} \tilde T_{p, n}  \right),
\end{align*}
this is equivalent to 
\begin{equation}\label{eq:deltaRS}
    \sum_{i\in \mathbb{N}^*} \tilde{T}^{-1}_{ji} \tilde{T}_{in}= \delta_{j,n}, \; \forall j, n\in \N^*,
\end{equation}
where $\delta_{j,n}$ denotes the Kronecker symbol. The proof is inspired from the finite dimensional Cauchy matrix (\cite{zbMATH03152374}). For any $N\in \N^*$ we introduce the finite (truncated) matrices 
\begin{equation}
    \tilde{T}^N:= (\tilde{T}_{ij}^N)_{1\leq i,j\leq N}, \;  (\tilde{T}^{-1})^N:= ((\tilde{T}^{-1}_{ij})^N)_{1\leq i,j\leq N}\in \mathbb{R}^{N\times N} 
\end{equation}
with, 
\begin{gather}
    \tilde{T}^N_{ij}= \tilde{T}_{ij}= \frac{1}{\lambda_i- \lambda_j- \lambda}, \quad  1\leq i,j\leq N, \\
    (\tilde{T}^{-1}_{ij})^N=  \frac{\lambda^2}{\lambda_j- \lambda_i- \lambda}\prod_{\substack{m\neq i\\ 1\leq m\leq N}} \left(1+ \frac{\lambda}{\lambda_i- \lambda_m}\right) \prod_{\substack{n\neq j\\ 1\leq n\leq N}} \left(1- \frac{\lambda}{\lambda_j- \lambda_n}\right), \quad 1\leq i,j\leq N.
\end{gather}
By denoting $x_i:= \lambda_i$ and $y_j:= \lambda_j+ \lambda$, we notice that $(\tilde{T}^N_{i, j})$ is a Cauchy matrix as $\tilde{T}_{ij}^N= 1/(x_i- y_j)$. This implies that the $N\times N$ matrix $(\tilde{T}^N_{i, j})$ is invertible with its inverse being $(\tilde{T}^{-1})^N= ((\tilde{T}^{-1})^N_{ij})$, namely 
\begin{equation}
      \sum_{1\leq i\leq N} (\tilde{T}^{-1})^N_{ji} \tilde{T}^N_{in}= \delta_{j,n}, \quad \forall 1\leq j, n\leq N. 
\end{equation}

For any fixed  $j, n\in \N^*$, we consider $N\geq |j+ n|$, the previous equation yields
\begin{equation}\label{eq:deltafinite}
      \sum_{1\leq i\leq N} (\tilde{T}^{-1})^N_{ji} \tilde{T}_{in}= \sum_{1\leq i\leq N} (\tilde{T}^{-1})^N_{ji} \tilde{T}_{in}^N=  \delta_{j,n}, \quad \forall N \geq |j+ n|.
\end{equation}

We know that for any $N\geq |j+n|$ and any $i$ satisfying  $i\leq N$ there is
\begin{align*}
    \tilde{T}^{-1}_{ji}&- (\tilde{T}^{-1}_{ji})^N= \tilde{T}^{-1}_{ji} \left(1-  \prod_{m\geq N+1} \left(1+ \frac{\lambda}{\lambda_j- \lambda_m}\right)^{-1} \prod_{n\geq N+1} \left(1- \frac{\lambda}{\lambda_i- \lambda_n}\right)^{-1}\right) 
\end{align*}

For any  $N\geq |j+n|+ (10\lambda/c)^{1/(\alpha-1)}$, where the constant $c$ is given in \eqref{eq:cond3}, and any $i$ satisfying  $i\leq N$ there is, 
\begin{equation*}
    \left|\frac{\lambda}{\lambda_j- \lambda_m}\right|\leq \frac{|\lambda|}{c|j- m|m^{\alpha-1}}\leq \frac{|\lambda|}{c N^{\alpha-1}}\leq \frac{1}{10}, \quad \forall m\geq N+1,
\end{equation*}
and thus,
\begin{equation*}
  \left|1- \frac{\lambda}{|\lambda_j- \lambda_m|}\right|\leq \left|1+ \frac{\lambda}{\lambda_j- \lambda_m}\right|\leq \left|1+ \frac{\lambda}{|\lambda_j- \lambda_m|}\right|, \quad \forall m\geq N+1.
\end{equation*}
One has similar estimates for the term 
\begin{equation*}
\left|1- \frac{\lambda}{|\lambda_i- \lambda_n|}\right|\leq  \left|1- \frac{\lambda}{\lambda_i- \lambda_n}\right|\leq \left|1+ \frac{\lambda}{|\lambda_i- \lambda_n|}\right|, \quad \forall n\geq N+1.
\end{equation*}
Since the value of $\lambda\not \in \mathcal{N}$ is fixed, one can show thanks to Assumption \ref{assume:structure}
the existence of $c_0= c_0(\lambda)>0$ such that 
\begin{equation}\label{boud:fix:lambda:ij}
    |\lambda_i- \lambda_j- \lambda|\geq c_0|\lambda_i - \lambda_j|, \forall i, j\in \mathbb{N}^*.
\end{equation}
In conclusion, we fix some $j, n\in \mathbb{N}^*$, then for any $N\geq 2|j+n|+ (10\lambda/c)^{1/(\alpha-1)}$ and any $i \leq N$, there is, 
\begin{align*}
\tilde{T}^{-1}_{ji}- (\tilde{T}^{-1}_{ji})^N&\leq |\tilde T^{-1}_{j, i}| 
\left(1-  \prod_{m\geq N+1} \left(1+ \frac{\lambda}{|\lambda_j- \lambda_m|}\right)^{-1} \prod_{n\geq N+1} \left(1+ \frac{\lambda}{|\lambda_i- \lambda_n|}\right)^{-1}\right)\\
&\leq |\tilde T^{-1}_{j, i}| \left(1- \exp\left(-\sum_{m\geq N+1} \log \left(1+ \frac{\lambda}{|\lambda_j- \lambda_m|}\right) 
- \sum_{n\geq N+1} \log \left(1+ \frac{\lambda}{|\lambda_i- \lambda_n|}\right)\right)\right) \\
&\leq |\tilde T^{-1}_{j, i}| \left(1- \exp\left(-\sum_{m\geq N+1}  \frac{\lambda}{|\lambda_j- \lambda_m|}- \sum_{n\geq N+1}  \frac{\lambda}{|\lambda_i- \lambda_n|}\right)\right)\\
& \leq |\tilde T^{-1}_{j, i}| \left(1- \exp\left(-\sum_{m\geq N+1}  \frac{\lambda}{c|j-m|m^{\alpha-1}}- \sum_{n\geq N+1}  \frac{\lambda}{c|i- n|n^{\alpha-1}}\right)\right) \\
& \leq |\tilde T^{-1}_{j, i}| \left(1- \exp\left(-\sum_{m\geq N+1}  \frac{2\lambda}{cm^{\alpha}}- \sum_{n\geq N+1}  \frac{2\lambda}{c n^{\alpha}}\right)\right)\\
& \leq |\tilde T^{-1}_{j, i}| \left(1- \exp\left(-\frac{4\lambda}{c (\alpha-1)} \cdot \frac{1}{N^{\alpha-1}}- \frac{C \log N}{N^{\alpha- 1}} \right)\right)\\
& \leq C e^{C \lambda^{\frac{1}{\alpha}}}\left|\dfrac{\lambda^2}{(\lambda_i-\lambda_j-\lambda)}\right| \left(1- \exp\left(-\frac{4\lambda}{c (\alpha-1)} \cdot \frac{1}{N^{\alpha-1}}- \frac{C \log N}{N^{\alpha- 1}} \right)\right)\\ 
& \leq C e^{C \lambda^{\frac{1}{\alpha}}}\left|\dfrac{\lambda^2}{(\lambda_i-\lambda_j-\lambda)}\right| \frac{\lambda+ \log N}{N^{\alpha-1}}\\
&\leq \frac{C}{|\lambda_i -\lambda_j|} \frac{\log N}{N^{\alpha-1}},
\end{align*}
where the constant $C= C(\lambda)$ is independent of $j, n$ and $N$ larger than $ 2|j+n|+ (10\lambda/c)^{1/(\alpha-1)}$ and $i\leq N$.
And similarly, together with the simple estimates,
\begin{equation*}
    \log (1- |x|)\geq - 2|x| \textrm{ and } 1- e^x\geq - 2x,  \quad \forall x\in [0, 1/10],
\end{equation*}
we obtain 
\begin{align*}
\tilde{T}^{-1}_{ji}- (\tilde{T}^{-1}_{ji})^N&\geq |\tilde T^{-1}_{j, i}| \left(1- \exp\left(-\sum_{m\geq N+1} \log \left(1- \frac{\lambda}{|\lambda_j- \lambda_m|}\right) \right. \right.  \\
& \qquad \qquad \qquad \qquad \qquad \qquad \left. \left.
- \sum_{n\geq N+1} \log \left(1- \frac{\lambda}{|\lambda_i- \lambda_n|}\right)\right)\right)\\
& \geq |\tilde T^{-1}_{j, i}|  \left(1- \exp\left(-\sum_{m\geq N+1} \left(- \frac{2\lambda}{|\lambda_j- \lambda_m|}\right)- \sum_{n\geq N+1}  \left(- \frac{2\lambda}{|\lambda_i- \lambda_n|}\right)\right)\right) \\
& = |\tilde T^{-1}_{j, i}|  \left(1- \exp\left(\sum_{m\geq N+1} \left( \frac{2\lambda}{|\lambda_j- \lambda_m|}\right)+ \sum_{n\geq N+1}  \left( \frac{2\lambda}{|\lambda_i- \lambda_n|}\right)\right)\right) \\
&\geq |\tilde T^{-1}_{j, i}|  \left(1- \exp\left( \frac{C (\lambda+ \log N)}{N^{\alpha-1}}\right)\right)\\
&\geq - \frac{C}{|\lambda_i -\lambda_j|} \frac{\log N}{N^{\alpha-1}}. 
\end{align*}
Therefore, there exists some $C= C(\lambda)$ such that for any $j, n\in \mathbb{N}^*$ , for any $N\geq 2|j+n|+ (10\lambda/c)^{1/(\alpha-1)}$ and for any $i\leq N$, there is
\begin{equation}\label{eq:esimconv}
\left|\tilde{T}^{-1}_{ji}- (\tilde{T}^{-1}_{ji})^N\right|\leq \frac{C\log N}{N^{\alpha-1}|\lambda_i -\lambda_j|}.
\end{equation}

Using the estimate \eqref{eq:esimconv} we can prove the required equality \eqref{eq:deltaRS}.
Fix $j, n\in \mathbb{N}^*$. By comparing \eqref{eq:deltaRS} and \eqref{eq:deltafinite} we obtain that as $N\in (2|j+n|+ (10\lambda/c)^{1/(\alpha-1)}, +\infty)$ tends to $\infty$ there is 
\begin{align*}
      \;\;\;\; |\sum_{i\in \N^*} & \tilde{T}^{-1}_{ji} \tilde{T}_{in} -  \sum_{1\leq i\leq N} (\tilde{T}^{-1})^N_{ji} \tilde{T}_{in}| \\
      &= |\sum_{i\geq N+1}  \tilde{T}^{-1}_{ji} \tilde{T}_{in}+ \sum_{i=1}^N (\tilde{T}^{-1}_{ji}- (\tilde{T}^{-1}_{ji})^N) \tilde{T}_{in}| \\
      &\leq \sum_{i\geq N+1}  C e^{C \lambda^{\frac{1}{\alpha}}}\left|\dfrac{\lambda^2}{(\lambda_i-\lambda_j-\lambda)}\right| \frac{1}{|\lambda_i- \lambda_n- \lambda|}+ \sum_{i=1}^N  \frac{C}{|\lambda_i -\lambda_j|} \frac{\log N}{N^{\alpha-1}} \frac{1}{|\lambda_i- \lambda_n- \lambda|} \\
      &\leq C \sum_{i\geq N+1}\frac{1}{|\lambda_i- \lambda_j|} \frac{1}{|\lambda_i- \lambda_n|}+  C \sum_{i=1}^N  \frac{1}{|\lambda_i -\lambda_j|} \frac{1}{|\lambda_i- \lambda_n|} \frac{\log N}{N^{\alpha-1}} \\
      & \leq C  \sum_{i\geq N+1} \frac{1}{\left(\frac{i}{2}\right)^{\alpha}}  \frac{1}{\left(\frac{i}{2}\right)^{\alpha}}+ C \sum_{i=1}^N \frac{1}{|i- j|^{\alpha}}  \frac{1}{|i- n|^{\alpha}}   \frac{\log N}{N^{\alpha-1}}\\
      &\leq \frac{C}{N^{2\alpha-1}}+ \frac{C\log N}{N^{\alpha-1}},
\end{align*}
where we have used the estimate \eqref{boud:fix:lambda:ij}. Thus we conclude that 
\begin{equation}
   \sum_{i\in \N^*} \tilde{T}^{-1}_{ji} \tilde{T}_{in}= \lim_{N\rightarrow +\infty}  \sum_{1\leq i\leq N} (\tilde{T}^{-1})^N_{ji} \tilde{T}_{in}= \delta_{j, n}.
\end{equation}

Therefore, the equality \eqref{eq:deltaRS} holds implying that $\check{T}^{-1}\circ \check{T}= Id$. Using the same method we can also deduce that 
\begin{equation}\label{eq:TjiTin-1}
   \sum_{i\in \N^*} \tilde{T}_{ji} \tilde{T}^{-1}_{in}= \delta_{j, n}.
\end{equation}
implying that $\check{T}\circ \check{T}^{-1}=  Id$.
 This ends the proof of Proposition \ref{tinverse}\\

\end{proof}

\subsubsection{Uniqueness condition and existence of the feedback.} 
\label{sec:TBB}
We are now in position to solve the uniqueness condition $TB=B$, as given by \eqref{sec2:eqop}. Recall that this condition may not be formulated as is, since $B\notin H$ which may prevent $TB$ from being well-defined \textit{a priori}. To circumvent this technical point, we may exploit the infinite matrices representation of the operators $\check{T}$ and $\check{T}^{-1}$, introduced in the previous subsection, to work on a weaker formulation, analogous to \eqref{eq:TBBeq-0}:
\begin{equation}
\label{eq:TBBeq}
\sum_{n\in \N^\ast} \dfrac{k_n b_n}{\lambda_j-\lambda_n-\lambda}=1, \quad \forall j\in \N^\ast.
\end{equation}
Solving this uniqueness condition has a double impact: it defines the feedback coefficients $k_n$ uniquely, and thereby defines a unique candidate transformation $T$. As we will show in Section \ref{sec-well-posedness}, it is also crucial to prove that $T$ and $K$ do satisfy the operator equation \eqref{eq:operator-identity}, which ensures the equivalence between the closed-loop system and the target system. 

\begin{corollary}\label{cor:existkn}
Let $\lambda\notin \mathcal{N}$. Then
there exists $\{k_n\}_{n\in \N^*} \in \ell^\infty(\NN;\RR)$ such that 
\[
\sum_{n\in \N} \dfrac{k_n b_n}{\lambda_j-\lambda_n-\lambda}=1,
\]
is satisfied for every $j\in \N^*$.
\end{corollary}
\begin{proof}
We have proven that 
\[
\sum_{n\in \N} \dfrac{k_n b_n}{\lambda_j-\lambda_n-\lambda}=1, \quad \forall j\in \N^*,
\]
is equivalent to 
\begin{equation}\label{def:kb:tildeT}
  k_n b_n = \sum_{j\in \N^*} (\tilde{T}^{-1})_{nj}, \quad \forall n\in \N^*  
\end{equation}

Indeed, this is a consequence of Proposition \ref{tinverse}.  First, thanks to Proposition \ref{prop:regT-1}, we know that $\{(\tilde{T}^{-1})_{nj}\}_{j\in \N^*} \in \ell^1(\NN;\RR)$. This in particular implies that the above series on the right-hand side is convergent. Next, by implementing the above defined $\{k_n b_n\}_{n\in \NN^*}$ into the summation and using the identity \eqref{eq:TjiTin-1} we obtain:
\begin{align*}
\sum_{n\in \N} \dfrac{k_n b_n}{\lambda_j-\lambda_n-\lambda}&= \sum_{n\in \N} \tilde T_{j, n} \sum_{k\in \N^*} (\tilde{T}^{-1})_{n, k}\\
& = \sum_{k\in \N^*} \sum_{n\in \N} \tilde T_{j, n}(\tilde{T}^{-1})_{n,k} \\
&= \sum_{k\in \N} \delta_{j, k}= 1.
\end{align*}

 Moreover, by the controllability assumption, there exists $0< c \leq |b_n|  \leq C , \forall n\in \N$. With the fact that $\sup_{n\in \N^*} \| \{(\widetilde{T}^{-1})_{nj}\}_{j\in \N^*} \|_{\ell^1} < C$, we deduce that $\{k_n\}_{n\in \N^*} \in \ell^\infty(\NN^*;\RR)$. 
\end{proof}

From Corollary \ref{cor:existkn} we deduce that there exists $\{k_{n}\}_{n\in\mathbb{N}^*}\in \ell^{\infty}(\NN^*;\RR)$ such that \eqref{eq:TBBeq} holds.  Note that, at this point, we have not yet proved the invertibility of $T$ (in contrast with $\check{T}$ which we know is invertible from Proposition \ref{tinverse}): the properties established on the auxiliary transformation $\check{T}$ suffice to study the uniqueness condition and define the feedback $K$. From \eqref{def-T} we know that with this choice of $\{k_{n}\}_{n\in\mathbb{N}^*}\in \ell^{\infty}(\NN;\RR)$ there exists a unique $T\in\mathcal{L}(H)$, and it satisfies a weak equivalent to the $TB=B$ condition, namely \eqref{eq:TBBeq}. What we have shown is summarized in:
\begin{proposition}
\label{prop:candidate-T}
Let $\{k_{n}\}_{n\in\mathbb{N}^*}$ be given by Corollary \ref{cor:existkn}, then $T\in \mathcal{L}(H)$ given by \eqref{def-T}, that is,
\[T\varphi_n=k_n \sum_{p\in \N} \frac{b_p}{\lambda_p-\lambda-\lambda_n} \varphi_p , \quad \forall n \in \N^*\]
is well defined and satisfies \eqref{eq:TBBeq}.
\end{proposition}
In the following section, we show that this transformation $T$ is invertible, and consequently that the initial system is exponentially stable with decay rate $\lambda$ with the choice of feedback operator $\{k_{n}\}_{n\in\mathbb{N}^*}$. Moreover we will give additional estimates with respect to the dependency in $\lambda$ of $T$.

In this section we turn to the operator $T$. Recall from \eqref{T-tauBK-checkT}--\eqref{def:eq:tauBKTcheck} that $T$ is linked to $\check T$ by
\begin{equation}\label{def:exprT}
 T=    \tau_B \check{T} \tau_K
\end{equation}
and
\begin{equation}
\label{eq:expr-T}
T \varphi_{n} = k_{n} \sum\limits_{j\in\mathbb{N}}b_{j}\widetilde{T}_{j, n}\varphi_{j}.
\end{equation}
This section is organized as follows: first we show some properties on $K$ (Section \ref{sec:propK}). Then, we show that $T$ is invertible thanks to our choice of $K$ using a method similar to \cite{CGM,GLM,Gagnon-Hayat-Xiang-Zhang} (Section \ref{sec:invertT}). Finally, we show precise quantitative estimates on the norm of $T$ and $T^{-1}$, which is the main challenge of this section (Section \ref{sec:quantitative}). These sharp estimates will in turn be used to prove the null controllability of Section \ref{sec-null-control}.

\begin{remark}
    In fact, one can show that $T$ can be written as follows:
    \begin{equation}
        T = \frac{1}{\lambda}\text{Id} + T_{c},
    \end{equation}
    where $T_{c}$ is a compact operator (see \cite{WaterWave,hayat2024fredholm}).
\end{remark}

\subsection{Properties of $K$}
\label{sec:propK}
In this section, we discuss the properties of the feedback law $K$. 
First, we prove that in the present context $K$ cannot be a bounded functional on $H$ (see Lemma \ref{lem:KboundedTcompact}). This might be surprising, given that the solution of the system belongs to $H$. In fact, this is rendered possible by the fact that the construction of $K$ depends on $A$ and consequently $K$ still makes sense when applied to solutions of the system. More precisely, we will prove that $K$ is admissible for the strongly continuous semigroup $(e^{At})_{t\geq 0}$ (see Lemma \ref{lem:Kadmissible}). First, we have the following result, classical for the Fredholm backstepping method.   
\begin{lemma}\label{lem:KboundedTcompact}
Let $\{k_n \}_{n\in \N^*}\in \ell^\infty$. If $K : H \rightarrow \R$ were to be a bounded functional, then $T : H \rightarrow H$ would be a compact operator and therefore not invertible over $H$.
\end{lemma}

This is a consequence of the Riesz representation theorem and of the fact that Hilbert--Schmidt
operators are compact. As a consequence
$K$ is not a bounded functional from $H$ to $\R$, or else it would imply $T$ is not an invertible transformation from $H$ to itself. Despite this ill-behavior of $K$, the feedback operator $K$ is still admissible.
\begin{lemma}\label{lem:Kadmissible}
Let $k_n\in \ell^\infty(\N^
*;\mathbb{C})$. Then $K$ is admissible for the strongly continuous semigroup $(e^{At})_{t\geq 0}$, that is $Ke^{At}f \in L^2((0,T);\CC)$, for all $f\in H$. 
\end{lemma}
This is quite classical (see \cite[Chapter 5]{TucsnakWeissBook}).
Indeed, for \(f=\sum f_n\phi_n\),
\[
\int_0^T |K e^{At}f|^2\,dt
\leq \|k\|_{\ell^\infty}^2 \|f\|_H^2
\int_0^T \sum_{n\geq 1} e^{2\lambda_n t}\,dt
\leq C\|k\|_{\ell^\infty}^2 \|f\|_H^2
\sum_{n\geq 1}\frac1{|\lambda_n|}<\infty,
\]
since \(|\lambda_n|\simeq n^\alpha\) with \(\alpha>1\).

\subsection{Invertibility of $T$}
\label{sec:invertT}
As it can be seen from \eqref{eq:expr-T} the invertibility of $T$ depends on the behavior of $(k_{n})_{n\in\mathbb{N}}$.
Recall that the $(k_{n})_{n\in\mathbb{N}}$ are defined by Corollary \ref{cor:existkn}. We can show the following.
\begin{lemma}\label{lem:invk}
Assume there exist $c,C>0$ such that $c\leq |k_n| \leq C$, for all $n\in \N^*$. Then $T$ is invertible from $H$ to itself.
\end{lemma}
\begin{proof}
Notice
\begin{align*}
  Tf &=\sum_{n\in \NN^*} f_n T \varphi_n  = \sum_{n\in \NN^*} f_n k_n (A-(\lambda_n+\lambda)I)^{-1}B \\
  &= \sum_{n\in \NN^*} f_n k_n \sum_{k\in \NN^*} \dfrac{b_k}{\lambda_k-\lambda_n-\lambda} \varphi_k  
   = \sum_{k\in \NN^*} \sum_{n\in \NN^*} f_n  k_n  \dfrac{b_k}{\lambda_k-\lambda_n-\lambda} \varphi_k \\
   &= \tau_B \check{T} \tau_K f
\end{align*}
where, 
\[
\tau_K : \varphi_n \mapsto k_n \varphi_n, \quad \tau_B : \varphi_n \mapsto b_n \varphi_n,
\]
From the controllability assumption, $\tau_B$ is invertible from $H$ to itself and from the hypothesis of Lemma \ref{lem:invk}, so is $\tau_K$. The invertibility of $T$ then follows from that of $\check{T}$.
\end{proof}

\section{Quantitative estimates for the transformation}
\label{sec:quantitative}
This section is devoted to the sharp quantitative estimates. The coefficients $k_n$ play a fundamental role in the quantitative estimates of $K$ and hence of $T$. In this section, we prove not only that the coefficients $|k_n|$ are bounded from below away from zero, but also their behavior with respect to $\lambda$. This allows, in turn, to deduce precise estimates on the norms of $T$ and $T^{-1}$.   Note that the estimates depend on whether the operator $A$ is self-adjoint or skew-adjoint, with the analysis in the self-adjoint case being more involved.   
\subsection{Exhibiting the underlying isomorphism of $T$}
The expression of $T$ with respect to $K$ is given by \eqref{def:exprT} and the one of $T^{-1}$ is given by, 
\begin{equation}\label{def:exprT-1}
    T^{-1}= \tau_K^{-1} \check T^{-1} \tau_B^{-1}.
\end{equation}

 The norm of $\tau_B$ and $\tau_B^{-1}$ are easily given by the assumption on $b_n$. On the other hand, the norm of $\check T^{-1}$ is given by Lemma \ref{lem:boundT-1} and Proposition \ref{prop:bound}, which provides an estimate with respect to $\lambda$ provided that $\lambda$ is away from $\mathcal{N}$. A similar estimate also holds for $\check{T}$. Therefore, to have a quantitative estimate of $T$ and $T^{-1}$, it suffices to estimate the norm of $\tau_K$ and $\tau_K^{-1}$. This is translated in the following lemma.  

\begin{lemma}\label{lem:pre:boundT-1}
Let $\lambda\notin \mathcal{N}$. Then, 
 the isomorphism $T=T_\lambda $ satisfies
\begin{align*}
    \| T\|&\leq   \| \check T\| \sup_{n} |b_n|  \sup_{n} |k_n|,\\
     \| T^{-1}\|&\leq   \| \check T^{-1}\| \sup_{n} |b_n^{-1}|  \sup_{n} |k_n^{-1}|.
\end{align*}
\end{lemma}

\begin{proof}
First, recall that the bound of $\check T$ and $\check T^{-1}$ was proven in \eqref{eq:boundT-1} in Proposition \ref{prop:bound}.

Then, from the definition of $T$, given by \eqref{def:exprT}, and the definition of $T^{-1}$, given by \eqref{def:exprT-1}, we have 
\[
\|T\| \leq \|\tau_B \| \| \check{T} \| \| \|\tau_K \|, \quad  \|T^{-1}\| \leq \|\tau_K^{-1}\| \| \check{T}^{-1}\| \| \|\tau_B^{-1} \|
\]
and since 
\begin{align*}
\|\tau_B \| \leq \|b_n\|_{\ell^\infty}, \quad & \|\tau_K \| \leq \|k_n\|_{\ell^\infty}, \\
\|\tau_B^{-1}\| \leq \|b_n^{-1}\|_{\ell^\infty}, \quad  & \|\tau_K^{-1} \| \leq \|k_n^{-1}\|_{\ell^\infty},
\end{align*}
the result follows easily.
\end{proof}

{\color{black}

We now provide estimates on $\|\check{T}\|$ and $\|\check{T}^{-1}\|$ with respect to $\lambda$ and $\text{Dist}_\alpha(\lambda)$ (recall the definition \eqref{eq:distalpha}). Note that this framework allows us to obtain these estimates in a range of spaces. As we will see, stronger norms come from more regular spaces in which a common domain can be used to define the feedback and the well-posedness.

We state the following

\begin{lemma}\label{lem-T-check-reg} Assume that  Assumption \ref{assume:structure} holds and  that $\lambda\notin \mathcal{N}$. 
\begin{itemize} 
    \item 
(Case $A$ is self-adjoint)
Then $\check{T}$ is an isomorphism on $D(A^s)$, for $s\in \left(\tfrac{1}{2\alpha}-1,1-\tfrac{1}{2\alpha}\right)$
and we have the following estimates:
\[\begin{aligned}
    \|\check{T}\|_{\textcolor{black}{D(A^{s})}} &\leq {\color{black}C_s} \left(1 + \dfrac{{\color{black}\lambda}}{\Dist_\alpha(\lambda)} \right) \\
    \|\check{T}^{-1}\|_{\textcolor{black}{D(A^{s})}} & \leq  {\color{black}C_s} \left(1 + \dfrac{1}{\Dist_\alpha(\lambda)} \right) e^{C\lambda^{\frac{1}{\alpha}}},
\end{aligned}\]
where the constant $C$ does not depend on $s, \lambda$, and $C_s$ does not depend on $\lambda$.

Accordingly, the isomorphism $T$ satisfies, for $s\in \left(\tfrac{1}{2\alpha}-1,1-\tfrac{1}{2\alpha}\right)$
\[\begin{aligned}
    \|T\|_{\textcolor{black}{D(A^{s})}} & \textcolor{black}{\leq   {\color{black}C_s} \left(1 + \dfrac{{\color{black}\lambda}}{\Dist_\alpha(\lambda)} \right)\sup_{n} |b_n|  \sup_{n} |k_n|}\\
    \|T^{-1}\|_{\textcolor{black}{D(A^{s})}} &\leq   \textcolor{black}{  {\color{black}C_s}\left(1 + \dfrac{1}{\Dist_\alpha(\lambda)} \right) {\color{black}e^{C\lambda^{\frac{1}{\alpha}}}} \sup_{n} |b_n^{-1}|  \sup_{n} |k_n^{-1}|.}
\end{aligned}\]
\item
(Case $A$ is skew-adjoint) 
Then,  $\check{T}$ is an isomorphism on $D(A^s)$, for $s\in \left(\tfrac{1}{2\alpha}-1,1-\tfrac{1}{2\alpha}\right)$
and we have the following estimates:
\textcolor{black}{
\[\begin{aligned}
    \|\check{T}\|_{\textcolor{black}{D(A^{s})}} &\leq C_s \left(1 + \lambda \right) \\
    \|\check{T}^{-1}\|_{\textcolor{black}{D(A^{s})}} & \leq  C_s  e^{C\lambda^{\frac{1}{\alpha}}},
\end{aligned}\]
where the constant $C$ does not depend on $s, \lambda$, and $C_s$ does not depend on $\lambda$.}

Accordingly, the isomorphism $T$ satisfies, for $s\in \left(\tfrac{1}{2\alpha}-1,1-\tfrac{1}{2\alpha}\right)$
\[\begin{aligned}
    \|T\|_{\textcolor{black}{D(A^{s})}} & \textcolor{black}{\leq   C_s \left(1 + \lambda \right)\sup_{n} |b_n|  \sup_{n} |k_n|}\\
    \|T^{-1}\|_{\textcolor{black}{D(A^{s})}} &\leq   \textcolor{black}{  C_s e^{C\lambda^{\frac{1}{\alpha}}} \sup_{n} |b_n^{-1}|  \sup_{n} |k_n^{-1}|.}
\end{aligned}\]
\end{itemize}
\end{lemma}

\begin{proof}
Before showing its proof, we first present two auxiliary results. The first one can be directly deduced from the proof of \cite[Lemma 4.7]{WaterWave}, where a stronger regularizing result is shown.   
\begin{lemma}\label{lem:10}
    Let $r\in (1/2- \alpha, \alpha- 1/2)$. Then there exists some effectively computable $C_r$ such that
   \begin{equation*}
       \sum_{p\in \mathbb{N}^*} p^{2r} \left|\sum_{n\in \mathbb{N}^*\setminus \{p\}} \frac{|a_n n^{-r}|}{|\lambda_n- \lambda_p|} \right|^2 \leq C_{r} \sum_{n\in \mathbb{N}^*} |a_n|^2,
   \end{equation*} 
for every sequence $\{a_n\}_{n\in \mathbb{N}^*}\in l^2(\mathbb{N}^*)$.
\end{lemma}
The second one is a straightforward consequence of the following, 
\begin{align*}
    \frac{|\lambda_j- \lambda_i|}{|\lambda_j- \lambda_i-\lambda|}\leq  \frac{|\lambda_j- \lambda_i- \lambda|+ |\lambda|}{|\lambda_j- \lambda_i-\lambda|}\leq 1+ \frac{|\lambda|}{|\lambda_j- \lambda_i-\lambda|}\leq 1+ \frac{\lambda}{\Dist_{\alpha}(\lambda)}.
\end{align*}

\begin{lemma}\label{lem:11}
Assume that $A$ is self-adjoint. For every $\lambda>0$ and every $i\neq j$, there is 
 \begin{equation*}
     \frac{1}{|\lambda_j- \lambda_i- \lambda|} \leq \left(\frac{\lambda}{\Dist_{\alpha}(\lambda)}+ 1 \right) \frac{1}{|\lambda_j- \lambda_i|}.
 \end{equation*}
Assume that $A$ is self-adjoint. For every $\lambda>0$ and every $i\neq j$, there is 
\begin{align*}
    \frac{|\lambda_j- \lambda_i|}{|\lambda_j- \lambda_i-\lambda|}\leq 1.
\end{align*} 
\end{lemma}

In the following, the constant $C$ may change from line to line, but it is independent of the choice of $\lambda>0$. Let $f= \sum_{n\in \mathbb{N}^*} f_n \varphi_n \in D(A^s)$. Recall from the beginning of the proof of Proposition \ref{prop:bound} that 
\begin{equation*}
    \check T^{-1} f=  \sum_{k\in \N^*} f_k (\tilde{T}^{-1})_{k, k} \varphi_k + \sum_{k\in \N^*} \left( \sum_{n\in \N^* \setminus\{k\}} f_n (\tilde{T}^{-1})_{k, n} \right) \varphi_k,
\end{equation*}
we also recall the bounds on $|(\tilde T^{-1})_{i, j}|$  from \eqref{eq:firstboundT-1}.
First notice that 
\begin{equation*}
    \|\sum_{k\in \N^*} f_k (\tilde{T}^{-1})_{k, k} \varphi_k\|_{D(A^s)}^2\leq  C e^{C \lambda^{1/\alpha}} \|f\|_{D(A^s)}^2.
\end{equation*}
Next,  assuming that $A$ is self-adjoint we obtain 
\begin{align*}
 &\left\|\sum_{k\in \N^*} \left( \sum_{n\in \N^* \setminus\{k\}} f_n (\tilde{T}^{-1})_{k, n} \right) \varphi_k\right\|_{D(A^s)}^2 
 \leq   C  \left\|\sum_{k\in \N^*} \left( \sum_{n\in \N^* \setminus\{k\}} f_n (\tilde{T}^{-1})_{k, n} \right) \varphi_k\right\|_{\mathcal{H}^{s \alpha}}^2 \\
 & \leq C \sum_{k\in \N^*} k^{2 s\alpha} \left| \sum_{n\in \N^* \setminus\{k\}} |f_n| |(\tilde{T}^{-1})_{k, n}| \right|^2 \\
 & \leq C e^{C \lambda^{1/\alpha}} \sum_{k\in \N^*} k^{2 s\alpha} \left| \sum_{n\in \N^* \setminus\{k\}} |f_n| \frac{1}{|\lambda_n- \lambda_k- \lambda|}\right|^2, \;\; \; \textrm{(used \eqref{eq:firstboundT-1})} \\
  & \leq C e^{C \lambda^{1/\alpha}}  \left(\frac{\lambda}{\Dist_{\alpha}(\lambda)}+ 1 \right)^2 \sum_{k\in \N^*} k^{2 s\alpha} \left| \sum_{n\in \N^* \setminus\{k\}} |f_n| \frac{1}{|\lambda_n- \lambda_k|}\right|^2, \;\; \; \textrm{(used Lemma \ref{lem:11})}  \\
   & \leq C_s e^{C \lambda^{1/\alpha}}  \left(\frac{1}{\Dist_{\alpha}(\lambda)}+ 1 \right)^2 \sum_{n\in \N^*} |f_n n^{s\alpha}|^2, \;\; \; \textrm{(used Lemma \ref{lem:10})}   \\
   & \leq C_s e^{C \lambda^{1/\alpha}} \left(\frac{1}{\Dist_{\alpha}(\lambda)}+ 1 \right)^2 \|f\|_{D(A^s)}^2.
\end{align*}
By combining the above estimates, we conclude the bound of $\|\check T^{-1}\|_{D(A^s)}$.  Estimating $\|\check T\|_{D(A^s)}$  and considering the case $A$ is skew-adjoint are even simpler, we omit them.
\end{proof}

}

\subsection{Cost of the backstepping transformation with respect to the damping parameter $\lambda$}\label{sec:42}
In this section, we prove one of the main results of this article 
\begin{theorem}
\label{cor:boundTT-1} 
 Under Assumption \ref{assume:structure}, the transformation $T$ constructed
 in Section \ref{sec:def:TK} 
satisfies the following: 
\begin{itemize}
 \item (Case $A$ is self-adjoint)
For any $N\in\mathbb{N}$, there exists $\lambda\in[N,N+1]$ such that the isomorphism $T= T(\lambda)$ satisfies the following estimate 
\begin{equation}\label{T-Tinv-norm-estimate}
  \|T\|_{{\color{black}\mathcal{L}(D(A^s))}}+  \|T^{-1}\|_{{\color{black}\mathcal{L}(D(A^s))}}\leq {\color{black}C_s} e^{C\lambda^{\frac{1}{\alpha}}}, \;  {\color{black}\forall s\in \left(\tfrac{1}{2\alpha}-1,1-\tfrac{1}{2\alpha}\right),}
\end{equation} 
where $C$ is a constant that does not depend on $(s, N, \lambda)$,  {\color{black} and $C_s$ is a constant that does not depend on $(N, \lambda)$. }
\item  (Case $A$ is skew-adjoint)
For any $\lambda\geq 1$ the isomorphism $T= T(\lambda)$ satisfies the following estimate 
\begin{equation}\label{T-Tinv-norm-estimate-bis}
  \|T\|_{{\color{black}\mathcal{L}(D(A^s))}}+  \|T^{-1}\|_{{\color{black}\mathcal{L}(D(A^s))}}\leq {\color{black}C_s} e^{C\lambda^{\frac{1}{\alpha}}}, \;  {\color{black}\forall s\in \left(\tfrac{1}{2\alpha}-1,1-\tfrac{1}{2\alpha}\right),}
\end{equation} 
where $C$ is a constant that does not depend on $(s, N, \lambda)$,  {\color{black} and $C_s$ is a constant that does not depend on $(N, \lambda)$. }
\end{itemize}
\end{theorem}
Heuristically, for the case that $A$ is self-adjoint, from Lemma \ref{lem:pre:boundT-1}, Lemma \ref{lem-T-check-reg} 
we have

\[\begin{aligned}
    \|T\|_{\textcolor{black}{D(A^{s})}} & \textcolor{black}{\leq   {\color{black}C_s} \left(1 + \dfrac{{\color{black}\lambda}}{\Dist_\alpha(\lambda)} \right)  \sup_{n} |k_n {\color{black}b_n} |},\\
    \|T^{-1}\|_{\textcolor{black}{D(A^{s})}} &\leq   C_s\left(1 + \dfrac{1}{\Dist_\alpha(\lambda)} \right) {\color{black}e^{C\lambda^{\frac{1}{\alpha}}}}   \sup_{n} |k_n{\color{black}b_n|^{-1}}.
\end{aligned}\]

The quantity $\sup|k_{n}|$ is easy to estimate (see Lemma \ref{lem:rearangement} and Proposition \ref{growthentire} below). Estimating $\sup|1/k_{n}|$, 
however, is much more subtle. \\
Even simpler,  for the case that  $A$ is skew-adjoint one has 
\[\begin{aligned}
    \|T\|_{\textcolor{black}{D(A^{s})}} & \textcolor{black}{\leq   {\color{black}C_s} \left(1 + \lambda \right)\sup_{n} |k_n b_n|},\\
    \|T^{-1}\|_{\textcolor{black}{D(A^{s})}} &\leq   \textcolor{black}{  {\color{black}C_s}{\color{black}e^{C\lambda^{\frac{1}{\alpha}}}} \sup_{n} |k_n b_n|^{-1}.}
\end{aligned}\]

The object of most of this subsection is to prove the following theorem.

\begin{theorem}
\label{th-knbn}
Assume that Assumption \ref{assume:structure} holds, the transformation $T$ constructed in Section \ref{sec:def:TK} 
satisfies the following: 
\begin{itemize}
    \item (Case $A$ is self-adjoint). For any $N\in\mathbb{N}$, there exists $\lambda\in[N,N+C]$, where $C>0$ is a constant independent of $N$, such that the following quantitative lower bound of $\{b_n k_n \}_{n\in\mathbb{N}}$
with respect to $\lambda$ holds
\begin{equation}
    |k_{n}b_{n}|\geq e^{-c\lambda^{\frac{1}{\alpha}}}, \quad \forall\; n\in\mathbb{N}^*,
\end{equation}
 where $c>0$ is a constant independent of $N$, and $\lambda$.
  \item(Case $A$ is skew-adjoint). For any $\lambda\geq 1$, one has $|k_n b_n|\geq \lambda$.  
 \end{itemize}

\end{theorem}

We can conclude the proof of Theorem \ref{cor:boundTT-1} by combining the estimates above, this previous theorem and from Proposition \ref{prop:estimate-dist-mu-n} below, deriving estimates of the form 
\[
\text{Dist}_\alpha(\lambda) \geq e^{-c\lambda^{\frac{1}{\alpha}}}, \;\;\;\;  \textrm{ if $A$ is self-adjoint}
\] 
for values of $\lambda $ in $[N,N+C]$ and for any $N\in \N$.

The proof of Theorem \ref{th-knbn} is based on four steps: 
\begin{enumerate}
    \item Decompose $k_{n}b_{n}$ in two terms, $F_{n}$ and $J_{n}$ that have different behavior (see Lemma \ref{lem:rearangement}).
    \item Show that $J_{n}(\lambda) = 1$ for any $n\in\mathbb{N}$ (see Lemma \ref{lem-Jn-1}). This is the object of Section \ref{sec:estimJn}. 
    \item Use careful estimations to estimate $F_{n}(\lambda)$ as a function of $\Dist_{\alpha}(\lambda)$ and $\lambda$ (see Proposition \ref{cor-main2}). This is the object of Section \ref{sec:estimFn}.
    \item Show that for good choice of $\lambda$ one has a good lower bound of $ \Dist_{\alpha}(\lambda)$ (see Proposition \ref{prop:estimate-dist-mu-n}). This is done in Section \ref{sec:estimDist}.
\end{enumerate}

\subsection{Rearrangement of $k_nb_n$}\label{sec:sec:reaknbn}

We begin by a rearrangement of the expression $k_nb_n$ to decompose it in two parts. 
\begin{lemma}
\label{lem:rearangement}
The following holds for any $n\in\mathbb{N}$
\begin{equation}\label{kn-bn-decomposed}
    k_{n}b_{n} = -\lambda F_{n}(\lambda)J_{n}(\lambda),
\end{equation}
where
\begin{gather}
 F_n(\lambda):=\prod_{\substack{m\in \N^\ast\\ m \neq n}}\left(1+\dfrac{\lambda}{\lambda_n-\lambda_m}\right),
  \label{eq:DefFn} \\
 \label{def_Jn}
     J_n(\lambda):=\sum_{j\in \N}  \dfrac{\displaystyle   \prod_{\substack{m\in \N^\ast\\ m \neq n}} (\lambda_j-\lambda_m-\lambda)}{ \displaystyle  \prod_{\substack{m\in \N^\ast\\ m \neq j}} (\lambda_j-\lambda_m) }. 
\end{gather}
\end{lemma}
\begin{proof}
By definition, we know from \eqref{def:kb:tildeT}  that 
\begin{align*} 
k_nb_n & = \sum_{j\in \N^*} \tilde{T}^{-1}_{nj} \\
&= \sum_{j\in  \N^*} \dfrac{\lambda^2}{\lambda_j-\lambda_n-\lambda} \prod_{\substack{m\in \N^\ast\\ m \neq n}}\left(1+\dfrac{\lambda}{\lambda_n-\lambda_m}\right) \prod_{\substack{m\in \N^\ast\\ m \neq j}} \left( 1 - \dfrac{\lambda}{\lambda_j-\lambda_m}\right)\\
&= \left(\prod_{\substack{m\in \N^\ast\\ m \neq n}}\left(1+\dfrac{\lambda}{\lambda_n-\lambda_m}\right) \right)  \left(\sum_{j\in  \N^*} \dfrac{\lambda^2}{\lambda_j-\lambda_n-\lambda}  \prod_{\substack{m\in \N^\ast\\ m \neq j}} \left( 1 - \dfrac{\lambda}{\lambda_j-\lambda_m}\right) \right)\\
&= -\lambda \left( \prod_{\substack{m\in \N^\ast\\ m \neq n}}\left(1+\dfrac{\lambda}{\lambda_n-\lambda_m}\right)\right) \left(\sum_{j\in  \N^*} \dfrac{-\lambda}{\lambda_j-\lambda_n-\lambda} \prod_{\substack{m\in \N^\ast\\ m \neq j}} \left( 1 - \dfrac{\lambda}{\lambda_j-\lambda_m}\right)\right)\\
&= -\lambda \left( \prod_{\substack{m\in \N^\ast\\ m \neq n}}\left(1+\dfrac{\lambda}{\lambda_n-\lambda_m}\right)\right) 
\sum_{j\in \N}  \dfrac{\displaystyle   \prod_{\substack{m\in \N^\ast\\ m \neq n}} (\lambda_j-\lambda_m-\lambda)}{ \displaystyle  \prod_{\substack{m\in \N^\ast\\ m \neq j}} (\lambda_j-\lambda_m) } \\
&=- \lambda F_n(\lambda)J_n(\lambda).
\end{align*}
where we used that
\begin{align}
   \sum_{j\in  \N^*} \dfrac{-\lambda}{\lambda_j-\lambda_n-\lambda} \prod_{\substack{m\in \N^\ast\\ m \neq j}} \left( 1 - \dfrac{\lambda}{\lambda_j-\lambda_m}\right) 
    &=  \sum_{j\in  \N^*} \dfrac{-\lambda}{\lambda_j-\lambda_n-\lambda}  \left(\dfrac{ \prod_{\substack{m\in \N^\ast\\ m \neq j}} (\lambda_j- \lambda_m- \lambda) }{ \prod_{\substack{m\in \N^\ast\\ m \neq j}} (\lambda_j- \lambda_m) } \right)  \notag  \\
    &=   \sum_{j\in  \N^*} \dfrac{1}{\lambda_j-\lambda_n-\lambda}  \left(\dfrac{ \prod_{\substack{m\in \N^\ast}} (\lambda_j- \lambda_m- \lambda) }{ \prod_{\substack{m\in \N^\ast\\ m \neq j}} (\lambda_j- \lambda_m) } \right)  \notag  \\
     &=   \sum_{j\in  \N^*} \left(\dfrac{1}{\prod_{\substack{m\in \N^\ast\\ m \neq j}} (\lambda_j- \lambda_m) } \right) \left(  \dfrac{ \prod_{\substack{m\in \N^\ast}} (\lambda_j- \lambda_m- \lambda) }{  \lambda_j-\lambda_n-\lambda } \right)  \notag  \\
    &= \sum_{j\in \N}  \dfrac{\displaystyle   \prod_{\substack{m\in \N^\ast\\ m \neq n}} (\lambda_j-\lambda_m-\lambda)}{ \displaystyle  \prod_{\substack{m\in \N^\ast\\ m \neq j}} (\lambda_j-\lambda_m) } \label{eq:formJn}
\end{align}
\end{proof}

One can note that $F_n$ is a holomorphic function of $\lambda \in \mathbb{C}$, with a countable number of zeros $\lambda=\lambda_m-\lambda_n, \, m\in \N\setminus\{n\}$ and such that $F_n(0)=1$. In the following we study the behavior of $F_{n}$ and $J_{n}$. In particular, we will show a remarkable property: $J_n(\lambda)=1$, for any $\lambda\in \mathbb{C}$ and any $n\in \NN^*$(see Lemma \ref{lem-Jn-1}). We will also show that there exists $C,c>0$ such that 
\begin{equation}
|F_{n}(\lambda)|\geq C\Dist_{\alpha}(\lambda) e^{-c\lambda^{\frac{1}{\alpha}}},\quad \forall\;n\in\mathbb{N}^{*},\forall \;\lambda>1,
\end{equation}
where $\Dist_{\alpha}(\lambda)$ is given by \eqref{eq:distalpha}. Note that from Lemma \ref{lem:sumlambdaijdistbound} we also have
\begin{equation}
    |F_{n}(\lambda)|\leq C e^{C \lambda^{\frac{1}{\alpha}}}.
\end{equation}

\subsection{Behavior of $J_{n}(\lambda)$.}
\label{sec:estimJn}
We begin by drastically simplifying one of the terms introduced above.
\begin{lemma}
\label{lem-Jn-1}
For $\lambda\in \mathbb{C}$ and $n\in \N^*$, we have
\[
J_n(\lambda)=1.
\]
\end{lemma}

\begin{proof}[of Lemma \ref{lem-Jn-1}] 

The proof is composed by two steps:
\begin{itemize}
    \item[Step A:] Show that a truncated version $J^{N}_{n}(\lambda)$ of $J_n(\lambda)$ is equal to $1$ for any $\lambda\in \CC$, any $N\in \mathbb{N}^*$ and any $1\leq n \leq N$.
    \item[Step B:] Pass the limit in $N$ thanks to the distribution of $\{\lambda_n\}_{n\in \NN^*}$, and leveraging that $\lambda_n\sim n^{\alpha}$ with $\alpha>1$.
\end{itemize}

 In the following we only consider the case that $A$ is self-adjoint, since the skew-adjoint case can be treated similarly.
Before starting the proof of the first step, let us define the following polynomial
\begin{equation}
    P(\lambda):= \sum_{j=1}^{N-1} 1\cdot \left(\prod_{\substack{1\leq m\leq N-1 \\ m\neq j}}\frac{\lambda- \lambda_m}{\lambda_j- \lambda_m} \right),
\end{equation}
which is of order at most $N-2$. This is actually the Lagrange interpolation formula such that $P(\lambda_p)=1, 1\leq p\leq N-1$. 
We get, with this definition,
\begin{equation}
    \sum_{j=1}^N \dfrac{\prod_{m\neq N}(\lambda_{N}- \lambda_m)}{\displaystyle  \prod_{\substack{m=1 \\ m \neq j}}^N (\lambda_j-\lambda_m)}=  -\sum_{j=1}^{N-1} \dfrac{\displaystyle  \prod_{\substack{m=1 \\ m \neq j}}^{N-1} (\lambda_N-\lambda_m)}{\displaystyle  \prod_{\substack{m=1 \\ m \neq j}}^{N-1} (\lambda_j-\lambda_m)}+ 1=1- P(\lambda_N)=0. 
\end{equation}

{\bf Step A: }Let $\{\zeta_1, \zeta_2,..., \zeta_N\}$ be a given sequence of constants satisfying $\zeta_i\neq \zeta_j$.
Inspired by the definition of $J_n$ in equation \eqref{def_Jn} we further define, for any $N\in \mathbb{N}^*$ and any $1\leq n\leq N$:
\begin{equation}
\label{eq:defJnN}
J^N_n(\lambda):=J^N_n(\lambda; \zeta_1, \zeta_2,..., \zeta_N):= \sum_{j=1}^N  \dfrac{\displaystyle   \prod_{\substack{m\leq N \\ m \neq n}} (\zeta_j-\zeta_m-\lambda)}{ \displaystyle  \prod_{\substack{m\leq N \\ m \neq j}} (\zeta_j-\zeta_m) }.
\end{equation}
In the following we will prove that 
\begin{equation}
\label{eq:JnN1}
    J^N_n(\lambda; \zeta_1, \zeta_2,..., \zeta_N)=1, \quad \forall \lambda\in \mathbb{C}, \forall n\in \{1, 2,..., N\}.
\end{equation}
To do so, we immediately notice the following symmetry
\begin{equation}
    J^N_{N-1}(\lambda; \zeta_1, \zeta_2,...,\zeta_{N-1}, \zeta_N)=  J^N_{N}(\lambda; \zeta_1, \zeta_2,...,\zeta_{N-2},\zeta_{N}, \zeta_{N-1}).
\end{equation}

We further observe that 
\begin{equation}
J^N_N(\lambda)= \sum_{j=1}^{N-1}  \left(\dfrac{\displaystyle   \prod_{\substack{m\leq N-1}} (\zeta_j-\zeta_m-\lambda)}{ \displaystyle  \prod_{\substack{m\leq N \\ m \neq j}} (\zeta_j-\zeta_m) }\right)+ \dfrac{\displaystyle   \prod_{\substack{m\leq N-1}} (\zeta_N-\zeta_m-\lambda)}{ \displaystyle  \prod_{\substack{m\leq N-1}} (\zeta_N-\zeta_m) }.
\end{equation}

By multiplying $J^N_N(\lambda)$ with $\prod_{\substack{m\leq N-1}} (\zeta_N-\zeta_m)$ we shall prove that  
\begin{align*}
    P^N_N(\lambda):=& \sum_{j=1}^{N-1}  \left(\dfrac{\displaystyle   \prod_{\substack{m\leq N-1}} (\zeta_j-\zeta_m-\lambda) \prod_{\substack{m\leq N-1}} (\zeta_N-\zeta_m)}{ \displaystyle  \prod_{\substack{m\leq N \\ m \neq j}} (\zeta_j-\zeta_m) }\right) \\
    &+ \prod_{\substack{m\leq N-1}} (\zeta_N-\zeta_m-\lambda)-\prod_{\substack{m\leq N-1}} (\zeta_N-\zeta_m),
\end{align*}
equals to zero, which is equivalent to prove that 
\begin{align*}
   P^N_N(\lambda)=&  \sum_{j=1}^{N-1}  \left(\dfrac{\displaystyle \lambda  \prod_{\substack{m\leq N-1\\ m \neq j}} (\zeta_j-\zeta_m-\lambda) \prod_{\substack{m\leq N-1\\ m \neq j}} (\zeta_N-\zeta_m)}{ \displaystyle  \prod_{\substack{m\leq N-1 \\ m \neq j}} (\zeta_j-\zeta_m) }\right) \\
   &+ \prod_{\substack{m\leq N-1}} (\zeta_N-\zeta_m-\lambda)-\prod_{\substack{m\leq N-1}} (\zeta_N-\zeta_m)=0.
\end{align*}
Inspired by the preceding formula, for any given $\{\zeta_1,..., \zeta_{N-1}, \lambda\}$ satisfying $\zeta_i\neq \zeta_j$ we define 
\begin{align}
    Q(\mu; \zeta_1,...,\zeta_{N-1},\lambda):= &\sum_{j=1}^{N-1}  \left(\dfrac{\displaystyle \lambda  \prod_{\substack{m\leq N-1\\ m \neq j}} (\zeta_j-\zeta_m-\lambda) \prod_{\substack{m\leq N-1\\ m \neq j}} (\mu-\zeta_m)}{ \displaystyle  \prod_{\substack{m\leq N-1 \\ m \neq j}} (\zeta_j-\zeta_m) }\right) \notag\\
   & + \prod_{\substack{m\leq N-1}} (\mu-\zeta_m-\lambda)-\prod_{\substack{m\leq N-1}} (\mu-\zeta_m),
\end{align}
which is clearly a polynomial of $\mu$ of order $N-2$. Now we show that $Q(\mu)=0$. Indeed, by choosing $\mu= \zeta_1$ we get that 
\begin{align*}
    Q(\zeta_1)&=\sum_{j=1}  \left(\dfrac{\displaystyle \lambda  \prod_{\substack{m\leq N-1\\ m \neq j}} (\zeta_j-\zeta_m-\lambda) \prod_{\substack{m\leq N-1\\ m \neq j}} (\zeta_1-\zeta_m)}{ \displaystyle  \prod_{\substack{m\leq N-1 \\ m \neq j}} (\zeta_j-\zeta_m) }\right) \notag\\
   & + \prod_{\substack{m\leq N-1}} (\zeta_1-\zeta_m-\lambda)-\prod_{\substack{m\leq N-1}} (\zeta_1-\zeta_m),\\
   &= \lambda  \prod_{\substack{m\leq N-1\\ m \neq 1}}(\zeta_1-\zeta_m-\lambda)+\prod_{\substack{m\leq N-1}} (\zeta_1-\zeta_m-\lambda)\\
   &=0.
\end{align*}
Thanks to the symmetry we obtain that \begin{equation}
    Q(\zeta_i)=0, \forall i\in \{1, 2,..., N-1\},
\end{equation}
which, to be combined with the fact that $Q(\mu)$ is a polynomial of order $N-2$, yields $Q(\mu)=0$. In particular, by choosing $\mu=\zeta_N$ we get 
\begin{equation}
   0= Q(\zeta_N)= P^N_N(\lambda).
\end{equation}
Hence, for any sequence $\{\zeta_1,..., \zeta_N\}$ satisyfing $\zeta_i\neq \zeta_j$ and for  any $\lambda$ we know that $P^N_N(\lambda)=0$. This concludes the proof of Step 1 and \eqref{eq:JnN1}.

{\bf Step B: } This step is standard and consists in passing to the limit in \eqref{eq:JnN1}. 
We can apply, for instance, a dominated convergence by denoting
\begin{equation}
\label{eq:defIiN}
    I_{j}^{N} := \frac{\lambda}{\lambda_{j}-\lambda_{n}-\lambda}\prod_{\substack{m\leq N\\ m \neq j}}\left(1-\frac{\lambda}{\lambda_{j}-\lambda_{m}}\right),\;\;    I_{j} := \frac{\lambda}{\lambda_{j}-\lambda_{n}-\lambda}\prod_{\substack{m\in\mathbb{N}^{*}\\ m \neq j}}\left(1-\frac{\lambda}{\lambda_{j}-\lambda_{m}}\right).
\end{equation}
We observe that for $j\leq N$
\begin{equation}
\label{eq:ratioIj}
    \frac{I_{j}^{N}}{I_{j}} = \prod_{\substack{m> N\\ m \neq j}}\left(1-\frac{\lambda}{\lambda_{j}-\lambda_{m}}\right)^{-1}.
\end{equation}
Note that we have, 
for $j$ given, and when $m\rightarrow +\infty$
\begin{equation}
     \log\left(1-\frac{\lambda}{\lambda_{j}-\lambda_{m}}\right)\sim -\frac{\lambda}{C|m-j|^{\alpha}}
\end{equation}
and therefore
\begin{equation}
    \sum\limits_{m> N}\log\left(1-\frac{\lambda}{\lambda_{j}-\lambda_{m}}\right) \rightarrow 0, \quad {\color{black} N\rightarrow \infty.}
\end{equation}
This and \eqref{eq:ratioIj} imply $I_{j}^{N}\rightarrow I_{j}$ when $N\rightarrow+\infty$ for any $j\in\mathbb{N}^{*}$. Looking again at \eqref{eq:ratioIj} and using \eqref{eq:cond3}, for  $N>(\lambda/c)^{1/(\alpha-1)}-1$, where $c$ is given by \eqref{eq:cond3},
\begin{align}
    |I_{j}^{N}\mathbbm{1}_{j\leq N}| & = 
    |I_{j}\mathbbm{1}_{j\leq N}|\prod_{\substack{m> N\\ m \neq j}}\left(1-\frac{\lambda}{\lambda_m-\lambda_j}\right)^{-1} \nonumber \\
    &\leq |I_{j}\mathbbm{1}_{j\leq N}|\prod_{\substack{m> N\\ m \neq j}}\left(1-\frac{\lambda}{c|m-j|m^{\alpha-1}}\right)^{-1}\nonumber \\
    &\leq |I_{j}|\prod_{m> N}\left(1-\frac{\lambda}{c|m-N|m^{\alpha-1}}\right)^{-1},
\end{align}
where we used the fact that if $j>N$ then $\mathbbm{1}_{j\leq N}$ vanishes. This gives, 
\begin{equation}
\label{eq:hypdominated}
    |I_{j}^{N}\mathbbm{1}_{j\leq N}| \leq |I_{j}|\prod_{k> 0}\left(1-\frac{\lambda}{c|k|(k+N)^{\alpha-1}}\right)^{-1}\leq {\color{black}|I_{j}|} \prod_{k>0}\left(1-\frac{1}{{\color{black}c}k^{\alpha}}\right)^{-1}.
\end{equation}
Finally, we can show that the right-hand side is integrable: from the proof of Lemma \ref{lem:boundT-1} and Lemma \ref{lem:sumlambdaijdistbound}, we have
\begin{equation}
 \sum\limits_{j\in\mathbb{N}^{*}}|I_{j}|= \sum\limits_{j\in\mathbb{N}^{*}}\left|\frac{\lambda}{\lambda_{j}-\lambda_{n}-\lambda}\right|\left|\prod_{\substack{m\in\mathbb{N}^{*}\\ m \neq j}}\left(1-\frac{\lambda}{\lambda_{j}-\lambda_{m}}\right)\right|<+\infty,
\end{equation}
and 
\begin{equation}
    \prod_{k>0}\left(1-\frac{1}{{\color{black}c}k^{\alpha}}\right)^{-1}=\exp\left(\sum_{k>0} \ln \left( 1+\dfrac{1}{ck^\alpha-1} \right)\right)<+\infty,
\end{equation}
thus from \eqref{def_Jn} and \eqref{eq:formJn}, \eqref{eq:defJnN}, \eqref{eq:defIiN} and \eqref{eq:hypdominated} we conclude by dominated convergence that
\begin{equation}
J_{n}^{N}(\lambda) = -\sum\limits_{j\in\mathbb{N}^{*}} I_{j}^{N}\mathbbm{1}_{j\leq N} \rightarrow     -\sum\limits_{j\in\mathbb{N}^{*}} I_{j} = J_{n}(\lambda). \notag
\end{equation}
\end{proof}

\subsection{Behavior of $F_{n}$}
\label{sec:estimFn}

Now, we turn to the other term of decomposition \eqref{kn-bn-decomposed}. Note that we have the following proposition as a direct consequence of Lemma \ref{lem:boundT-1}.
\begin{proposition}[Upper bound]\label{growthentire}
Assume that $A$ is either self-adjoint or skew-adjoint.  Let $\lambda >0$. Then, there exists $C>0$ independent of $\lambda$ such that,
\[
|F_n(\lambda)| \leq Ce^{C\lambda^{\frac{1}{\alpha}}}.
\]
\end{proposition}

Let us recall the definition of $\Dist_{\alpha}(\lambda)$, 
\begin{equation*}
      \Dist_{\alpha}(\lambda)= \min \{|\lambda_n- \lambda_m-\lambda|: m, n\in \N^*\}.
\end{equation*}
Note that this is indeed a minimum rather than an infinimum from \eqref{eq:cond3}.
We show the following lower bound on $|F_{n}(\lambda)|$, as a function of $\Dist_{\alpha}(\lambda)$.
\begin{proposition}[Lower bound]
\label{cor-main2}
The quantity $F_n(\lambda)$ satisfies the following estimate.
\begin{itemize}
    \item Assume that $A$ is self-adjoint.  There exist positive constants $c$, $C$ such that, for any $\alpha>1$, $n\in\mathbb{N}^{*}$ and any $\lambda>1$,
\begin{equation}
    |F_n(\lambda)|\geq \Dist_{\alpha}(\lambda) \cdot C e^{-c \lambda^{\frac{1}{\alpha}}}, \quad \forall n\in \N^*, \forall \lambda>1.
\end{equation}
   \item  Assume that $A$ is skew-adjoint. Then  $|F_n(\lambda)|\geq 1$. 
\end{itemize}

\end{proposition}
\begin{proof}[of Proposition \ref{cor-main2}]
First notice that the case $A$ is skew-adjoint is trivial due to the explicit formulation of $F_n(\lambda)$. In the following, we concentrate on the case that $A$ is self-adjoint.  Recall that, thanks to \eqref{eq:cond2},

 for any $n>m\geq 1$ there exist positive constants $C,c$ such that,
\begin{align}
\label{eq:condestim2}
    C(n-m)n^{\alpha-1}>|\lambda_n- \lambda_m|&
    >c(n- m) n^{\alpha- 1}.
\end{align}
Similarly, still from \eqref{eq:cond2}, 

\begin{equation}
\label{eq:cond20}
n^{\alpha}\lesssim |\lambda_{n}|+1\lesssim n^{\alpha}
, \; \forall n\in \mathbb{N}^*.
\end{equation}

We further have,
\begin{lemma}\label{lem-sum-alpha} 
There exists some effectively computable constant $C>0$ such that,
\begin{equation}
    \prod_{m\in\mathbb{N}^{*}} \left(1+ \frac{\lambda}{|\lambda_m|}\right)\lesssim e^{C \lambda^{\frac{1}{\alpha}}}, \; \forall \lambda>1.
\end{equation}

\end{lemma}
\begin{proof}
Indeed, thanks to the estimate \eqref{eq:cond20}, it is equivalent to show the following,
\begin{equation}
    \prod_{m\in\mathbb{N}^{*}} \left(1+ \frac{\lambda}{m^{\alpha}}\right)\lesssim e^{C \lambda^{\frac{1}{\alpha}}}, \; \forall \lambda>1.
\end{equation}
The proof is similar to the proof of Lemma \ref{lem:boundT-1}. Indeed, let 
\begin{equation}
\begin{split}
    \ln\left(\prod_{m\in\mathbb{N}^{*}\setminus\{n\}} \left(1+ \frac{\lambda}{m^{\alpha}}\right)\right)&=\sum\limits_{m\in\mathbb{N}^{*}\setminus\{n\}}\ln\left(1+\frac{\lambda}{m^{\alpha}}\right)\\
    &\leq \log(1+ \lambda)+ \int_1^{+\infty} \log \left(1+\dfrac{\lambda}{x^{\alpha}}\right) dx \\
    &\leq \log(1+ \lambda)+ x\log\left(1+\dfrac{\lambda}{x^{\alpha}}\right)\Big|_{1}^{+\infty}- \int_1^{+\infty} x d \log  \left(1+\dfrac{\lambda}{x^{\alpha}}\right) \\
    &= 2 \int_1^{+\infty} \frac{\lambda}{\lambda+ x^{\alpha}} dx 
     \leq C \lambda^{\frac{1}{\alpha}}.
    \end{split}
\end{equation}

\end{proof}

To continue the proof of Proposition \ref{cor-main2}, we are going to proceed in four steps.\\

\textbf{Step 1: show that, for $F_n(\lambda)$, only the terms $m\in [1, n-1]$ matter.}

Notice that
\begin{equation}
    |F_n(\lambda)|= \prod_{m< n} \left|1+ \frac{\lambda}{\lambda_{n}- \lambda_{m}}\right|\cdot \prod_{m> n} \left|1+ \frac{\lambda}{\lambda_{n}- \lambda_{m}}\right|>\prod_{m< n} \left|1+ \frac{\lambda}{\lambda_{n}- \lambda_{m}}\right|,
\end{equation}
since $\lambda_{m}-\lambda_{n}\leq 0$ for $m>n$ (recall that the eigenvalues are negative). In fact $\text{Re}(\lambda_{m}-\lambda_{n})\leq 0$ would suffice. Thus, to show Proposition \ref{cor-main2} it suffices to show that,
\begin{equation}
\label{ine-key-2}
     \tilde F_n(\lambda):= \left|\prod_{m< n} \left(1+ \frac{\lambda}{\lambda_{n}- \lambda_{m}}\right)\right|\gtrsim \Dist_{\alpha}(\lambda)e^{-C \lambda^{\frac{1}{\alpha}}}.
\end{equation}
In fact, this is even an equivalence since, from Lemma \ref{lem:boundT-1},
\begin{equation*}
     \prod_{m> n} \left|1+ \frac{\lambda}{\lambda_{n}- \lambda_{m}}\right|\lesssim e^{C \lambda^{\frac{1}{\alpha}}}.
\end{equation*}

\textbf{Step 2: case 
$\lambda\in(1,  c_{0}|\lambda_{n}-\lambda_{n-1}|)$ with $c_{0}<1$ and independent from $n$.
}\\

Let's assume that $n>m$. Since $\lambda\leq c_{0}|\lambda_{n}-\lambda_{n-1}|$ and $(\lambda_{n})_{n\in\mathbb{N}^{*}}$ are real,
\begin{equation}
\label{eq:estimstep20}
    \frac{\lambda}{\lambda_{m}- \lambda_{n}}\leq 
    \frac{\lambda}{\lambda_{n-1}- \lambda_{n}}
    \leq c_{0}.
\end{equation}
Now, observe that since $c_{0}<1$, there exists $k_{c_{0}}>0$ depending only on this $c_{0}$ such that 
\begin{equation}
\label{eq:defkc}
(1-x)(1+k_{c_{0}}x) \geq 1,\;\forall x\in \left[0,c_{0}\right].
\end{equation}
One can see this since this function is concave for $k_{c_{0}}>0$ and by looking at its values at $0$ and $c_{0}$.
This, together with \eqref{eq:estimstep20}, implies
\begin{equation}
\label{defkalpha}
    \left(1-\frac{\lambda}{\lambda_{m}- \lambda_{n}}\right)\left(1+\frac{k_{c_{0}}\lambda}{\lambda_{m}- \lambda_{n}}\right)\geq 1,
\end{equation}
and therefore,
\begin{equation}
\label{eq:estimFnstep2}
    \tilde{F}_{n}(\lambda)\prod\limits_{m<n}\left|1+\frac{k_{c_{0}}\lambda}{\lambda_{m}-\lambda_{n}}\right|\geq 1.
\end{equation}
Since, using Lemma \ref{lem:boundT-1},
\begin{equation}
    \prod\limits_{m<n}\left|1+\frac{k_{c_{0}}\lambda}{\lambda_{m}-\lambda_{n}}\right|<Ce^{Ck_{c_{0}}^{\frac{1}{\alpha}}\lambda^{\frac{1}{\alpha}}},
\end{equation}
and with \eqref{eq:estimFnstep2} we deduce 
\eqref{ine-key-2}.

\textbf{Step 3: case $\lambda\in [|\lambda_{n}-\lambda_{1}|, +\infty)$.}

First, notice that if $\Dist_\alpha(\lambda)=0$, then the proof of \eqref{ine-key-2} is trivial. We therefore focus on the case $\Dist_\alpha(\lambda)\neq 0$. We underline that if $|\lambda_n- \lambda_1| = \lambda$, then $\lambda = \lambda_1- \lambda_n$ and $\Dist_\alpha(\lambda)=0$. Therefore, $\lambda>\lambda_1-\lambda_n>0$ and from the strict ordering of the eigenvalues of Assumption \ref{assume:structure}, 
\begin{equation*}
    \frac{\lambda}{\lambda_{m}- \lambda_{n}}\geq \frac{\lambda}{\lambda_{1}-\lambda_{n}}> 1.
\end{equation*}
Using again $\lambda>\lambda_1-\lambda_n>0$, this implies, using \eqref{eq:condestim2}
\begin{equation}
\label{eq:estimFnstep31}
\begin{split}
         \tilde F_n(\lambda)&= \prod_{m< n} \left( \frac{\lambda}{\lambda_{m}- \lambda_{n}}-1\right)> \prod_{m< n} \left( \frac{\lambda_{1}-\lambda_{n}}{\lambda_{m}- \lambda_{n}}-1\right)=\prod_{m< n} \left( \frac{\lambda_{1}-\lambda_{m}}{\lambda_{m}- \lambda_{n}}\right)\\
         &\geq 
         \prod_{m< n}\frac{1}{C} \left(\frac{\lambda_{1}-\lambda_{m}}{(n- m)n^{\alpha-1}}\right),
\end{split}
\end{equation}
where 
$C$ is given by \eqref{eq:condestim2}.
We are now going to give an estimate on $(\lambda_{1}-\lambda_{m})$.
{Denote $c:=\lambda_1/\lambda_2$ satisfying $0<c<1$. As in \eqref{eq:defkc}, there exists $k_c>0$ such that,
\[
(1-x)(1+k_cx)\geq 1, \text{ for all } x\in [0,c].
\]
As a consequence, defining $\tilde{k}:=-\lambda_1 k_c>0$, we have 
\[
 \left(1- \frac{\lambda_{1}}{\lambda_{m}}\right) \left(1- \frac{\tilde{k}}{\lambda_{m}}\right) \geq 1.
\]
}

Therefore, 
\begin{equation}
\label{eq:estimstep31}
    \prod_{m\in \mathbb{N}^{*}, m>1}\left(1- \frac{\lambda_{1}}{\lambda_{m}}\right) \left(1+ \frac{\tilde{k}}{|\lambda_{m}|}\right)\geq 1.
\end{equation}

Observe that, from Lemma \ref{lem-sum-alpha} and \eqref{eq:cond20}
\begin{equation}
    \prod_{m\in \mathbb{N}^{*}, m>1} \left(1+ \frac{\tilde{k}}{|\lambda_{m}|}\right)\leq C', 
\end{equation}
where $C'>0$ depends on $\lambda_1$ and $\lambda_2$ but not on $m>2$, $n>2$ or $\lambda$. Together with \eqref{eq:estimstep31}, this gives,
\begin{equation}
      \prod_{m\in \mathbb{N}^{*}, m>1} \frac{\lambda_{1}-\lambda_{m}}{|\lambda_{m}|}\geq \frac{1}{C'}.
\end{equation}
Therefore, using also \eqref{eq:condestim2}, \eqref{eq:estimFnstep31} and Stirling formula,
\begin{align*}
     \tilde F_n(\lambda)
    & \geq \prod_{m< n}\frac{1}{C} \left(\frac{|\lambda_m|(\lambda_{1}-\lambda_{m})}{|\lambda_m|(n- m)n^{\alpha-1}}\right) \\
     & \gtrsim \prod_{m< n} \frac{c}{C}\frac{m^{\alpha}}{(n- m) n^{\alpha-1}}\\
     &= \frac{\left((n-1)!\right)^{\alpha}}{(n-1)!} \left(\frac{c}{C}\right)^{n-1} \frac{1}{n^{(\alpha-1)(n-1)}}\\
     &\gtrsim e^{- cn} \gtrsim e^{-c \lambda^{\frac{1}{\alpha}}},
\end{align*}
which is \eqref{ine-key-2}.

\textbf{Step 4: case 
$\lambda\in(c_{0}(\lambda_{n-1}-\lambda_{n}),\lambda_{1}-\lambda_{n}) $.
}

Recall that $c_{0}$ is given in Step 2 and can be arbitrarily chosen as long as it is strictly smaller than 1 and does not depend on $n$ nor $m$. We assume again that $\Dist_{\alpha}(\lambda)>0$, otherwise there is nothing to show.

The first thing to observe is that we only need to focus on the pairs $(n,m)$ such that
\begin{equation*}
    m\in M_{1}:=\{m\in\mathbb{N}^{*} :m< n,\;\frac{\lambda}{\lambda_{m}-\lambda_{n}}> c_{0} \}.
\end{equation*}
Indeed, setting $M_{2}=\{1,...,n-1\}\setminus M_{1}$ one has, using \eqref{defkalpha} and Lemma \ref{lem:boundT-1},
\begin{equation}
\prod_{m\in M_{2}}\left(1+\frac{\lambda}{\lambda_{n}-\lambda_{m}}\right)\geq \left(\prod_{m\in M_{2}}\left(1+\frac{k_{c_{0}}\lambda}{\lambda_{m}-\lambda_{n}}\right)\right)^{-1}\gtrsim e^{-C\lambda^{\frac{1}{\alpha}}},
\end{equation}
where $C$ is a constant that only depends on $\alpha$ and $c_{0}$. Note that since $\lambda >c_{0}(\lambda_{n-1}-\lambda_{n})$, $(n-1)\in M_{1}$. In fact, there exists $\tilde{n}\in\{1,...,n-1\}$ such that  
$$M_{1} = \{n-\tilde n,...,n-1\}$$ 
and $\tilde{n}$ is given by (note that $n\rightarrow -\lambda_{n}$ is a strictly increasing function)

\begin{equation}
\label{eq:tilden}
    \tilde{n}:=\max\{\check{n}\in \{1,...,n-1\}\;|\;\lambda_{n-\check{n}}< \frac{\lambda}{c_{0}}+\lambda_{n}\}.
\end{equation}

Thus we deduce that it suffices to show that
\begin{equation}
    \prod_{m=n-\tilde n}^{n-1}\left|1+ \frac{\lambda}{\lambda_{n}-\lambda_{m}}\right|=\prod_{p=1}^{\tilde n}\left|1-\frac{\lambda}{\lambda_{n-p}-\lambda_{n}}\right|\geq e^{-C\lambda^{\frac{1}{\alpha}}} \Dist_{\alpha}(\lambda).
\end{equation}

We start with the case 
$\lambda\in(c_{0}(\lambda_{n-1}-\lambda_{n}),\lambda_{n-1}-\lambda_{n})$. Note that since $\Dist_{\alpha}(\lambda)>0$, we have in fact $\lambda\neq \lambda_{n-1}-\lambda_{n}$.

In this case, note that we always have, with $c$ and $C$ given by \eqref{eq:condestim2}
\begin{align*}
    \prod_{p=1}^{\tilde n}\left|1-\frac{\lambda}{\lambda_{n-p}-\lambda_{n}}\right| =& \left|1-\frac{\lambda}{\lambda_{n-1}-\lambda_{n}}\right|\prod_{p=2}^{\tilde{n}}\left|1-\frac{\lambda}{\lambda_{n-p}-\lambda_{n}}\right|\\
    \geq& \left|1-\frac{\lambda}{\lambda_{n-1}-\lambda_{n}}\right|\prod_{p=2}^{\tilde{n}}\left|1-\frac{\lambda_{n-1}-\lambda_{n}}{\lambda_{n-p}-\lambda_{n}}\right|\\
    =& \left|1-\frac{\lambda}{\lambda_{n-1}-\lambda_{n}}\right|\prod_{p=2}^{\tilde{n}}\left|\frac{\lambda_{n-p}-\lambda_{n-1}}{\lambda_{n-p}-\lambda_{n}}\right|\\
    \geq& \Dist_{\alpha}(\lambda)\frac{\prod_{p=2}^{\tilde{n}}|\lambda_{n-p}-\lambda_{n-1}|}{\prod_{p=1}^{\tilde{n}}|\lambda_{n-p}-\lambda_{n}|}\\
    \geq& \frac{\Dist_{\alpha}(\lambda)}{C}\left(\frac{c}{C}\right)^{\tilde{n}-2}\left(\frac{n-1}{n}\right)^{(\alpha-1)\tilde{n}}\frac{(\tilde{n}-1)!}{\tilde{n}!}\\
    =& \frac{\Dist_{\alpha}(\lambda)}{C}\left(\frac{c}{C}\right)^{\tilde{n}-2}\left(1-\frac{1}{n}\right)^{(\alpha-1)\tilde{n}}\frac{1}{\tilde{n}}\\
    \gtrsim & \Dist_{\alpha}(\lambda) e^{-c'\tilde{n}}
\end{align*}
where $c'$ does not depend on $\lambda$, $\tilde{n}$ or $n$. To conclude, it suffices to show that $\tilde{n}\leq \lambda^{\frac{1}{\alpha}}$. This is done later in \eqref{eq:estimtildenlamb}, after the case $\lambda\in(\lambda_{n-1}-\lambda_{n},\lambda_{1}-\lambda_{n})$.

Let us now look at the case $\lambda\in(\lambda_{n-1}-\lambda_{n},\lambda_{1}-\lambda_{n})$. 
We have,
\begin{equation*}
    \prod_{p=1}^{\tilde n}\left|1-\frac{\lambda}{\lambda_{n-p}-\lambda_{n}}\right|=\prod_{p=1}^{\tilde n}\frac{1}{\lambda_{n-p}-\lambda_{n}}\prod_{p=1}^{\tilde n}\left|\lambda_{n-p}-\lambda_{n}-\lambda\right|=:I_{1}\cdot I_{2},
\end{equation*}
Using \eqref{eq:condestim2}, we have
\begin{equation*}
    \frac{1}{\lambda_{n-p}-\lambda_{n}}\geq \frac{1}{C p n^{\alpha-1}},
\end{equation*}
where $C$ is given by \eqref{eq:condestim2}. Thus,
\begin{equation}
\label{eq:estimI1}
    I_{1}\geq \frac{1}{\tilde n !}\frac{1}{C^{\tilde n}n^{\tilde n(\alpha-1)}}.
\end{equation}
Let us now look at $I_{2}$. We introduce the strictly increasing function $f :x\rightarrow \lambda_{n-x}-\lambda_{n}$, where for $x\in\mathbb{R}\setminus\mathbb{N}$, $\lambda_{x}$ is defined as, 
\begin{equation}
    \lambda_{x} := \lambda_{\lfloor x\rfloor}+(x-\lfloor x\rfloor)(\lambda_{\lfloor x\rfloor+1}-\lambda_{\lfloor x\rfloor})
\end{equation}

Observe that, from the definition of $\tilde{n}$ given in \eqref{eq:tilden} and since $\lambda<\lambda_{1}-\lambda_{n}$, either $\tilde{n}< n-1$ and therefore $\lambda\leq c_{0}(\lambda_{n-(\tilde{n}+1)}-\lambda_{n})<\lambda_{n-(\tilde{n}+1)}-\lambda_{n}$ or $\tilde{n} = n-1$ and 
therefore $\lambda<\lambda_{n-\tilde{n}}-\lambda_{n}=\lambda_{1}-\lambda_{n}$. 
Similarly since, $\lambda >\lambda_{n-1}-\lambda_{n}$, we deduce that 
\begin{equation*}
    f(1)< \lambda< f(\min(\tilde n+1,n-1)).
\end{equation*}

Therefore there exists an integer $p_{0}$ such that $p_{0}\leq \min(\tilde n+1,n-1)-1$ and $f(p_{0})<\lambda$ and $f(p_{0}+1)\geq\lambda$. But since $\Dist_{\alpha}(\lambda)>0$, we have $f(p_{0}+1)\neq \lambda$. 
Therefore, as $f$ is strictly increasing,
\begin{equation*}
  \prod_{p=1}^{\tilde n}\left|\lambda_{n-p}-\lambda_{n}-\lambda\right|=  \prod_{p\leq p_{0}}(\lambda-(\lambda_{n-p}-\lambda_{n}))\prod_{\tilde n\geq p\geq p_{0}+1}((\lambda_{n-p}-\lambda_{n})-\lambda)=:I_{2}^{-}I_{2}^{+},
\end{equation*}
with the convention that if $p_{0}=\tilde{n}$ then the second product is empty and $I_{2}^{+}=1$.
Let us start with $I_{2}^{-}$. Using the definition of $f$ and the fact that $f(p)<\lambda, \, 1\leq p \leq p_0-1$,
\begin{equation}
\label{eq:I2-alpha}
\begin{split}
    I_{2}^{-} &= \left(\lambda - (\lambda_{n-p_{0}}-\lambda_{n})\right)\prod_{p\leq p_{0}-1}\lambda-(\lambda_{n-p}-\lambda_{n})\\
    &\geq (\lambda-f(p_{0}))\prod_{p\leq p_{0}-1} \lambda_{n-p_{0}}-\lambda_{n-p}\\
    &\geq c(\lambda-f(p_{0}))\prod_{p\leq p_{0}-1} (p_{0}-p)(n-p)^{\alpha-1}\\
    &= c(\lambda-f(p_{0}))(p_{0}-1)!\prod_{p\leq p_{0}-1} (n-p)^{\alpha-1},
\end{split}
\end{equation}
where $c$ is given by \eqref{eq:condestim2}. For $I_{2}^{+}$ we have, using a similar argument,
\begin{equation}
\label{eq:I2+alpha}
\begin{split}
    I_{2}^{+}&=((\lambda_{n-(p_{0}+1)}-\lambda_{n})-\lambda)\prod_{\tilde n\geq p\geq p_{0}+2}(\lambda_{n-p}-\lambda_{n})-\lambda\\
    &\geq (f(p_{0}+1)-\lambda)\prod_{\tilde n\geq p\geq p_{0}+2}(\lambda_{n-p}-\lambda_{n-(p_{0}+1)})\\
    &\geq c(\tilde n-p_{0}-1)!(f(p_{0}+1)-\lambda)\prod_{\tilde n\geq p\geq p_{0}+2}(n-p_{0}-1)^{\alpha-1} \\
    &= c(f(p_{0}+1)-\lambda) (\tilde n-p_{0}-1)! (n- p_0-1)^{(\alpha-1)(\tilde n- p_0-1)}
\end{split}
\end{equation}
where we used \eqref{eq:condestim2} in the last line.  
Combining \eqref{eq:I2-alpha} and \eqref{eq:I2+alpha} together, we have
\begin{equation}
\label{eq:estimI2}
    I_{2}\geq c^{2}(f(p_{0}+1)-\lambda)(\lambda-f(p_{0}))(\tilde n-p_{0}-1)!(p_{0}-1)!\prod_{\tilde n\geq p\geq p_{0}+2}(n-p_{0}-1)^{\alpha-1}\prod_{p\leq p_{0}-1} (n-p)^{\alpha-1}.
\end{equation}
Now observe that
\begin{equation}
\label{eq:estim-f-lambda}
    (f(p_{0}+1)-\lambda)(\lambda-f(p_{0}))\geq c\frac{\Dist_{\alpha}(\lambda)}{2}.
\end{equation}
Indeed, $f(p_{0}+1)-f(p_{0})=\lambda_{n-(p_{0}+1)}-\lambda_{n-p_{0}}\geq c(n-p_{0})^{\alpha-1}>c$, thus since $\lambda \in (f(p_{0}),f(p_{0}+1))$ we deduce that
either $(f(p_{0}+1)-\lambda)\geq c/2$ or $(\lambda-f(p_{0}))\geq c/2$, which leads to \eqref{eq:estim-f-lambda}. Hence, combining with \eqref{eq:estimI2}
\begin{equation}
\label{eq:estimI22}
\begin{split}
    I_{2}&\geq c^{3}\frac{\Dist_{\alpha}(\lambda)}{2}(\tilde n-p_{0}-1)!(p_{0}-1)!\prod_{\tilde n\geq p\geq p_{0}+2}(n-p_{0}-1)^{\alpha-1}\prod_{p\leq p_{0}-1} (n-p)^{\alpha-1}\\
    &=c^{3}\frac{\Dist_{\alpha}(\lambda)}{2}(\tilde n-p_{0}-1)!(p_{0}-1)!(n-p_{0}-1)^{(\tilde n-p_{0}-1)(\alpha-1)}\frac{((n-1)!)^{\alpha-1}}{((n-p_{0})!)^{\alpha-1}}.
    \end{split}
\end{equation}

We separate the analysis in two cases:
\begin{itemize}
\item Case $\tilde n\leq n/2$, in this case $p_{0}-1< n/2$ thus
\begin{equation*}
    \frac{(n-p_{0}-1)}{n} > \frac{1}{2},
\end{equation*}
hence
\begin{equation}
\label{eq:estimstep410}
(n-p_{0}-1)^{(\tilde n-p_{0}-1)(\alpha-1)}\frac{((n-1)!)^{\alpha-1}}{((n-p_{0})!)^{\alpha-1}}\geq \left(\frac{n}{2}\right)^{(\tilde n-2)(\alpha-1)},
\end{equation}
where we used the fact that the fraction of factorial is in fact a product of $p_{0}-1$ terms all of which are larger than $n/2$. We now recall the following useful lemma:
\begin{lemma}\label{lem-n!!}
If $n,m\in \N^*$ satisfy $0< n< m$, then 
\begin{equation}
    n! m!\geq (n+1)! (m-1)!.
\end{equation}
\end{lemma}
Setting $k = \lfloor \tilde n/2\rfloor-1$ and using this lemma ($\tilde{n}-p_{0}-1-k$ (resp. $k-(p_{0}-1)$) times) combined with \eqref{eq:estimstep410} gives 
\begin{equation}
    I_{2}\geq c^{3}\frac{\Dist_{\alpha}(\lambda)}{2}k!(\tilde n-1-k)!\left(\frac{n}{2}\right)^{(\tilde n-2) (\alpha-1)}.
\end{equation}
Then,  
assuming 
that $\tilde{n}$ is even (the odd case can be done exactly similarly) $$I_{2}\geq c^{3}\frac{\Dist_{\alpha}(\lambda)}{2}\left(\frac{2}{\tilde n}\right)^2 \cdot \left(\frac{\tilde n}{2}\right)!\left(\frac{\tilde n}{2}\right)!\left(\frac{n}{2}\right)^{(\tilde n-2)(\alpha-1)}, $$
and \eqref{eq:estimI1} together with Stirling formula,
\begin{equation}
\begin{split}
    I&\geq c^{3}\frac{\Dist_{\alpha}(\lambda)}{2}\left(\frac{2}{\tilde n}\right)^2 \cdot \left(\frac{\tilde n}{2}\right)!\left(\frac{\tilde n}{2}\right)!\left(\frac{n}{2}\right)^{(\tilde n-2) (\alpha-1)}\frac{1}{\tilde n !}\frac{1}{C^{\tilde n}n^{\tilde n(\alpha-1)}}\\
    &\gtrsim \Dist_{\alpha}(\lambda) \frac{1}{2^{(\tilde n-2)(\alpha-1)}(2C)^{\tilde n}n^{2(\alpha-1)}\tilde n^{2}}\\
    &\gtrsim \Dist_{\alpha}(\lambda) \frac{1}{2^{\alpha \tilde n}} \frac{1}{C^{\tilde n}} \frac{1}{n^{2\alpha}}\\
    &\gtrsim \Dist_{\alpha}(\lambda) e^{-c \tilde n},
\end{split}
\end{equation}
where $c$ is a constant depending on $\alpha$ (and C) but independent of $\lambda$. 
\item Case $\tilde n\in (n/2,n-1]$. In this case we have
\begin{equation}
\begin{split}
    (n-p_{0}-1)^{(\tilde n-p_{0}-1)(\alpha-1)}&\frac{((n-1)!)^{\alpha-1}}{((n-p_{0})!)^{\alpha-1}}\\
    &\geq (\tilde n-p_{0}-1)^{(\tilde n-p_{0}-1)(\alpha-1)}\frac{((\tilde n-1)!)^{\alpha-1}}{((\tilde n-p_{0})!)^{\alpha-1}}\\
    &\gtrsim ((\tilde n-1)!)^{\alpha-1},
\end{split}
\end{equation}
where we used that $n\rightarrow \frac{( n-1)!}{( n-p_{0})!}$ is increasing and $\tilde n\leq n-1$. Assuming again that $\tilde n$ is even (the odd case can be done exactly similarly) we have similarly as previously from \eqref{eq:estimI22} and \eqref{eq:estimI1}
\begin{equation}
\begin{split}
    I&\gtrsim \Dist_{\alpha}(\lambda) \left(\frac{2}{\tilde n}\right)^2 \cdot \left(\frac{\tilde n}{2}\right)!\left(\frac{\tilde n}{2}\right)!((\tilde n-1)!)^{\alpha-1}\frac{1}{\tilde n !}\frac{1}{C^{\tilde n}n^{\tilde n(\alpha-1)}}\\
    &\gtrsim \Dist_{\alpha}(\lambda) e^{-c\tilde n} 
\end{split}
\end{equation}
\end{itemize}
We are now going to show that whenever $\lambda\in(c_{0}(\lambda_{n-1}-\lambda_{n}),\lambda_{1}-\lambda_{n})$, $e^{-c \tilde n}\gtrsim e^{-c' \lambda^{\frac{1}{\alpha}}}$, where $c'$ is a constant independent from $\lambda$.
Since $\lambda<\lambda_{1}-\lambda_{n}<C(n-1)n^{\alpha-1}\leq C n^{\alpha}$ and $c_{0}(\lambda_{n-\tilde{n}}-\lambda_{n})<\lambda$, thus $\lambda>c c_{0}n^{\alpha-1}\tilde{n}$. As a consequence
\begin{equation}
\label{eq:estimtildenlamb}
\tilde{n}\leq \frac{\lambda}{c c_{0}n^{\alpha-1}}<\frac{C^{\frac{1}{\alpha}}\lambda}{c c_{0}\lambda^{(\alpha-1)/\alpha}}\lesssim \lambda^{\frac{1}{\alpha}}.
\end{equation}
 
Therefore
\begin{equation}
    I\gtrsim \Dist_{\alpha}(\lambda) e^{-c\lambda^{\frac{1}{\alpha}}}.
\end{equation}
This concludes the proof of Proposition \ref{cor-main2}.
\end{proof}

\subsection{Estimation of $\text{Dist}_{\alpha}(\lambda)$}
\label{sec:estimDist}

To conclude on an explicit lower bound on $F_{n}$ as a function of $\lambda$, we need an explicit estimate of $\Dist_{\alpha}(\lambda)$.  Indeed, we only need to consider the case that $A$ is self-adjoint, because for skew-adjoint operator $A$ one always has 
\begin{equation}
\Dist_{\alpha}(\lambda)\geq |\lambda|, \;\; \;\;   \textrm{ if $A$ is skew-adjoint}.    
\end{equation}
In the following we concentrate on the case that $A$ is self-adjoint.    What we would like to show is that there are sufficiently many $\lambda$ such that 
\begin{equation}
\label{eq:estimate-dist-lambda}
    \Dist_{\alpha}(\lambda)\gtrsim e^{-c \lambda^{\frac{1}{\alpha}}}, \;\;\;\;  \textrm{ if $A$ is self-adjoint}.
\end{equation}
This is more rigorously stated in the following proposition. 
\begin{proposition}
\label{prop:estimate-dist-mu-n}
The function $\Dist_{\alpha}(\lambda)$  satisfies the following estimates. 
\begin{itemize}
    \item  (Case $A$ is self-adjoint). There exist positive constants $c$, $C$ and a sequence $\{\mu_{N}\}_{N\in\mathbb{N}^{*}}$ such that for any $N\in \mathbb{N}^{*}$,
\begin{equation}
\begin{split}
&\mu_{N}\in [N,N+C],\\
        & \Dist_{\alpha}(\mu_{N})\gtrsim e^{-c \mu_{N}^{1/\alpha}}.
\end{split}
\end{equation}
\item   (Case $A$ is skew-adjoint). $\Dist_{\alpha}(\lambda)\geq |\lambda|$. 
\end{itemize}
\end{proposition}

\begin{remark}
This inequality can be improved to stronger estimates using the explicit representation of the spectrum. For example, for the heat equation, we are even able to replace the right-hand side of the preceding inequality by 1.
\end{remark}

Proposition \ref{prop:estimate-dist-mu-n} is a direct consequence of the following proposition.
\begin{proposition}\label{prop-close}
For any $N\in \N^*$, there exist some numbers 
\begin{equation}
    \mu_{N}\in \left\{N+ c\left(\frac{1+2i}{2M_N}\right): i=0, 1,..., M_N-1\right\} \textrm{ where } M_N:= \left\lfloor\left(\frac{N+c}{c}\right)^{\frac{1}{\alpha-1}}\right\rfloor+2,
\end{equation}
with $c$ given by \eqref{eq:condestim2},
such that 
\begin{equation}
    |\lambda_{m}- \lambda_{n}- \mu_{N}|\geq \frac{c}{2M_N}, \quad \forall n, m\in \N^*.
\end{equation}

\end{proposition}
\begin{proof}[of Proposition \ref{prop-close}]
 As illustrated above, we only consider the case that $A$ is self-adjoint.
We use a contradiction argument. Suppose that this proposition does not hold. Then for any 
$l_i=: N+ c\left(\frac{1+2i}{2M_N}\right)$, $i\in \{0, 1,..., M_N-1\}$,   there exists at least one pair of integers $(n_i, m_i)$ such that,
\begin{equation}
\label{eq:distnml}
     |\lambda_{m_{i}}- \lambda_{n_i}- l_i|< \frac{c}{2M_N}.
\end{equation}
Namely, 
\begin{equation}\label{choiceofni}
   \lambda_{m_i}- \lambda_{n_i}\in \left(N+ c\frac{2i}{2M_N}, N+ c\frac{2i+2}{2M_N}\right). 
\end{equation}

Now we present the following lemma,
\begin{lemma}\label{lem-n-small}
If for some $(n, m)\in \N^*$ there is,
\begin{equation*}
    \lambda_{m}-\lambda_{n}\in (N, N+c),
\end{equation*}
then, 
\begin{equation}
    n\leq M_N-2.
\end{equation}
\end{lemma}
\begin{proof}[of Lemma \ref{lem-n-small}]
By the definition of $(n, m)$ we know that,
\begin{equation*}
   N<  \lambda_{m}- \lambda_{n}< N+c,
\end{equation*}
thus from \eqref{eq:condestim2}
\begin{equation}
\label{eq:Lemma12proof}
   N+c> \lambda_{m}- \lambda_{n}\geq \lambda_{n-1}- \lambda_{n}> c n^{\alpha-1}.
\end{equation}
Therefore, there exists $0<r<1$ such that,
\begin{equation}
    n< \left(\frac{N+c}{c}\right)^{\frac{1}{\alpha-1}}=M_N -2 +r\leq M_N-1.
\end{equation}
\end{proof}

Now we come back to the proof of Proposition \ref{prop-close}. Thanks to equation \eqref{choiceofni}  and Lemma \ref{lem-n-small}, we know that 
\begin{equation}
    n_i\leq M_N-2, \quad \forall i\in \{0, 1,..., M_N-1\},
\end{equation}
which implies, since $n_{i}\in\mathbb{N}^{*}$, that there exist $i_1, i_2\in \{0, 1,..., M_N-1\}$ such that $i_{1}\neq i_{2}$ and $n_{i_1}= n_{i_2}$. Since, from \eqref{choiceofni},
\begin{gather*}
    N<  \lambda_{m_{i_{1}}}-\lambda_{n_{i_{1}}}< N+c, \\
     N<  \lambda_{m_{i_{2}}}-\lambda_{n_{i_{2}}}< N+c,
\end{gather*}
and $n_{i_{1}}=n_{i_{2}}$ we have, 
\begin{equation}
    \label{eq:m1m20}
    |\lambda_{m_{i_1}}- \lambda_{m_{i_2}}|<c.
\end{equation}
However, we know from \eqref{eq:condestim2} that if $m_{i_{1}}\neq m_{i_{2}}$ then, 
\begin{equation}
    |\lambda_{m_{i_{1}}}-\lambda_{m_{i_{2}}}|>c(m_{i_{1}}-m_{i_{2}})>c,
\end{equation}
which, combined with \eqref{eq:m1m20} means that $m_{i_{1}}=m_{i_{2}}$. Thus, using \eqref{eq:distnml},
\begin{equation}
|l_{i_{1}}-l_{i_{2}}|<\frac{c}{M_{N}},
\end{equation}
but from the definition of $l_{i}$, since $i_{1}\neq i_{2}$, $|l_{i_{1}}-l_{i_{2}}|=2c|i_{1}-i_{2}|/2M_{N}\geq c/M_{N}$ which leads to a contradiction, and concludes the proof of Proposition \ref{prop-close}.
\end{proof}

\begin{remark}
    Note that this proof relies strongly on the fact that $\alpha>1$; see in particular \eqref{eq:Lemma12proof}.
\end{remark}

The lower bound of Theorem \ref{th-knbn} then follows from the straightforward combination of Lemma \ref{lem-Jn-1}, and Propositions \ref{cor-main2} and \ref{prop:estimate-dist-mu-n}.

\section{Well-posedness and exponential stability of the closed-loop system}
\label{sec-well-posedness}

In this section, recall that, following Assumption \ref{assume:structure} and Section \ref{sec-func-framework}, we assume without loss of generality that $A$ is diagonalizable with simple eigenvalues. {We also assume that $\lambda\notin \mathcal{N}$.}

\subsection{Operator equality and domain of the closed-loop operator}
In keeping with the spirit of the spectral method (\cite{ho1986spectral, rebarber1989spectral, shun1981spectrum, xu1996spectrum}), a key ingredient in the study of the closed-loop operator $A+BK : D(A+BK)\subset H \rightarrow H$ is to determine its eigenfunctions.
This will allow us to describe more precisely its domain $D(A+BK)$, and rigorously establish the operator equality \eqref{eq:operator-identity}
\[T(A+BK)=(A-\lambda I) T\]
on sharp spaces. We  introduce the following family of non-zero functions,
\[\chi_n:=(A-(\lambda_n-\lambda)I)^{-1} B, \quad \forall n \in \N^\ast.\]
First note that these functions are well defined since $\lambda$ is chosen so that $\lambda_n-\lambda\neq \lambda_p$ for all $n,p\in \N^\ast$. Moreover, since $B\in \HH^{-\alpha}$ and $B\neq 0$, by property of the resolvent, we have $\chi_n \neq 0$ and
\begin{equation}\label{chi-n-in-H}\chi_n\in H, \quad \forall n\in \N^\ast.\end{equation}
Moreover, they decompose simply along the Riesz basis $\{\varphi_n\}_{n\in \NN^*}$
\[\chi_n=\sum_{p\in\N^\ast} \frac{b_p}{\lambda_p-\lambda_n+\lambda} \varphi_p.\]
In this section, we prove that the family $\{\chi_n\}_{n\in \NN^*}$ consists in eigenfunctions of $A+BK$. 

 \begin{lemma}\label{lmm-intrinsic-eigenvectors}
There holds the following, 
\begin{equation}
    \begin{aligned}
        &\chi_n \in D(A+BK), \\
        &(A+BK)\chi_n=(\lambda_n -\lambda )\chi_n,
    \end{aligned}
    \quad \forall n \in \N^\ast.
\end{equation}
 \end{lemma}
\begin{proof}
First, let us prove that the feedback $K$ defined above is well defined on the family $\{\chi_n\}_{\N^\ast}$. Recall from the identity \eqref{eq:TBBeq} that, for all $n\in \N^\ast$,
\begin{equation}
\label{K-on-chi-n}\begin{aligned}
    1&=\sum_{p\in \N^\ast} \frac{k_p b_p}{\lambda_n-\lambda_p-\lambda}
    =-\sum_{p\in \N^\ast}\frac{k_p b_p}{\lambda_p-\lambda_n+\lambda} \\
    &=-\sum_{p\in \N^\ast} k_p \langle \chi_n, \varphi_p\rangle
    =-\langle K, \chi_n \rangle.
\end{aligned}\end{equation}
Thus, $\langle K, \chi_n\rangle=-1, \ \forall n\in \N^\ast$. Since $A\chi_n ,B \in \HH^{-\alpha}=D(A)'$, we can compute in $D(A)^\prime$ (or $\HH^{-\alpha}$) for all $n\in \N^\ast$,
\[\begin{aligned}(A+BK)\chi_n&= A\chi_n -B \\
&=(A-(\lambda_n-\lambda)I)\chi_n+ (\lambda_n-\lambda)\chi_n-B \\
&=(A-(\lambda_n-\lambda)I)(A-(\lambda_n-\lambda)I)^{-1}B +(\lambda_n-\lambda)\chi_n-B \\
&=(\lambda_n-\lambda)\chi_n.\end{aligned}\]
Hence, 
\begin{equation}\label{eq:eigvectorABK}
(A+BK)\chi_n=(\lambda_n-\lambda)\chi_n, \textrm{ in } D(A)'.
\end{equation}
Now, from \eqref{chi-n-in-H},  \eqref{eq:eigvectorABK} actually holds in $H$ for all $n\in \N^*$. Hence,  $\chi_n \in D(A+BK), \ \forall n\in \N^*$.
\end{proof}

Now, given the operator equation \eqref{eq:operator-identity}, one would expect that the backstepping transformation $T$ should map the $\chi_n$ to the $\varphi_n$, given that it \textit{formally} conjugates $A-\lambda I$ to $A+BK$. We now confirm this with the following proposition.

\begin{proposition}\label{prop:expchin}
    The following hold,
    \[T\chi_n=\operatorname{sgn}(\langle T\chi_n, \varphi_n \rangle)\|T\chi_n\|\varphi_n, \quad \forall n \in \N^\ast, \]
\textit{i.e.,} $T\chi_n$ is an eigenfunction of $A$ for the eigenvalue $\lambda_n$. Moreover, 
the $\{T^{-1}\varphi_n\}_{n\in \NN^*}$ form a Riesz basis of $H$ of eigenfunctions of $A+BK$.
\end{proposition}
\begin{proof}Let us compute, for all $n, p \in \N^\ast$ with $n\neq p$, using \eqref{eq:Tnk:full}--\eqref{eq:trans:Tnk:full},
\[\begin{aligned}\langle (A-\lambda_n I)T\chi_n, \varphi_p\rangle_{D(A)',D(A)}&=(\lambda_p-\lambda_n)\langle T\chi_n, \varphi_p\rangle_H\\
&=(\lambda_p-\lambda_n) \left\langle \sum_{\ell\in \N^\ast} \frac{b_\ell}{\lambda_\ell-\lambda_n+\lambda} T\varphi_\ell, \varphi_p \right\rangle_H \\
&=(\lambda_p-\lambda_n)\sum_{\ell\in \N^\ast} \frac{b_\ell}{\lambda_\ell-\lambda_n+\lambda} \frac{k_\ell b_p}{\lambda_p-\lambda -\lambda_\ell}\\
&=(\lambda_p-\lambda_n)\sum_{\ell\in\N^\ast} \frac{b_p}{\lambda_\ell-\lambda_n+\lambda}\frac{k_\ell b_\ell}{\lambda_p-\lambda-\lambda_\ell}\\
&=\sum_{\ell\in \N^\ast}(\lambda_p-\lambda_\ell-\lambda+\lambda_\ell-\lambda_n+\lambda)\frac{b_p}{\lambda_\ell-\lambda_n+\lambda}\frac{k_\ell b_\ell}{\lambda_p-\lambda-\lambda_\ell}\\
&=b_p\sum_{\ell\in \N^\ast}\left(1+\frac{\lambda_p-\lambda_\ell-\lambda}{\lambda_\ell-\lambda_n+\lambda}\right)\frac{k_\ell b_\ell}{\lambda_p-\lambda-\lambda_\ell} \\
&=b_p\left[\sum_{\ell\in \N^\ast}\frac{k_\ell b_\ell}{\lambda_p-\lambda-\lambda_\ell}-\sum_{\ell\in \N^\ast}\frac{k_\ell b_\ell}{\lambda_n-\lambda-\lambda_\ell}\right]\\
&=0,
\end{aligned}\]
where the last line comes from identity \eqref{eq:TBBeq}. 

Hence, as $\chi_n\neq 0$, it is an eigenfunction of $A$ for the eigenvalue $\lambda_n$. Since this eigenvalue is simple, $T\chi_n$ is colinear to $\varphi_n$, and the first part of the proposition follows easily. Finally, since $T^{-1}\varphi_n=\textcolor{black}{\text{sgn}(\langle T\chi_{n},\varphi_{n} \rangle)}\chi_n/\|T\chi_n\|$, these vectors are eigenfunctions of $A+BK$ thanks to Lemma \ref{lmm-intrinsic-eigenvectors}. Since $T$ is an isomorphism from $H$ into itself, they form a Riesz basis of $H$, which concludes the proof of the proposition.
\end{proof}

We can now describe the domain of the closed-loop operator $A+BK$, through its Riesz basis of eigenvectors $(T^{-1}\varphi_n)$, and real eigenvalues $\lambda_n-\lambda$.  {\color{black} Note that this Riesz basis has a natural bi-orthonormal family, given by $\{T^\ast \varphi_n\}_{n\in \NN^*}$, so that we have the decomposition:
\[f=\sum_{n\in \N^\ast} \langle f, T^\ast \varphi_n \rangle T^{-1}\varphi_n, \quad \forall f\in H.\]
This leads to the following characterization:}
    \[D(A+BK)=\left\{f \in H, \quad \sum_{n\in \N^\ast}|\lambda_n-\lambda|^2\left|\langle f, T^\ast \varphi_n \rangle \right|^2 <+\infty\right\}.\]
     
It is clear that this is equivalent to 
      \[D(A+BK)=\left\{f \in H, \quad \sum_{n\in \N^\ast}|\lambda_n|^2\left|\langle f, T^\ast \varphi_n \rangle \right|^2 <+\infty\right\}.\]

    Now, given the assumptions on the $(\lambda_n)$, $A+BK$ clearly has compact resolvent. Since it admits a Riesz basis of eigenvectors, with simple eigenvalues, it is a Riesz spectral operator in the sense of Guo and Zwart \cite{guo2001riesz}. Fractional powers 
 are then defined via the spectral functional calculus for scalar-type spectral operators (see \cite[Chapter XIII]{DunfordSchwartzIII}).
  
We can then define another family of nested Banach spaces, associated to the closed-loop operator $A+BK$. 
      \[D_s(A+BK):=\left\{f \in H, \quad \sum_{n\in \N^\ast}|\lambda_n|^{2s}\left|\langle f, T^\ast \varphi_n \rangle \right|^2 <+\infty\right\}, \quad s\geq 0,\]
      and, recalling that $\lambda_n\neq0$, we define the norm
      \begin{equation}\label{norm-D_s(A+BK)}\|f\|_{D_s(A+BK)}:=\sum_{n\in \N^\ast} |\lambda_n|^{2s} \left|\langle f, T^\ast \varphi_n \rangle \right|^2 , \quad \forall f\in D_s(A+BK).\end{equation}
      Classically (see references above), the family of functions $(T^{-1}\varphi_n / \sqrt{|\lambda_n|^{2s}})_n$ is a Riesz basis for $D_s(A+BK)$ (with $s\geq 0$). 
      \begin{remark}
          We stress that $\|\cdot\|_{D_0(A+BK)}$ is equivalent to $\|\cdot\|$ (since the $\{T^{-1}\varphi_n\}$ form a Riesz basis of $H$), but not equal.
      \end{remark}
      
       On the other hand, for $s<0$, note that expression \eqref{norm-D_s(A+BK)} is well defined on all $H$. We then define $D_{s}(A+BK)$ as the completion of $H$ with respect to this norm, which we also denote by $\|\cdot\|_{D_s(A+BK)}$.
       
      Once again, for $s_1 > s_2$ the injection $D_{s_1}(A+BK) \to D_{s_2}(A+BK)$ is continuous and compact.  
      
     We can now establish the operator equality \eqref{eq:operator-identity}. Recall here that the notation $D(A^s)$ stands for $D((-A)^s)$ (see Remark \ref{rem:DAs}).
\begin{proposition}\label{prop-operator-equality}
There holds,
\begin{equation}\label{space-correspondences}D(A^s)=TD_s(A+BK), \quad \forall s\in \R,\end{equation}
and $T:(D_s(A+BK), \|\cdot\|_{D_s(A+BK)}) \to (D(A^s), \|\cdot\|_{D(A^s)})$ is an isometry.
Moreover, 
\begin{equation}\label{operator-equality-full-range}
T(A+BK)=(A-\lambda I) T \ \textrm{on}  \ D_s(A+BK), \quad \forall s\in \R.\end{equation}
\end{proposition}

\begin{proof}
We begin by proving that $T$ can be extended to the spaces $D_s(A+BK)$, $s<0$. 
    First note that for $s\geq 0$, one can write
      \[Tf=T\sum_n {\color{black}\langle f, T^\ast \varphi_n \rangle } T^{-1}\varphi_n= \sum_n {\color{black}\langle f, T^\ast \varphi_n \rangle } \varphi_n, \quad \forall f\in D_s(A+BK).\]
      Hence, by the definitions of these norms, for $s\geq 0$, 
      \begin{equation}\label{equal-norms-smooth}\|Tf\|_{D(A^s)}=\|f\|_{D_s(A+BK)}, \quad \forall f\in  D_s(A+BK). \end{equation}

      For $s<0$, we also have, from the same argument, 
     \begin{equation}\label{equal-norms-sing}\|Tf\|_{D(A^s)}=\|f\|_{D_s(A+BK)}, \quad \forall  f\in  H.\end{equation}
     Thus, since $D(A^s)$ (resp. $D_s(A+BK)$) is the completion of $H$ for the norm $\|\cdot\|_{D(A^s)}$ (resp. $\|\cdot\|_{D_s(A+BK)}$), and $T:(H, \|\cdot\|_{D_s(A+BK)}) \to (H, \|\cdot\|_{D(A^s)})$ (resp. $T^{-1}:(H, \|\cdot\|_{D(A^{s})})\to (H, \|\cdot\|_{D_s(A+BK)})$) is bounded, $T$ and $T^{-1}$ can be extended continuously to elements of $\mathcal{L}(D_s(A+BK), D(A^s))$ and $\mathcal{L}(D(A^s), D_s(A+BK))$ respectively. 
     
     It is straightforward that these extensions are still mutual inverses, and thus, from the definitions of $D_s(A+BK)$ and $D(A^s)$, we get: 
      \[f\in D_s(A+BK) \iff Tf \in D(A^s), \quad \forall s\in \textcolor{black}{\mathbb{R}},\]
      which proves \eqref{space-correspondences}.
      Moreover, from \eqref{equal-norms-smooth} and \eqref{equal-norms-sing} extended to $f\in D_s(A+BK)$, $T$ is an isometry  from $(D_{s}(A+BK),\|.\|_{D_s(A+BK)})$ to $(D(A^{s}),\|\cdot\|_{D(A^{s})})$

      Now, for $n\in \N^\ast$, 
      \[\begin{aligned}
          T(A+BK)T^{-1}\varphi_n&=T(\lambda_n-\lambda ){T^{-1}\varphi_n}\\
          &=(\lambda_n-\lambda){\varphi_n}\\
          &=(A-\lambda I){\varphi_n} \\
          &=(A-\lambda I)T(T^{-1}\varphi_n).
      \end{aligned}\]
      Hence the operator equality \eqref{eq:operator-identity} holds on Riesz basis of the spaces $D_s(A+BK)$ for $s\in \R$. Since $T$ is bounded on these spaces, \eqref{operator-equality-full-range} follows.
\end{proof}

\begin{remark}
It is important to stress the fact that we have established the operator equality in a wide array of spaces, but one should keep in mind that those spaces are the natural scale of Banach spaces associated with the closed-loop operator $A+BK$, and as such are specifically tailored for this purpose. {\color{black}In this sense}, we have {\color{black}described the comprehensive functional range for the operator equality.}
\end{remark}

Another important question is the relationship of the spaces $D_s(A+BK)$ to the more intrinsic regularity spaces $D(A^s)$. We give some insights on this matter, showing that for a far more restricted range of spaces, $D_s(A+BK)=D(A^s)$.

\begin{proposition}\label{prop-spaces-equality}
    For $s\in [0, 1-\tfrac{1}{2\alpha})$,
    \[D_s(A+BK)=D(A^s).\]
\end{proposition}
\begin{proof}
   On the one hand, we know from Lemma \ref{lem-T-check-reg} that $T$ is an isomorphism on $D(A^s)$ for $s\in [0, 1-\tfrac{1}{2\alpha})$. Combining this with Proposition \ref{prop-operator-equality}, we get
   \[D_s(A+BK)=T^{-1}(D(A^s))=D(A^s).\]
\end{proof}
\subsection{Well-posedness and stability of the closed-loop system}

 Now that we have properly established the operator equality \eqref{eq:operator-identity} in the adequate functional setting, we can use it to obtain the well-posedness and stability for the closed-loop system, in the same way as in finite dimension.

 To prepare for the quantitative work of the next section, which will rely on a diverging sequence of values for the damping parameter $\lambda$, we now denote by $(T_\lambda, K_\lambda)$ the backstepping transformation and the feedback obtained in the previous sections, for a given damping parameter $\lambda$. 

 We prove well-posedness properties for the closed-loop system, first in the natural scale $D_s(A+BK_\lambda)$, then in the less natural (but more useful for our purposes) scale $D(A^s).$

 We begin with the most natural well-posedness result:
\begin{theorem}
    \label{thm-well-posedness}
    {\color{black}Let $\lambda\notin\mathcal{N}$.} For every {\color{black}$s\in \R$},
     the operator $(A+BK_\lambda, D_{s+1}(A+BK_\lambda))$ is the generator of an exponentially stable semigroup on $D_s(A+BK_\lambda)$, and the following stability estimates hold,
    \begin{equation}
    \label{stability-estimates}
    \|e^{t(A+BK_\lambda)}y_0\|_{D_s(A+BK_\lambda)}\leq C_s e^{-\lambda t} \|y_0\|_{D_s(A+BK_\lambda)}, \quad \forall y_0\in D_s(A+BK_\lambda), \ \forall t\geq 0,
    \end{equation}
where $C_{s}$ is a positive constant that does not depend on $\lambda$.
\end{theorem}

\begin{proof}
    Let $\textcolor{black}{s\in \mathbb{R}}$. Thanks to Proposition \ref{prop-operator-equality}, the operator equality \eqref{eq:operator-identity} rewrites as
    \begin{equation}\label{order-s-norm-of-T}
    (A+BK_\lambda)= T_\lambda^{-1}(A-\lambda I)T_\lambda \ \textrm{on} \ T_\lambda^{-1}D(A^s)=D_s(A+BK_\lambda).
    \end{equation}
From the assumptions made on $A$, we know that $(A-\lambda I, D(A^{s+1}))$ generates an exponentially stable $C_0$-semigroup $S_s(t), t\geq 0$ on $D(A^s)$, with the following stability estimate,
\[\|S_s(t)y_0\|_{D(A^{s})} \leq C_s e^{-\lambda t} \|y_0\|_{D(A^{s})}, \quad \forall y_0\in D(A^s),
\]
\textcolor{black}{where $C_{s}$ is a positive constant that does not depend on $\lambda$.}
    Now, let us define the following semigroup on $D_s(A+BK_\lambda)$,
    \[e^{t(A+BK_\lambda)}y:=T_\lambda^{-1}S_s(t)T_\lambda y, \quad \forall y \in D_s(A+BK_\lambda).\]
    This semigroup is well defined since $T_\lambda D_s(A+BK_\lambda)=D(A^s)$. Let us check that its infinitesimal generator is $(A+BK_\lambda, D_{s+1}(A+BK_\lambda))$. Let $y\in D_s(A+BK_\lambda)$, and compute
    \[\begin{aligned}
        \frac{e^{t(A+BK_\lambda)}y-y}t&=\frac{T_\lambda ^{-1}S_s(t)T_\lambda y-y}t \\
        &=T_\lambda ^{-1}\frac{S_s(t)T_\lambda y-T_\lambda y}t .
    \end{aligned}\]
   By continuity of $T_\lambda ^{-1}$, the above quantity converges in $D_s(A+BK_\lambda)$ when $t\to 0$ if and only if $\frac{S_s(t)T_\lambda y-T_\lambda y}{t}$ converges in $D(A^s)$ when $t\to 0$. By definition of $S_s(t)$ and its infinitesimal generator $(A-\lambda I, D(A^{s+1}))$, 
    \[\frac{S_s(t)T_\lambda y-T_\lambda y}t \xrightarrow[t\to 0]{D(A^s)} (A-\lambda I)T_\lambda y \ \iff \ T_\lambda y\in D(A^{s+1}).\]
    Now, the above is equivalent to $y\in D_ {s+1}(A+BK_\lambda)$, thanks to \eqref{space-correspondences}. Thus, 
    \[\frac{e^{t(A+BK_\lambda)}y-y}{t} \xrightarrow[t\to 0]{D_s(A+BK_\lambda)} T_\lambda ^{-1}(A-\lambda I)T_\lambda y \ \iff \ y\in D_{s+1}(A+BK_\lambda).\]
    Finally, we conclude using the operator equality \textcolor{black}{\eqref{order-s-norm-of-T}:} 
    \[\frac{e^{t(A+BK_\lambda)}y-y}{t} \xrightarrow[t\to 0]{D_s(A+BK_\lambda)} (A+BK_\lambda)y \ \iff \ y\in D_{s+1}(A+BK_\lambda),\]
hence, $(A+BK_\lambda, D_{s+1}(A+BK_\lambda))$ is the infinitesimal generator of $e^{t(A+BK_\lambda)}$ on $D_s(A+BK_\lambda)$.

Now consider $y\in D_s(A+BK_\lambda)$, we then have, using \eqref{equal-norms-smooth}, \eqref{equal-norms-sing} and the definition of the semigroup $e^{t(A+BK_\lambda)}$,
\[\begin{aligned}\|e^{t(A+BK_\lambda)} y\|_{D_s(A+BK_\lambda)}&=\|T_\lambda ^{-1}S_s(t)T_\lambda y\|_{D_s(A+BK_\lambda)}\\
&=\|S_s(t)T_\lambda y\|_{D(A^s)} \\
&\leq C_s e^{-\lambda t} \|T_\lambda y\|_{D(A^s)} \\
&=C_s e^{-\lambda t} \|y\|_{D_s(A+BK_\lambda)},\end{aligned}\]
which ends the proof.
\end{proof}

We complete this stability result with the admissibility of $K_\lambda$ (see also Lemma \ref{lem:Kadmissible}) in the functional framework associated to $A+BK_\lambda$.

\begin{proposition}\label{prop-K-adm-A+BK}
    Let $s>\tfrac{1}{2\alpha}$. For $\lambda\notin \mathcal{N}$, $K_\lambda\in\left(D_{s}(A+BK_\lambda)\right)^\prime$. Thus, it is well-defined on trajectories of the semigroup $e^{t(A+BK_\lambda)}$ on the space $D_{s}(A+BK_\lambda)$. Moreover, for any $N\in \N$ there exists $\lambda\notin \mathcal{N}$ such that $\lambda\in [N, N+1]$, and
     \begin{equation}\label{feedback-a-priori-estimate}
     \begin{aligned}
     |K_\lambda f| &\leq C_s e^{c\lambda^{\frac{1}{\alpha}}} \|f\|_{D_{s}(A+BK_\lambda)}, \\
     |K_\lambda e^{t(A+BK_\lambda)}f|&\leq C^{\prime}_s e^{c\lambda^{\frac{1}{\alpha}}}  e^{-\lambda t } \|f\|_{D_{s}(A+BK_\lambda)}, \quad \forall f\in D_{s}(A+BK_\lambda), \ \forall t\geq 0.\end{aligned}\end{equation}
     where $C_s,C^\prime_s>0$ depend only on $A$, $B$, and $s$.
\end{proposition}
\begin{proof}
    First, from Proposition \ref{prop:expchin} and Proposition \ref{prop-operator-equality}, we have \begin{equation}
    T_\lambda ^{-1}\varphi_n=\operatorname{sgn}(\langle T_\lambda \chi_n, \varphi_n \rangle)\frac{\chi_n}{\|T_\lambda \chi_n \|}, \text{ in } D(A+BK).
    \end{equation}
    Then, from \eqref{K-on-chi-n}, we have \[K_\lambda T_\lambda^{-1}\varphi_n =-\frac{\operatorname{sgn}(\langle T_\lambda \chi_n, \varphi_n \rangle)}{\|T_\lambda \chi_n \;\|}\textcolor{black}{\in \mathbb{R}}.\]
    Moreover, since $T_\lambda^{-1}\varphi_n$ is an eigenvector of $A+BK_\lambda $, we can compute, in $D(A)'$,
    \[AT_\lambda^{-1}\varphi_n + (K_\lambda T_\lambda^{-1}\varphi_n)B= (\lambda_n-\lambda) T_\lambda^{-1} \varphi_n, \]
    which leads to \textcolor{black}{(recall that $K_{\lambda}T^{-1}_{\lambda}\varphi_{n}$ is a scalar)}
    \[T_\lambda^{-1}\varphi_n= -(K_\lambda T_\lambda^{-1}\varphi_n) (A-(\lambda_n-\lambda)I)^{-1}B \in H.\]
    Now, since $\{T_\lambda^{-1}\varphi_n\}$ is a Riesz basis of $H$ \textcolor{black}{and $(\varphi_{n})_{n\in\mathbb{N}^{*}}$ is an orthonormal family of eigenvectors of $A$}, there exists $C>0$ such that
    \begin{equation}\label{naive-estimate-forced-modes}\|T_\lambda^{-1}\|\geq \|T_\lambda^{-1}\varphi_n\|_H=\sqrt{\sum_{k\in \N}|\langle T^{-1}_\lambda \varphi_n, \varphi_k \rangle |^2 }\geq |\langle T_\lambda^{-1}\varphi_n, \varphi_n \rangle_H | = |K_\lambda T_\lambda^{-1}\varphi_n| \frac{|b_n|}{|\lambda|},\end{equation}
    \textit{i.e.},
    \begin{equation}\label{naive-estimate-forced-modes-2}|K_\lambda T_\lambda^{-1}\varphi_n| \leq \frac{|\lambda|}{|b_n|}\|T_\lambda^{-1}\| \leq C^\prime \lambda\|T_\lambda^{-1}\| ,\end{equation}
    where $C^\prime>0$ is independent of $n$, \textcolor{black}{from the controllability-type condition \eqref{b_n-growth}}

    Now, let $s>\tfrac{1}{2\alpha}$, and let $f=\sum_n f_n T_\lambda^{-1}\varphi_n \in D_{1/2}(A+BK_\lambda )$ where here $f_n:=\langle f,T^*\varphi_n\rangle$. From the above we derive,
    \[\begin{aligned}
         |K_\lambda f|&=\left|\sum_n f_n (K_\lambda T_\lambda^{-1}\varphi_n)\right| \\
         &= \left|\sum_n f_n |\lambda_n|^{s} \frac{(K_\lambda T_\lambda^{-1}\varphi_n)}{|\lambda_n|^{s}}\right| \\
         &\leq \sqrt{\sum_n |\lambda_n|^{2s}|f_n|^2} \sqrt{ \sum_n \frac{|K_\lambda T_\lambda^{-1}\varphi_n|^2}{|\lambda_n|^{2s}}}.
    \end{aligned}
   \]
From \eqref{eq:cond20} \textcolor{black}{(which is a consequence of Assumption \ref{assume:structure})}, and the assumption on $s$, we know that
   \[\sum_n \frac{1}{|\lambda_n|^{2s}} < +\infty.\]
   Thus, by definition of $D_{s}(A+BK_\lambda )$, and using \eqref{equal-norms-smooth} and \eqref{naive-estimate-forced-modes-2}, there exists $C_s>0$ such that
   \begin{equation}\label{qualitative-continuity-K}|K_\lambda f| \leq C_s \lambda \|T_\lambda^{-1}\| \,\|f\|_{D_{s}(A+BK_\lambda )},\end{equation}
   which implies that $K_\lambda \in \left(D_{s}(A+BK_\lambda )\right)^\prime$.
   
   To get the quantitative estimates, let $N\in \N$, $\lambda\notin \mathcal{N}$ and $\lambda\in [N, N+1]$ such that Theorem \ref{cor:boundTT-1} is satisfied. Then, \eqref{naive-estimate-forced-modes} \textcolor{black}{implies}
   \[Ce^{c\lambda^{\frac{1}{\alpha}}} \geq \|T_\lambda^{-1}\varphi_n\|\geq |K_\lambda T_\lambda^{-1}\varphi_n|\frac{|b_n|}{\lambda},\]
   which yields
   \[|K_\lambda T_\lambda^{-1}\varphi_n|\leq C^{\prime\prime}e^{c\lambda^{\frac{1}{\alpha}}}\lambda.\]
   Then, \eqref{qualitative-continuity-K} becomes
   \[|K_\lambda f| \leq C^{\prime\prime\prime}_se^{c\lambda^{\frac{1}{\alpha}}} \lambda \|f\|_{D_{s}(A+BK_\lambda )}.\]
   Since $\alpha>0$, there exists a constant $C_\alpha>0$ such that
   \[\lambda\leq C_\alpha e^{c\lambda^{\frac1\alpha}},\]
   which yields the first inequality of \eqref{feedback-a-priori-estimate}. Combining this with the fact that $e^{t(A+BK_\lambda )}$ is a $C^0$ semigroup on $D_{\textcolor{black}{s}}(A+BK_\lambda )$, and using Theorem \ref{thm-well-posedness}, we easily obtain \eqref{feedback-a-priori-estimate}.
  
\end{proof}

As remarked in the previous section, the above theorem is given in the natural scale of spaces $D_s(A+BK_\lambda)$. However, they depend on $\lambda$, and the Riesz basis we have provided for them are implicit. In many cases, in particular to achieve finite time stabilization or small-time null controllability, one would like to obtain quantitative estimates in $\lambda$. This is more easily achieved in the explicit (\textcolor{black}{and more practical} but less natural \textcolor{black}{from an operator point of view}) scale of spaces $D(A^s)$, since they do not depend on $\lambda$. We start with a stability estimate for the linear closed-loop system:
\begin{theorem}
\label{th:wellposed-closed-As}
        Let $s\in [0, 1-\tfrac{1}{2\alpha})$. For $N\in \N$ there exists $\lambda\notin \mathcal{N}$ such that $\lambda\in [N, N+1]$ and the closed-loop operator $A+BK_\lambda$ generates an exponentially stable semigroup on $D(A^{s})$, such that
    \begin{equation}
        \label{stability-estimate-D(A^s)}
        \|e^{t(A+BK)} y\|_{D(A^{s})} \leq C_s e^{{\color{black}c}\lambda^{\frac1\alpha}} e^{-\lambda t} \|y\|_{D(A^{s})},
    \end{equation}
where the constant $c$ does not depend on $(s, N, \lambda)$, and $C_s$ does not depend on $(N, \lambda)$.
\end{theorem}

\begin{proof}
    Let $s\in [0, 1-\tfrac{1}{2\alpha})$. As in the proof of Theorem \ref{thm-well-posedness}, we define the semigroup
    \[e^{t(A+BK)}:=T^{-1} S_{s}(t)T.\]
    Since $T$ is an isomorphism on $D(A^{s})$, and since $S_{s}(t)$ is an exponentially stable semigroup on $D(A^{s})$, $e^{t(A+BK)}$ is well-defined and we have the following estimate \textcolor{black}{thanks to Lemma~\ref{lem-T-check-reg}}:
    \[\begin{aligned}
        \|e^{t(A+BK)}y \|_{D(A^s)}&=\|T^{-1} S_{s}(t) T y\|_{D(A^s)} \\
        &\leq \|T^{-1}\|_{D(A^s)} \, e^{-\lambda t} \, \|Ty\|_{D(A^s)}\\
        &\leq \|T^{-1}\|_{D(A^s)} \|T\|_{D(A^s)} \, e^{-\lambda t} \|y\|_{D(A^s)}\\
        &\leq C_se^{{\color{black}s}\lambda^{\frac1\alpha}} \, e^{-\lambda t} \|y\|_{D(A^s)}.
    \end{aligned}\]
\end{proof}

Again, we complete the above result with an admissibility result for $K_\lambda$.
\begin{proposition}\label{prop-K-A-bounded}
   Let $s\in (\tfrac{1}{2\alpha}, 1-\tfrac{1}{2\alpha}).$ For $\lambda \notin \mathcal{N}$, $K_\lambda \in D(A^{s})^\prime$. Thus, it is well-defined on trajectories of the semigroup $e^{t(A+BK_\lambda)}$ in the space $D(A^s)$. Moreover, for \textcolor{black}{any} $N\in \N$, there exists $\lambda\notin \mathcal{N}$, such that $\lambda\in [N, N+1]$, and
    \begin{equation}
        \label{feedback-norm-in-A}
        \begin{aligned}
            &|K_\lambda f|\leq C_s e^{c\lambda ^{\frac{1}{\alpha}}} \|f\|_{D(A^{s})}, \\
            &|K_\lambda e^{t(A+BK_\lambda)}f|\leq C^\prime_s e^{{\color{black}s}\lambda^{\frac1\alpha}} e^{-\lambda t}\|f\|_{D(A^s)}, \quad \forall y\in D(A^{s}),
        \end{aligned}
    \end{equation}
       where the constant $c$ does not depend on $(s, N, \lambda)$, and $C_s, C^\prime_s$ do not depend on $(N, \lambda)$.
\end{proposition}
\begin{proof}
Let $s\in (\tfrac{1}{2\alpha}, 1-\tfrac{1}{2\alpha})$. First, recall that from Proposition \ref{prop-spaces-equality} $D(A^s)=D_s(A+BK)$. Moreover, for $f\in D(A^s)$, we have 
\[\|f\|_{D_{s}(A+BK_{\lambda})} = \|T f\|_{D(A^{s})} \leq \|T\|_{D(A^s)}\|f\|_{D(A^{s})}.\] 
Combining this with Proposition \ref{prop-K-adm-A+BK} and Lemma \ref{lem-T-check-reg}, we immediately get \eqref{feedback-norm-in-A}.
    
\end{proof}

\section{Small-time null-controllability}
\label{sec-null-control}

In this last section, we make use of the sharp estimations on $\|T\|_{D(A^s)}$ and $\|T^{-1}\|_{D(A^s)}$ with respect to $\lambda$ \textcolor{black}{obtained in Theorem \ref{cor:boundTT-1} and \ref{lem-T-check-reg}} to deduce small-time null-controllability.

Let $T>0$. We begin by building an explicit closed-loop control which achieves null-controllability in time $T$,  using estimates from Theorem \ref{th-knbn}.  {\color{black}Fix some $\gamma> \frac{\alpha}{\alpha-1}$.}  From Proposition \ref{prop:estimate-dist-mu-n}, for \textcolor{black}{any} $N\in \N^\ast$ there exists $\lambda(N) \notin \mathcal{N}$ such that the bounds {\color{black} \eqref{T-Tinv-norm-estimate} concerning $T_{\lambda}$ and $T_{\lambda}^{-1}$ holds} and,
\begin{equation}\label{lambda(N)-growth}\lambda(N)\in [N^{\gamma}, N^\gamma+C],\end{equation}
for some $C>0$ independent of $N$.
Now, let $\sigma \in (\alpha/(\alpha-1), \gamma) $. Then, 
\[ \lambda(N)^{-\frac1\sigma}{\color{black}\sim} N^{-\frac{\gamma}{\sigma}}, \quad \forall N \in \N^\ast,\]
and since $\frac\gamma\sigma>1$,  the following series convergence:
\[\sum_{N=1}^\infty \lambda(N)^{-\frac1\sigma} =: L_\sigma <\infty.\]
We then define 
\begin{equation}
\label{eq:defTN}
\begin{array}{c}\delta_N:= \frac{T}{L_\sigma} \lambda(N)^{-\frac1\sigma}, \quad \forall N\in \N^\ast \\
t_0:=0, \quad t_N:=\sum_{p=1}^N \delta_p, \quad \forall N\geq 1.\end{array}
\end{equation}
 In particular, by construction $t_N\xrightarrow[N\to\infty]{} T$.
 
We now define the following piecewise constant feedback control,
\begin{equation}
    \label{simple-piecewise-feedback}
    K(t):= K_{\lambda(N)} \ \textrm{if} \  t \in [t_{N-1}, t_{N}), \quad \forall t\in [0,T), \; \forall N\in \N^\ast,
\end{equation}
which allows us to state the following closed-loop null-controllability result.

\begin{proposition}\label{prop-null-control}
\textcolor{black}{For any $s\in [0,1-\tfrac{1}{2\alpha})$, t}he closed-loop control system
    \begin{equation}
    \label{null-control-system}
    \begin{cases} \dfrac{d}{dt} y(t)=(A+BK(t)) y(t), \quad t>0  \\\ y(0)=y_0, 
    \end{cases}
\end{equation}
is well-posed \textcolor{black}{in $D(A^{s})$} and
satisfies $y(T)=0$ for any initial condition $y_0\in \textcolor{black}{D(A^{s})}$.

Moreover,  if
$\textcolor{black}{s > \tfrac{1}{2\alpha}}$, the control $K(t)y(t)$ satisfies 
\begin{equation}
\label{eq:Kyto0}
|K(t)y(t)|\xrightarrow[t\to T^-]{}0.
\end{equation}
\end{proposition}

\begin{proof}
The well-posedness of the above system is straightforward, since it can be established successively on each interval $[t_N, t_{N+1}]$ thanks to the definition \eqref{simple-piecewise-feedback} of the control term, and the results of Section \ref{sec-well-posedness} \textcolor{black}{(see Theorem \ref{th:wellposed-closed-As})}.

To prove the null-controllability result, \textcolor{black}{that is $y(T)=0$,} let $s\in [0, 1-\tfrac{1}{2\alpha})$ and note that estimate \eqref{stability-estimate-D(A^s)} implies that exist constants {\color{black}$c$}, $C_s >0${\color{black}, independent of $\lambda(N)>0$,} such that, for \textcolor{black}{any} $N\in \N^\ast$, \textcolor{black}{and any $t\in[t_{N},t_{N+1})$,}
\[\|e^{t(A+BK_{\lambda(N)})}f\|_{D(A^{s})}\leq C_s e^{{\color{black}c}\lambda(N)^{\frac1\alpha} - \lambda(N)t}\|f\|_{D(A^{s})}.\]
Now, for \textcolor{black}{any} $N \in \N$,  given a solution $y(t)$ of \eqref{null-control-system},
\[\|y(t_{N})\|_{D(A^{s})}\leq C_s e^{{\color{black}c}\lambda(N)^{ \frac1\alpha} - \lambda(N)\delta_N} \|y(t_{N-1})\|_{D(A^{s})}, \quad \forall N\in \N^\ast. \]

Iterating this leads to
\begin{equation}\label{iterated-estimate-finite-time}\begin{aligned}
\|y(t)\|_{D(A^{s})}&\leq C^N_se^{\sum_{k=1}^{N}{\color{black}c}\lambda(k)^{ \frac1\alpha} - \lambda(k) \delta_{k}  } \, e^{{\color{black}c}\lambda(N+1)^{\frac1\alpha}-\lambda(N+1)(t-t_N)}\|y(0)\|_{D(A^{s})} \\
&\leq C_s^N e^{\sum_{k=1}^{N} {\color{black}c}\lambda(k)^{ \frac1\alpha} - \lambda(k) \delta_{k}  } \, e^{{\color{black}c}\lambda(N+1)^{\frac1\alpha}} \|y(0)\|_{D(A^{s})}\\
&= C_s^N e^{ \sum_{k=1}^{N} {\color{black}c}\lambda(k+1)^{ \frac1\alpha} - \lambda(k) \delta_{k}  } \, e^{{\color{black}c} \lambda(1)^{\frac1\alpha}} \|y(0)\|_{D(A^{s})} \\
&=e^{ {\color{black}c_s'} N +\sum_{k=1}^{N} {\color{black}c}\lambda(k+1)^{ \frac1\alpha} - \lambda(k) \delta_{k}  } \, e^{{\color{black}c}\lambda(1)^{\frac1\alpha}} \|y(0)\|_{D(A^{s})} \\
&\leq e^{{\color{black}c} \lambda(1)^{\frac1\alpha}}   e^{\sum_{k=1}^{N} \left(c_s^{\prime}+ {\color{black}c}\lambda(k+1)^{ \frac1\alpha} - \lambda(k) \delta_{k}\right)  } \,  \|y(0)\|_{D(A^{s})}, \quad \forall t \in [t_N, t_{N+1}) ,
\end{aligned}\end{equation}
where $c^{\prime}_s>0$ is independent of $N$.
Now, from \eqref{lambda(N)-growth} and the definition of $\delta_k$, we have
\[\begin{aligned} {\color{black}c} \lambda(k+1)^{\frac1\alpha}-\lambda(k)\delta_k &= {\color{black}c} \lambda(k+1)^{\frac1\alpha}-\frac{T}{L_\sigma}\lambda(k)^{1-\frac1\sigma} \\
&= \lambda(k+1)^{1-\frac1\sigma}\left({\color{black}c} \lambda(k+1)^{\frac{1-\alpha}\alpha +\frac1\sigma }-\left(\frac{\lambda(k)}{\lambda(k+1)} \right)^{1-\frac1\sigma}\frac{T}{L_\sigma}\right)\end{aligned}\]
and since $\sigma>\frac\alpha{\alpha-1}$, we have $\frac{1-\alpha}\alpha+\frac1\sigma<0$, so that 
\[\lambda(k+1)^{\frac{1-\alpha}\alpha +\frac1\sigma } \xrightarrow[k\to \infty]{}0.\]
Moreover, given the growth of $\lambda(k)$, it is straightforward that,
\[\frac{\lambda(k)}{\lambda(k+1)} \xrightarrow[k\to\infty]{}1.\]

Thus, for $k\geq 1$ large enough,
\begin{equation}\label{asympt-estimate-null-control-state}
    \begin{aligned}
        c^{\prime}_s+ {\color{black}c}\lambda(k+1)^{\frac1\alpha}-\lambda(k)\delta_k &= c_s^{\prime} -\frac{T}{L_\sigma}\lambda(k+1)^{1-\frac1\sigma}+o\left(\lambda(k+1)^{1-\frac1\sigma}\right) \\
        &= -\frac{T}{L_\sigma}\lambda(k)^{1-\frac1\sigma} +o\left(\lambda(k)^{1-\frac1\sigma}\right) .
    \end{aligned}
\end{equation}
Hence, since $\lambda(k)^{1-\frac1\sigma} \xrightarrow[k\to\infty]{}+\infty$,
\begin{equation}\label{series-divergence}\sum_{k=1}^{N} c_s^{\prime}+ \lambda(k+1)^{ \frac1\alpha} - \lambda(k) \delta_{k} \xrightarrow[N\to\infty]{} -\infty.\end{equation}
Combining \eqref{iterated-estimate-finite-time}, \eqref{series-divergence} \textcolor{black}{and the fact that $t_{N} \xrightarrow[N\to\infty]{}T$} then yields
\[y(t)\xrightarrow[t\to T^-]{} 0.\]
Now, let $s\geq \alpha/2$ and $y_0\in D(A^s)$. The control $K(t)y(t)$ is defined for all $t\geq 0$ thanks to Proposition \ref{prop-K-adm-A+BK} and Proposition \ref{prop-spaces-equality}. To estimate the control term, note that, using first \eqref{feedback-norm-in-A} , then \eqref{iterated-estimate-finite-time}, we have
\[\begin{aligned}
    |K(t) y(t) | &= |K_{\lambda(N+1)} y(t)| \\
    & =|K_{\lambda(N+1)} e^{(t-t_N)(A+BK_{\lambda(N+1)})}y(t_N)| \\
& \leq C_{s} \lambda(N+1)e^{-\lambda(N+1) (t-t_N)} e^{C\lambda(N+1)^{\frac{1}{\alpha}}}\|y(t_N)\|_{D(A^{s})} \\
&\leq C_{s} \lambda(N+1)e^{-\lambda(N+1) (t-t_N)} e^{C\lambda(N+1)^{\frac{1}{\alpha}}} e^{c \lambda(1)^{\frac1\alpha}}   e^{\sum_{k=1}^{N} \left(c^{\prime}+c\lambda(k+1)^{ \frac1\alpha} - \lambda(k) \delta_{k}\right)  } \,  \|y(0)\|_{D(A^{s})} \\
&\leq C_{s} \lambda(N+1)e^{c \lambda(1)^{\frac1\alpha}}   e^{\sum_{k=1}^{N} \left(c^{\prime}+c\lambda(k+1)^{ \frac1\alpha} - \lambda(k) \delta_{k}\right)  } \,  \|y(0)\|_{D(A^{s})} \\
&\leq C_{s}e^{c \lambda(1)^{\frac1\alpha}}   e^{\log(\lambda(N+1))+\sum_{k=1}^{N} \left(c^{\prime}+c\lambda(k+1)^{ \frac1\alpha} - \lambda(k) \delta_{k}\right)  } \,  \|y(0)\|_{D(A^{s})} \\
&\leq C_{s}e^{c \lambda(1)^{\frac1\alpha}}   e^{\sum_{k=1}^{N} \left(\log(\lambda(k+1))+c^{\prime}+c\lambda(k+1)^{ \frac1\alpha} - \lambda(k) \delta_{k}\right)  } \,  \|y(0)\| _{D(A^{s})}, \quad \forall t\in [t_N, t_{N+1}).
\end{aligned}
\]
Clearly, given the growth of $\lambda(N)$ (see \eqref{lambda(N)-growth}), the $\log$ term is negligible, and we have, as in \eqref{asympt-estimate-null-control-state},
\begin{equation} \label{asympt-estimate-null-control-cont-term}
\log(\lambda(k+1))+c^{\prime}+c\lambda(k+1)^{ \frac1\alpha} - \lambda(k) \delta_{k}=-\frac{T}{L_\sigma}\lambda(k)^{1-\frac1\sigma} +o\left(\lambda(k)^{1-\frac1\sigma}\right) .\end{equation}
Combining this with the inequality above then yields
\[|K(t) y(t) | \xrightarrow[t\to T^-]{} 0.\]
\end{proof}

\begin{remark}
    In the above proof, it would have been sufficient to establish the estimate \eqref{iterated-estimate-finite-time} for the $\|\cdot\|$ norm, since it would have implied the null-controllability result in all $D(A^s)$ spaces. However, one can in fact establish this estimate in the range of spaces $D(A^s)$, $s\in [0, 1-\tfrac{1}{2\alpha})$. Such estimates would typically play an essential role to obtain a semi-global stabilization result, but this is out of the scope of the current paper.
\end{remark}

\appendix
\section{Technical lemmas on the spectrum of $A$}
\label{appendix-estimates-lambda}

In this appendix we provide the proofs of some technical lemmas, including Lemma \ref{lem:di}, Lemma \ref{lem:boundT-1}, and Lemma \ref{lem:sumlambdaijdistbound}.

\begin{proof}[of Lemma \ref{lem:di}]
We easily see that 
\[
D_i(\lambda)=\min_{j\in \N^*} | \lambda_j-\lambda_i+\lambda|, 
\] 
is well-defined since $\inf_{j\in \N^*} | \lambda_j-\lambda_i+\lambda|$ is reached for a finite $j\in \N^*$.  This is clear when $A$ is skew-adjoint as $D_{i}(\lambda) = \lambda$. Let us now assume that $A$ is self-adjoint. Since $\lambda_j\rightarrow -\infty$ as $j\rightarrow \infty$ and $j\in \N^* \mapsto | \lambda_j-\lambda_i+\lambda|$ is increasing with respect to $j$ for $\lambda_j>\lambda_i-\lambda$. 

Moreover, notice that $j\in \N^* \mapsto | \lambda_j-\lambda_i+\lambda|$ may be decreasing only if there exists $j_1 \in \N^*$ such that $\lambda_j<\lambda_i-\lambda$, $\forall 1\leq j \leq j_1$. In this case, using the assumption $\lambda_j+\lambda\neq \lambda_i, \forall i,j\in \N^*$
\[
-\lambda_1+\lambda_i-\lambda < \ldots < -\lambda_{j_1}+\lambda_i-\lambda < 0 < -  \lambda_{j_2}+ \lambda_i-\lambda  < \ldots,
\]
with $j_2=j_1+1$. Hence, 
\[
D_i(\lambda)=\min\{ | \lambda_{j_1}-\lambda_i+\lambda|, | \lambda_{j_2}-\lambda_i+\lambda| \}.
\]
Finally, we have $\Dist_\alpha(\lambda)\leq D_i(\lambda)$ by definition of $\Dist_\alpha(\lambda)$
\end{proof}

\begin{proof}[of Lemma \ref{lem:boundT-1}]
First, notice that due to Assumption \ref{assume:structure} concerning the asymptotic behavior of the eigenvalues, for $i\in \N$,
\[
1\pm \dfrac{\lambda}{\lambda_m-\lambda_i} \rightarrow 1, \quad m \rightarrow \infty, \, m\neq i.
\]
 We first prove assertion (i). Assertion (ii) is dealt with similarly. Using the gap condition Lemma \ref{lem:thirdgap},
\begin{align*}
\left| \prod_{\substack{m\in \N^\ast\\ m \neq i}} \left( 1+ \dfrac{\lambda}{\lambda_i-\lambda_m}  \right) \right|  =  \prod_{\substack{m\in \N^\ast\\ m \neq i}} \left| 1+ \dfrac{\lambda}{\lambda_i-\lambda_m}  \right|   & \leq   \prod_{\substack{m\in \N^\ast\\ m \neq i}} \left( 1+\left| \dfrac{\lambda}{\lambda_i-\lambda_m}  \right| \right)  \\
& \leq   \prod_{\substack{m\in \N^\ast\\ m \neq i}} \left( 1+C\dfrac{\lambda}{|i- m|^\alpha} \right) \\
& = \exp\left( \sum_{\substack{m\in \N^\ast\\ m \neq i}}\ln \left( 1+C\dfrac{\lambda}{|i-m|^\alpha} \right)  \right) \\
& \leq \exp\left( 2 \sum_{k\in \N^*}\ln \left( 1+C\dfrac{\lambda}{k^\alpha} \right)  \right).
\end{align*}
Since $\ln(1+1/y)$ is decreasing over $y\in (1,\infty)$, 
\begin{align*}
\sum_{k\in \N^*}\ln \left( 1+C\dfrac{\lambda}{k^\alpha} \right) & \leq  \ln(1+C\lambda)  + \int_1^\infty \ln \left( 1+C\dfrac{\lambda}{x^\alpha} \right) dx  \\
& =\ln(1+C\lambda) + x\ln\left(1+C\dfrac{\lambda}{x^\alpha}\right) \Big|_1^\infty + \alpha \int_{1}^\infty \dfrac{C\lambda}{x^\alpha + C\lambda} dx \\
& =\alpha (C\lambda)^{\frac{1}{\alpha}} \int_{(C\lambda)^{-\frac{1}{\alpha}}}^\infty \dfrac{1}{y^\alpha + 1 } dy \\
& \leq \alpha (C\lambda)^{\frac{1}{\alpha}} \int_0^\infty \dfrac{1}{y^\alpha + 1 } dy.
\end{align*}
Hence, there exists $C>0$ such that, for every $\lambda\in (0,\infty)$ and $i\in \N^*$,
\[
\left| \prod_{\substack{m\in \N^\ast\\ m \neq i}} \left( 1+ \dfrac{\lambda}{\lambda_i-\lambda_m}  \right) \right|\leq C e^{C \lambda^{\frac{1}{\alpha}}}.
\]
\end{proof}

\begin{proof}[of Lemma \ref{lem:sumlambdaijdistbound}] 
In the sequel we only present the proof for the inequality $(ii)$, since the inequality $(i)$ can be treated similarly.
 
Let us consider the case where $A$ is self-adjoint. Let $j\in \mathbb{N}^*$.
From Lemma \ref{lem:di}, we may assume that there exists $J_j \in \N^*$ which satisfies $|\lambda_{J_j}-\lambda_j+\lambda|=D_{j}(\lambda)$. Without loss of generality, we may assume $\lambda_j-\lambda_{J_j}-\lambda>0$, as a similar argument works for the case $\lambda_{J_j}-\lambda_j+\lambda<0$. Under this assumption, we also know that $\lambda_j-\lambda_{J_j- 1}-\lambda\leq - D_j(\lambda)<0$.

Then, we write, 
\[
\sum_{i\in \N^*} \left|\dfrac{\lambda^2}{\lambda_j-\lambda_i-\lambda}\right| \leq \dfrac{2\lambda^2}{D_j(\lambda)} + \sum_{1\leq i \leq J_j-2 } \dfrac{\lambda^2}{\lambda_i+\lambda -\lambda_j} + \sum_{i \geq J_j+1 } \dfrac{\lambda^2}{ \lambda_j-\lambda_i-\lambda},
\]
where we bounded the term $i=J_j-1$ using $D_j(\lambda)\leq |\lambda_{J_j-1}-\lambda_j+\lambda|$. Then, since $\lambda_{J_j}-\lambda_i>0$ for $i>J_j$, using the gap condition of Lemma \ref{lem:thirdgap},
\begin{align*}
\sum_{J_j< i } \dfrac{\lambda^2}{ \lambda_j-\lambda_i-\lambda} & = \sum_{J_j< i } \dfrac{\lambda^2}{ \lambda_{J_j}-\lambda_i+D_j(\lambda)} \leq  \sum_{J_j< i } \dfrac{\lambda^2}{ \lambda_{J_j}-\lambda_i} \\
& \leq  \sum_{J_j< i } \dfrac{\lambda^2}{ |i- J_j|^{\alpha}} 
 \leq C \sum_{k\in \N^*}  \dfrac{\lambda^2}{ k^\alpha}  \leq C \lambda^2
\end{align*}

Concerning the sum over $i\in \{1, 2,..., J_j- 2\}$, we know that 
\begin{align*}
\sum_{1\leq i\leq J_j- 2 } \dfrac{\lambda^2}{ \lambda_i+\lambda- \lambda_j} 
&=\sum_{1\leq i\leq J_j- 2 } \dfrac{\lambda^2}{ (\lambda_{J_j-1}+\lambda- \lambda_j)+ (\lambda_i- \lambda_{J_j-1})}\\
& \leq  
\sum_{1\leq i\leq J_j- 2 } \dfrac{\lambda^2}{ (\lambda_i- \lambda_{J_j-1})+D_j(\lambda)} \\
& \leq  
\sum_{1\leq i\leq J_j- 2 } \dfrac{\lambda^2}{ (\lambda_i- \lambda_{J_j-1})} 
 \leq C \sum_{1\leq i\leq J_j- 2 } \dfrac{\lambda^2}{ |J_{j-1}-i|^{\alpha}}\\ 
& \leq C \sum_{k\in \N^*}  \dfrac{\lambda^2}{ k^\alpha} 
 \leq C \lambda^2
\end{align*}
Collecting the estimates above yield the proof. For the case where $A$ is skew-adjoint the situation is simpler as $J_{i}=i$ and is unique and $D_{i}(\lambda) = \lambda$, thus
\begin{equation}
    \sum\limits_{i\in\mathbb{N}^{*}} \left|\frac{\lambda^{2}}{\lambda_{j}-\lambda_{i}-\lambda}\right| =     \sum\limits_{i\in\mathbb{N}^{*}} \frac{\lambda^{2}}{(|\lambda_{j}-\lambda_{i}|^{2}+\lambda^{2})^{1/2}} \leq \lambda+ \sum\limits_{i\neq j} \frac{\lambda^{2}}{|\lambda_{j}-\lambda_{i}|},
\end{equation}
and the rest of the proof is identical.

\end{proof}

\section*{Acknowledgements}
AH was supported by the ANR-Tremplin StarPDE ANR-24-ERCS-0010. SX was supported by NSFC 12301562 and 12571474. LG and CZ are partially supported by ANR QuBiCCS ANR-24-CE40-3008.

\bibliographystyle{plain}
\bibliography{biblio}

\end{document}